\documentclass [11pt]{amsart}
\usepackage{hyperref}
\usepackage {amsmath, amsthm, amssymb, amscd, mathrsfs, graphicx, extarrows, color, url, enumerate, epsf, tikz-cd}
\usepackage{tikz}
\usetikzlibrary{positioning}
\usepackage{subfig}

\usepackage[all,knot,arc]{xy}
\usepackage[text={6in,8.5in},centering,letterpaper,dvips]{geometry}

\newtheorem {theorem}{Theorem}[section]
\newtheorem {lemma}[theorem]{Lemma}
\newtheorem {proposition}[theorem]{Proposition}
\newtheorem {corollary}[theorem]{Corollary}

\newtheorem {definition}[theorem]{Definition}
\theoremstyle{remark}
\newtheorem {remark}[theorem]{Remark}
\newtheorem {example}[theorem]{Example}

\newcommand\Z{\mathbb{Z}}
\newcommand\R{\mathbb{R}}
\newcommand\Q{\mathbb{Q}}

\def\red{\operatorname{red}}
\def\H{\mathcal{H}}

\newcommand \bH {\overline{\H}}

\newcommand\Ta{\mathbb{T}_\alpha}
\newcommand\Tb{\mathbb{T}_\beta}

\def\Sym{\mathrm{Sym}}

\def\del {{\partial}}
\def\tH{\tilde{H}}
\def\tC{\tilde{C}}

\def\spinc {{\operatorname{spin^c}}}

\def\fin\qedhere
\def\pr {{\text{pr}}}
\DeclareMathOperator{\id}{id}

\DeclareMathOperator{\Hom}{Hom}

\def\from {{\leftarrow}}

\def\s{\mathfrak s}
\def\t{\mathbf t}
\def\x{\mathbf x}
\def\y{\mathbf y}

\def\Ring {\mathcal R}

\def\S{\mathbf S}

\newcommand\alphas{\boldsymbol\alpha}
\newcommand\betas{\boldsymbol\beta}
\newcommand\gammas{\boldsymbol\gamma}

\def\ws{\mathbf w}

\def\AI {\mathit{AI}}

\def\gr{\mathrm{gr}}
\def\F{\mathbb Z}

\def\ker {{\operatorname{Ker}}}
\def\im {{\operatorname{Im}}}

\def\Id  {{\operatorname{Id}}}

\def\fin\qedhere
\def\from {{\leftarrow}}

\def\ccdot {\! \cdot \!}
\def\pin{\operatorname{Pin}(2)}

\def\Cone{\mathit{Cone}}

\def\Ainf {\mathcal{A}_{\infty}}

\def\Tor{{\operatorname{Tor}}}
\def\Span{{\operatorname{Span}}}

\newcommand{\bunderline}[1]{\underline{#1\mkern-2mu}\mkern2mu }

\def\du {\bar{d}}
\def\dl {\bunderline{d}}

\def\CF {\mathit{CF}}
\def\HF {\mathit{HF}}

\newcommand\CFhat{\widehat{\CF}}

\newcommand \CFp {\CF^+}
\newcommand \CFm {\CF^-}
\newcommand \HFm {\HF^-}
\newcommand \CFinf {\CF^{\infty}}

\newcommand \CFo {\CF^{\circ}}
\newcommand \HFo {\HF^{\circ}}

\def\CFI {\mathit{CFI}}
\def\HFI {\mathit{HFI}}
\newcommand\HFIhat{\widehat{\HFI}}

\newcommand\HFIp {\HFI^+}
\newcommand \CFIp {\CFI^+}
\newcommand \CFIm {\CFI^-}
\newcommand \HFIm {\HFI^-}

\newcommand \HFIinf {\HFI^{\infty}}
\newcommand \CFIo {\CFI^{\circ}}
\newcommand \HFIo {\HFI^{\circ}}

\def\CFKm{\mathit{CFK}^-}
\def\CFKi{\mathit{CFK}^{\infty}}

\def\CFKinfty{\CFKi}
\def\inv{\iota}
\def\delinv{\partial^{\inv}}
\def\Inv{\mathfrak{I}}
\def\RRing{\mathfrak{R}}

\newcommand{\co}{\nobreak\mskip2mu\mathpunct{}\nonscript
  \mkern-\thinmuskip{:}\penalty300\mskip6muplus1mu\relax}

\def\swf{\mathit{SWF}}

\def\wt{\widetilde}

\begin{document}

\title{A connected sum formula for involutive Heegaard Floer homology}

\author[Kristen Hendricks]{Kristen Hendricks}
\author[Ciprian Manolescu]{Ciprian Manolescu}
\author[Ian Zemke]{Ian Zemke}
\thanks {KH was partially supported by NSF grant DMS-1506358. CM and IZ were partially supported by NSF grant DMS-1402914.}

\address {Department of Mathematics, Michigan State University\\ 
East Lansing, MI 48824}
\email {hendricks@math.msu.edu}

\address {Department of Mathematics, University of California\\ 
Los Angeles, CA 90095}
\email {cm@math.ucla.edu}
\email {ianzemke@math.ucla.edu}

\begin{abstract} We prove a connected sum formula for involutive Heegaard Floer homology, and use it to study the involutive correction terms of connected sums. In particular, we give an example of a three-manifold with $\dl(Y) \neq d(Y) \neq \du(Y)$. We also construct a homomorphism from the three-dimensional homology cobordism group to an algebraically defined Abelian group, consisting of certain complexes (equipped with a homotopy involution) modulo a notion of local equivalence. \end{abstract}

\maketitle

\section {Introduction}
Heegaard Floer homology is an invariant of three-manifolds introduced by Ozsv\'ath and Szab\'o in  \cite{HolDisk, HolDiskTwo}. By studying the absolute grading of Heegaard Floer homology, one can associate a ``correction term'' $d$ to any rational homology sphere equipped with a $\spinc$ structure \cite{AbsGraded}.  The correction term is the analogue of the Fr{\o}yshov invariants from \cite{FroyshovYM, FroyshovSW}, and gives constraints on the intersection forms of definite four-manifolds with boundary. Furthermore, the invariant $d$ descends to a surjective homomorphism from the three-dimensional homology cobordism group $\Theta^3_{\Z}$ to the even integers, and thus can be used to show that $\Theta^3_\Z$ has a $\Z$ summand. We recall that the group $\Theta^3_\Z$ plays a fundamental role in the study of triangulations on high dimensional manifolds; see  \cite{GS, Matumoto, Triangulations}.

In \cite{HMinvolutive}, the first and second authors defined a new invariant, called involutive Heegaard Floer homology. Its construction involves a certain homotopy involution on Heegaard Floer complexes, called the conjugation involution and denoted $\iota$. Involutive Heegaard Floer homology is associated to a three-manifold equipped with a spin structure. In the case of rational homology spheres, one can extract two ``involutive correction terms,'' $\dl$ and $\du$, which give more constraints on the intersection forms of definite four-manifolds. By restricting to  integer homology spheres, we obtain maps $\dl, \du \co \Theta^3_\Z \to 2\Z$. It is shown in \cite{HMinvolutive} that $\dl$ and $\du$ contain information beyond that from Heegaard Floer homology; for example, for the Brieskorn sphere $\Sigma(2,3,7)$ we have $\du=d=0$ but $\dl=-2$, so $\dl$ detects that $\Sigma(2,3,7)$ is not null-homology cobordant (a fact that can also be easily seen using the Rokhlin invariant).

The purpose of this paper is to study the behavior of involutive Heegaard Floer homology under taking connected sums. Let us start by recalling the connected sum formula for the usual Heegaard Floer groups, proved by Ozsv\'ath and Szab\'o in \cite[Section 6]{HolDiskTwo}. We will work with mod $2$ coefficients, and focus on the minus version of Heegaard Floer homology, $\HFm$, which is the homology of a free complex $\CFm$ over the polynomial ring $\Z_2[U]$. The connected sum formula says that, for three-manifolds $Y_1$ and $Y_2$ equipped with $\spinc$ structures $\s_1$ and $\s_2$, we have chain homotopy equivalences 
\begin{equation}
\label{eq:KunnethCF}
\CFm(Y_1 \# Y_2, \s_1 \# \s_2) \simeq \CFm(Y_1,\s_1) \otimes_{\Z_2[U]} \CFm(Y_2, \s_2)[-2],
\end{equation}
where $[-2]$ denotes a grading shift.\footnote{This shift was not present in \cite{HolDiskTwo}, where only relatively graded modules were considered. However, the shift is needed when we consider absolute gradings as in \cite{AbsGraded}, i.e., when $\s_1$ and $\s_2$ are torsion. Recall that, for example, $\HF^-(S^3)$ is generated by an element in degree $-2$.} Similar results hold for two other versions of the Heegaard Floer complexes, $\CFhat$ and $\CFinf$, but with the tensoring done over the rings $\Z_2$, resp. $\Z_2[U, U^{-1}]$, and with no grading shift.

Our main theorem is the following.

\begin{theorem} 
\label{thm:ConnSum}
Suppose $Y_1$ and $Y_2$ are three-manifolds equipped with spin structures $\s_1$ and $\s_2$. Let $\iota_1$ and $\iota_2$ denote the conjugation involutions on the Floer complexes $\CFm(Y_1, \s_1)$ and $\CFm(Y_2,\s_2)$. Then, under the equivalence \eqref{eq:KunnethCF}, the conjugation involution $\iota$ on $\CFm(Y_1 \# Y_2,$ $ \s_1 \# \s_2)$ is chain homotopy equivalent, over the ring $\Z_2[U]$, to $\iota_1 \otimes \iota_2.$ 

The analogous statements are also true for the theories $\CFhat$ and $\CFinf$.
\end{theorem}

The proof of Theorem~\ref{thm:ConnSum} involves analyzing maps induced by graph cobordisms, using the theory developed by the third author in \cite{Zemke2}. Interestingly, what the theory yields at first is that the involution $\iota$ corresponds to 
\begin{equation}
\label{eq:withPhi}
\iota_1 \otimes \iota_2 + U(\Phi_1 \iota_1 \otimes \Phi_2 \iota_2),
\end{equation}
where $\Phi_i$ ($i=1,2$) is the formal derivative of the differential on $\CFm(Y_i, \s_i)$, with respect to the $U$ variable. However, a further algebraic argument shows that, because the complexes $\CFm(Y_i, \s_i)$ are $\Z$-graded, the extra term $U(\Phi_1 \iota_1 \otimes \Phi_2 \iota_2)$ is chain homotopic to zero.

The involutive Heegaard Floer homology $\HFIm(Y,\s)$ is defined in \cite{HMinvolutive} as the homology of the mapping cone complex
$$
\CFIm(Y,\s) := \CF(Y,\s) \xrightarrow{\phantom{o} Q (1+\inv) \phantom{o}} Q \ccdot \CF(Y,\s) [-1],
$$
over the ring $\Ring := \Z_2[Q, U]/(Q^2)$. Theorem~\ref{thm:ConnSum} allows one to compute 
$$ \CFIm(Y_1 \# Y_2, \s_1\#\s_2)$$
from knowledge of the complexes $\CFm(Y_1, \s_1)$ and $\CFm(Y_2,\s_2)$, together with their conjugation involutions. Precisely, the $\mathcal{R}$-equivariant chain homotopy type of the complex $\CFIm(Y_1\# Y_2, \s_1\# \s_2)$ is the cone of the map
\[\CFm(Y_1)\otimes \CFm(Y_2) \xrightarrow{\phantom{o} Q (1+\iota_1\otimes \iota_2) \phantom{o}}Q\cdot (\CFm(Y_1)\otimes \CFm(Y_2)).\] 

\begin{remark}
We expect that there is also a connected sum formula of a different flavor, similar to the one established by Lin in \cite{LinConnSums} for $\pin$ monopole Floer homology. One should be able to turn the ring $\Ring$ into an $\Ainf$-algebra, by introducing higher compositions, and to equip $\HFIm(Y, \s)$ with the structure of an $\Ainf$-module over that $\Ainf$-algebra. Then, $\HFIm(Y_1 \# Y_2, \s_1\#\s_2)$ should be the $\Ainf$ tensor product of $\HFIm(Y_1, \s_1)$ and $\HFIm(Y_2,\s_2)$. However, proving such a formula is beyond the scope of this paper. See Section~\ref{sec:SW} for further remarks.
\end{remark}

Recall that, when applied to integer homology spheres, the correction terms $d$, $\dl$ and $\du$ produce maps $\Theta^3_{\Z} \to 2\Z$. In fact, we can also consider $\Z_2$-homology spheres, and obtain maps $d, \dl, \du \co \Theta^3_{\Z_2} \to \Q$, where $\Theta^3_{\Z_2}$ is the $\Z_2$-homology cobordism group. It is proved in \cite{AbsGraded} that $d$ is a homomorphism. On the other hand, by considering the connected sum $\Sigma(2,3,7) \# \Sigma(2,3,7)$ and its reverse, it is shown in \cite[Section 6.9]{HMinvolutive} that $\dl$ and $\du$ are not homomorphisms. Thus, it remains interesting to study the behavior of the involutive correction terms under connected sum. From Theorem~\ref{thm:ConnSum} we obtain the following inequalities, which are similar to those proved by Stoffregen \cite{Stoffregen2} and Lin \cite{LinConnSums} for the invariants $\alpha$, $\beta$, $\gamma$ coming from $\pin$-equivariant Seiberg--Witten Floer homology.

\begin{proposition} \label{prop:inequalities} Let $(Y_1,\s_1)$ and $(Y_2,\s_2)$ be rational homology spheres together with spin structures. Let $\s$ be the spin structure $\s_1 \# \s_2$ on $Y_1 \# Y_2$. Then
\[
\dl(Y_1,\s_1) + \dl(Y_2,\s_2) \leq \dl(Y_1 \# Y_2, \s) \leq \dl(Y_1, \s_1) + \du(Y_2, \s_2) \leq \du(Y_1\#Y_2, \s)\leq \du(Y_1, \s_1) + \du(Y_2, \s_2). 
\]
\end{proposition}

The involutive correction terms are calculated in \cite{HMinvolutive} for several classes of manifolds, including large integral surgeries on $L$-space knots, mirrors of $L$-space knots, and Floer homologically thin knots, e.g., alternating knots. (Here, a surgery is called large if its coefficient is greater or equal than the Seifert genus of the knot.) The methods of the current paper allow us to compute the involutive correction terms for connected sums of such manifolds. For example:

\begin{proposition}\label{prop:237sums}
For the connected sums of $n$ copies of the Brieskorn sphere $\Sigma(2,3,7)$, we have
\[ \dl(\#n \Sigma(2,3,7))=-2, \quad \quad d(\#n\Sigma(2,3,7))=\du(\#n\Sigma(2,3,7)) =0.\]
\end{proposition}

Interestingly, for all the manifolds for which $\dl$ and $\du$ were computed in \cite{HMinvolutive}, we had either $\dl=d$ or $\du=d$. By considering connected sums, we can now find a three-manifold with $\dl \neq d \neq \du$.

\begin{proposition}\label{prop:alldifferent} The $\Z_2$-homology sphere $Y = S^3_{-3}(T_{2,7}) \# S^3_{-3}(T_{2,7}) \# S^3_{5}(-T_{2,11})$, equipped with its unique spin structure, has
\[
\dl(Y) = -2, \quad \quad d(Y) = 0, \quad \quad \du(Y) = 2.
\]
\end{proposition}

From here, we get a topological application.

\begin{corollary} The manifold $Y$ from Proposition~\ref{prop:alldifferent} is not $\mathbb Z_2$-homology cobordant to any large surgery on an $L$-space knot, mirror of an $L$-space knot, or Floer homologically thin knot.\end{corollary}

Further, we see that the involutive correction terms of a connected sum are not determined by those  of the summands. The example below shows that $\dl$ is not determined this way; by changing orientations, we can also get an example with $\du$ instead of $\dl$.

\begin{proposition} 
\label{prop:undetermined}
For the manifolds $Y_1 =  S^3_{-3}(T_{2,7}) \# S^3_{-3}(T_{2,7})$, $Y_2 = S^3_{-3}(T_{2,7}) \# (-L(3,1))$ and $Z= S^3_{5}(-T_{2,11})$, we have
$$ \dl(Y_1) = \dl(Y_2)=-5, \ d(Y_1)=d(Y_2)=-1, \ \du(Y_1)=\du(Y_2)=-1,$$
but
$$ \dl(Y_1 \# Z)=-2 \neq 0=\dl(Y_2 \# Z).$$
\end{proposition}

In \cite[Section 4.3]{Stoffregen}, Stoffregen used $\pin$-equivariant Seiberg-Witten Floer chain complexes to define a homomorphism from $\Theta^3_{\Z}$ to an Abelian group, denoted $\mathfrak{CLE}$. The elements of $\mathfrak{CLE}$ are chain complexes of a certain form, modulo a relation called local equivalence. We can imitate Stoffregen's construction in our setting, and define an Abelian group $\Inv$ by considering pairs $(C, \iota)$, where $C$ is a free chain complex over $\Z_2[U]$ whose homology has a single infinite ``$U$-tail,'' and $\iota$ is an endomorphism of $C$ with $\iota^2 \sim \id$. (The model for $C$ is $\CFm(Y)[-2]$, with $Y$ a homology sphere, and with $\inv$ being the conjugation involution.) The elements of $\Inv$ are taken to be pairs $(C, \iota)$ as above, modulo a local equivalence relation that models the existence of maps induced by homology cobordism. 

\begin{theorem}
\label{thm:Inv}
For homology spheres $Y$, the assignment $Y \mapsto \bigl( \CFm(Y)[-2], \iota \bigr )$ produces a homomorphism
$$h: \Theta^3_{\Z} \to \Inv.$$
\end{theorem}

The construction of $\Inv$ is reminiscent to that of the group of knot Floer complexes modulo $\varepsilon$-equivalence, defined by Hom in \cite{Hom1}. That group was denoted $\mathcal{CFK}$ in \cite{Hom2} and was used there to prove the existence of a $\Z^{\infty}$ summand in the smooth concordance group of topologically slice knots, $\mathcal{C}_{TS}$. The argument used a total ordering on $\mathcal{CFK}$, and the fact that the knot Floer complex defines a homomorphism $\mathcal{C}_{TS} \to \mathcal{CFK}$.

In our context, the group $\Inv$ has a purely algebraic definition, but its structure is yet to be understood. There is a homomorphism $d_{\Inv} \co \Inv \to \Z$, but it remains an open problem to construct further such homomorphisms, or a total ordering on $\Inv$. In turn, this may help in uncovering properties of the homology cobordism group $\Theta^3_{\Z}$. For example, it is an open problem whether this group contains a $\Z^{\infty}$ summand.
\medskip

\noindent \textbf{Organization of the paper.}  In Section~\ref{sec:background} we review the definition and some properties of involutive Heegaard Floer homology. In Section~\ref{sec:graphs} we review the graph cobordism maps on Heegaard Floer homology constructed by the third author in \cite{Zemke2}, and explain their interaction with the conjugation symmetry. In Section~\ref{sec:absolute} we study the behavior of the graph cobordism maps with respect to the absolute gradings on Heegaard Floer homology. In Section~\ref{sec:proof} we use these tools to prove the first version of the connected sum formula, where $\iota$ is identified with the expression \eqref{eq:withPhi}. In Section~\ref{sec:nullchainhomotopy} we prove that the second term in \eqref{eq:withPhi} is chain homotopic to zero, and thus arrive at our main result, Theorem~\ref{thm:ConnSum}. In Section~\ref{sec:inequalities} we deduce the inequalities in Proposition~\ref{prop:inequalities}. In Section~\ref{sec:I} we construct the group $\Inv$, and prove Theorem~\ref{thm:Inv}. In Section~\ref{sec:computations} we do some explicit computations, which yield Propositions~\ref{prop:237sums}, \ref{prop:alldifferent}, and \ref{prop:undetermined}. Lastly, in Section~\ref{sec:SW} we make an analogy with Seiberg-Witten theory, and describe a potential  connected sum formula in the $\Ainf$ setting.

\medskip
\noindent \textbf{Acknowledgements.} We are indebted to Matt Stoffregen for several very helpful insights, including the suggestion that we consider the three-manifold in Proposition \ref{prop:alldifferent}; Section~\ref{sec:SW} has also come out of discussions with him. We thank Tye Lidman, Francesco Lin and Charles Livingston for comments on a previous version of the paper.

\section{Background}
\label{sec:background}

\subsection{Review of involutive Heegaard Floer homology}
\label{sec:HFIreview}
We recall the construction and some of the properties of involutive Heegaard Floer homology, following \cite{HMinvolutive}. We assume that the reader is familiar with ordinary Heegaard Floer homology, as in \cite{HolDisk, HolDiskTwo, HolDiskFour, AbsGraded}.

Given a closed, oriented three-manifold $Y$, there is a conjugation symmetry that acts on the set of $\spinc$ structures on $Y$. Involutive Heegaard Floer homology is associated to the three-manifold and any equivalence class of $\spinc$ structures under this involution. However, it is proved in \cite[Proposition 4.5]{HMinvolutive} that, in the case where the equivalence class consists of two distinct $\spinc$ structures, the involutive groups are determined by the ordinary ones. The interesting case is when the equivalence class consists of a single, self-conjugate $\spinc$ structure, i.e., one that comes from a spin structure.\footnote{Different spin structures may give rise to the same self-conjugate $\spinc$ structure. Our constructions only depend on the self-conjugate $\spinc$ structure, but in this paper we refer to the data of a spin structure, for simplicity.} This is the only case that we will consider in this paper.

Let us fix a spin structure $\s$ on $Y$. We choose a pointed Heegaard diagram
$$ \H = (\Sigma, \alphas, \betas, w)$$
representing $Y$. (Unlike in \cite{HMinvolutive}, here we drop almost complex structures from the notation for simplicity.) There are Heegaard Floer complexes $\CFo(\H, \s)$, where $\circ$ denotes any of the four versions: $\widehat{\phantom{a}}$, $+$, $-$, or $\infty$. When there is no possibility of confusion, we write $\CFo(Y, \s)$ for $\CFo(\H, \s)$; however, the chain complex depends on the diagram, and only the homology $\HFo(Y, \s)$ is an invariant of $(Y, \s)$. 

The conjugate Heegaard diagram
$$ \bH = (-\Sigma, \betas, \alphas, w)$$
represents $Y$ as well, and we have a canonical isomorphism between the corresponding Heegaard Floer chain complexes:
\begin{align*}
\eta \co \CFo(\H, \s) \xrightarrow{\phantom{u} \cong \phantom{u}} \CFo(\bH, \s).
\end{align*}

Furthermore, a sequence of Heegaard moves from $\bH$ to $\H$ induces a chain homotopy equivalence
\begin{align*}
\Phi(\bH, \H) \co \CFo(\bH,  \s) \xrightarrow{\phantom{u} \sim \phantom{u}} \CFo(\H, \s).
\end{align*}
We denote by $\inv$ the composition of these two maps:
$$ \inv = \Phi(\bH, \H) \circ \eta \co \CFo(\H, \s) \to \CFo(\H, \s).$$

The involutive Heegaard Floer complex $\CFIo(\H, \s)$, is defined as the mapping cone
\begin{align} \label{eq:involutive}
 \CFo(\H, \s) \xrightarrow{Q (1+\inv)} Q \ccdot \CFo(\H, \s)[-1].
\end{align}
As with Heegaard Floer homology, we sometimes write $\CFIo(Y, \s)$ for $\CFIo(\H, \s)$, when the Heegaard diagram is implicit. The (isomorphism class of the) homology of $\CFIo(\H, \s)$, as a module over $\Ring = \Z_2[Q, U]/(Q^2)$, is a three-manifold invariant, called {\em involutive Heegaard Floer homology}, and denoted $\HFIo(Y, \s)$. 

Note that, if $\x \in \CFo(\H, \s)$ is a homogeneous generator in degree $d$, then the corresponding $\x \in \CFIo(\H, \s)$, in the domain of the map \eqref{eq:involutive}, is in degree $d+1$, whereas $Q \x$, in the target of \eqref{eq:involutive}, is in degree $d$.

\begin{remark}\label{remark:suppress}
When $Y$ is a $\Z_2$-homology sphere, it admits a unique spin structure $\s$, which we usually drop from the notation.
\end{remark}

\begin{remark}
We expect that the modules $\HFIo(Y, \s)$ are natural, i.e., they depend only on $(Y, \s)$ and the basepoint $w$, up to canonical isomorphism. However, in \cite{HMinvolutive} it is only proved that their isomorphism class is an invariant. Therefore, in the following, whenever we talk about maps to or from some $\HFIo(Y, \s)$, what we mean more precisely is that there exist such maps given a choice of a Heegaard diagram.
\end{remark}

Let us list a few properties of $\HFIo(Y, \s)$, which were proved in \cite{HMinvolutive}. The first, Proposition 4.1 in \cite{HMinvolutive}, deals with the relation between the different flavors of involutive Heegaard Floer homology.

\begin{proposition} \label{propn:longexact} The involutive Heegaard Floer groups have long exact sequences
\begin{align*}
\cdots \rightarrow \HFIhat_*(Y,\s) \rightarrow \HFIp_*(Y,\s) \xrightarrow{\cdot U} \HFIp_{*-2}(Y, \s) \rightarrow \cdots \\
\cdots \rightarrow \HFIm_*(Y,\s) \xrightarrow{i} \HFIinf_*(Y, \s) \xrightarrow{\pi} \HFIp_*(Y, \s) \rightarrow \cdots 
\end{align*}
\noindent where $i$ and $\pi$ denote the maps induced by inclusion and projection.
\end{proposition}
 
Our next proposition, Proposition 4.6 in \cite{HMinvolutive}, deals with the relationship between $\HFo$ and $\HFIo$.

\begin{proposition} \label{propn:exact} Let $\s$ be a spin structure on $Y$. Then, there is an exact triangle of $U$-equivariant maps relating $\HFIo$ to $\HFo$:
\begin{equation}
\label{pic:exact1}
\begin{tikzpicture}[baseline=(current  bounding  box.center)]
\node(1)at(0,0){$\HFIo(Y,\s)$};
\node(2)at (-2,1){$\HFo(Y, \s)$};
\node(3)at (2,1){$Q \ccdot \HFo(Y,\s)[-1]$};
\path[->](2)edge node[above]{$Q(1+ \inv_*)$}(3);
\path[->](3)edge (1);
\path[->](1)edge(2);
\end{tikzpicture}
\end{equation}
Here, the map $\HFIo(Y, \s) \to \HFo(Y, \s)$ decreases grading by $1$, and the other two maps are grading-preserving.
\end{proposition}

This implies, for example, that when $b_1(Y)=0$, we have $\HFIinf(Y,\s) \simeq \F_2[Q,U,U^{-1}]/(Q^2)$. Next, we consider the behavior of the involutive Floer groups under orientation reversal. This is Proposition 4.4 in \cite{HMinvolutive}.

\begin{proposition}\label{prop:orientation}
 If $\s$ is a spin structure on $Y$, there is an isomorphism
\[
\CFIp_r(Y, \s) \rightarrow \CFI_{-}^{-r-1}(-Y,\s),
\]
where $\CFI_-$ denotes the cochain complex dual to $\CFIm$.
\end{proposition}

The next proposition is the large surgery formula, Theorem 1.5 in \cite{HMinvolutive}. To state it, recall that the knot surgery formula in ordinary Heegaard Floer homology \cite{IntSurg} makes use of a certain quotient complex $A_0^+$ of the knot Floer complex $\CFKinfty(K)$. There is a conjugation map $\iota_K$ on $\CFKinfty(K)$ which induces a conjugation map $\iota_0$ on $A_0^+$, and we let $\AI_0^+$ be the mapping cone
$$
A_0^+ \xrightarrow{\phantom{o} Q (1+\inv_0) \phantom{o}} Q \ccdot A_0^+ [-1]. 
$$

\begin{proposition}\label{prop:surgeries}
Let $K \subset S^3$ be a knot, and let $g(K)$ be its Seifert genus. Then, for all integers $p \geq g(K),$ we have an isomorphism of relatively graded $\Ring$-modules
$$ \HFIp(S^3_p(K), [0]) \cong H_*(\AI_0^+),$$
where $[0]$ denotes the $\spinc$ structure on $S^3_p(K)$ that corresponds to zero under the standard identification $\operatorname{Spin}^c(Y_p(K)) \cong \Z/p\Z$; cf. \cite[Lemma 2.2]{IntSurg}.
\end{proposition}

In fact, the proof of Theorem 1.5 in \cite{HMinvolutive} gives the following slightly stronger version:
\begin{proposition}\label{prop:surgeries2}
Let $K \subset S^3$ be a knot, and let $g(K)$ be its Seifert genus. Then, for all integers $p \geq g(K),$ we have a chain homotopy equivalence over $\Ring$:
$$ F: \CFIp(S^3_p(K), [0]) \xrightarrow{\simeq} \AI_0^+$$
that interchanges the two conjugation maps up to chain homotopy over $\Ring$:
$$ F \circ \iota \simeq \iota_0 \circ F.$$
\end{proposition}

\subsection{Correction terms} From involutive Heegaard Floer homology, one extracts two new invariants of homology cobordism, $\dl, \du \co \Theta^3_{\F_2} \rightarrow \mathbb Q$. In \cite[Section 5]{HMinvolutive}, these maps were defined in terms of the variant $\HFIp(Y,\s)$, as follows.

\begin{definition} \cite[Section 5.1]{HMinvolutive} \label{def:correctionterms} Let $Y$ be a rational homology sphere equipped with a spin structure $\s$. The involutive correction terms are
\[\dl(Y,\s) = \min \{r \mid \exists \ x \in \HFIp_r(Y, \s), x \in \operatorname{Im}(U^n), x \not \in \operatorname{Im}(U^nQ)  \text{ for } n \gg 0\} - 1\]
and
\[\du(Y,\s) = \min \{r \mid \exists \ x \in \HFIp_r(Y, \s), x \neq 0, \ x  \in \operatorname{Im}(U^nQ)  \text{ for } n \gg 0\}.\]
Equivalently,  $\du(Y,\s)=r$ is the minimal grading such that $t \equiv d(Y,\s)$ modulo $2\Z$  and $\pi \co \HFIinf_r(Y,\s) \rightarrow \HFIp_r(Y,\s)$ is nontrivial, and $\dl(Y,\s)+1=r$ is one more than the minimal grading $r$ such that $r \equiv d(Y,\s)+1$ modulo $2\Z$  and $\pi \co \HFIinf_r(Y,\s) \rightarrow \HFIp_r(Y,\s)$ is nontrivial.\end{definition}

In this paper, we work with $\HFIm$, and therefore it will be helpful to rephrase the definition above in terms of $\HFIm(Y,\s)$. 

\begin{lemma} \label{lem:newinv} The involutive correction terms are given by
\[ \dl(Y,\s) = \max \{r \mid \exists \ x \in \HFIm_r(Y, \s), \forall \ n, \ U^nx\neq 0 \ \text{and} \ U^nx \notin \operatorname{Im}(Q)\} + 1 \]
and
\[ \du(Y,\s) = \max \{r \mid \exists \ x \in \HFIm_r(Y,\s),\forall \ n, U^nx\neq 0; \exists \ m\geq 0 \ \operatorname{s.t.} \ U^m x \in \operatorname{Im}(Q)\} +2. \]

\noindent Equivalently, $\du(Y,\s)=r+2$ is two more than the maximal grading $r$ such that $r \equiv d(Y,\s)$ modulo $2\Z$ and the map $i \co \HFIm_r(Y,\s) \rightarrow \HFIinf_r(Y,\s)$ is nonzero, and $\dl(Y,\s)=q+1$ is one more that the maximal grading $q$ such that $q \equiv d(Y,\s)+1$ modulo $2\Z$ and $i \co \HFIm_q(Y,\s) \rightarrow \HFIinf_q(Y,\s)$ is nonzero.

\end{lemma}

\begin{proof} In light of the existence of the long exact sequence
\[
\cdots \rightarrow \HFIm(Y,\s) \xrightarrow{i} \HFIinf(Y, \s) \xrightarrow{\pi} \HFIp(Y, \s) \rightarrow \cdots
\]
and the isomorphism $\HFIinf(Y,\s) \simeq \F_2[Q,U,U^{-1}]/(Q^2)$, the characterization of $\dl$ and $\du$ in terms of the map $i$ follows immediately from the characterization of $\dl$ and $\du$ in terms of the map $\pi$ from Definition \ref{def:correctionterms}. It remains to be shown that this is equivalent to the first description given. However, notice that because the long exact sequence above is $U$-equivariant, $i(x)\neq 0$ if any only if $U^n x \neq 0$ for all $n\geq 0$. (For the if direction, recall that there are no elements $y$ in $\HFIp(Y,\s)$ with $U^n y \neq 0$ for all $n$, and therefore $x$ cannot be in the image of $\pi$. Hence $i(x)\neq 0$.) Furthermore, from the exact triangle
\begin{equation}
\label{pic:exact2}
\begin{tikzpicture}[baseline=(current  bounding  box.center)]
\node(1)at(0,0){$\HFIm(Y,\s)$};
\node(2)at (-2,1){$\HFm(Y, \s)$};
\node(3)at (2,1){$Q \ccdot \HFm(Y,\s)[-1]$};
\path[->](2)edge node[above]{$Q(1+ \inv_*)$}(3);
\path[->](3)edge (1);
\path[->](1)edge(2);
\end{tikzpicture}
\end{equation}
we see that homogeneous elements $x \in \HFIm(Y,\s)$ for which $U^n x \neq 0$ in gradings congruent to $d(Y,\s)$ modulo $2\Z$ must have the property that $U^mx \in \operatorname{Im}(Q \ccdot\HFm(Y,\s))$ for sufficiently large $m$, and homogeneous elements $x \in \HFIm(Y,\s)$ for which $U^nx \neq 0$ which lie in gradings congruent to $d(Y,\s)+1$ modulo $2\Z$ have $U^n x \notin \operatorname{Im}(Q \ccdot \HFm(Y,\s))$ for all $n$. The conclusion follows. \end{proof}

\begin{proposition} \cite[Equation 18 and Proposition 5.1]{HMinvolutive} The involutive correction terms satisfy
\begin{align*}
\dl(Y,\s) \leq & d(Y,\s) \leq \du(Y,\s), \\
\du(Y, \s) \equiv & \dl(Y, \s) \equiv d(Y, \s) \pmod {2\Z}.
\end{align*}
\end{proposition}

\begin{proposition} \cite[Proposition 5.2]{HMinvolutive} \label{prop:orientationsigns} Let $(Y,\s)$ be a rational homology sphere with a spin structure $\s$. The involutive correction terms are related under orientation reversal by 
\[\dl(Y,\s)=-\du(-Y,\s).\]
\end{proposition}

For this paper, it will be helpful to have the following reformulation of the definition of the involutive correction terms, which gives a formula for computing $\dl$ and $\du$ from $\CFm(Y,\s)$ and $\iota$, without computing the full complex $\HFIm(Y,\s)$.

\begin{lemma} \label{lemma:reformulation} Let $\iota$ be the conjugation symmetry on $\CFm(\H,\s)$. Then:
\begin{enumerate}[(a)]
\item The lower involutive correction term $\dl(Y,\s)$ is two more than the maximum grading of a homogeneous element $v \in \CFm(\H,\s)$ such that $\partial v = 0$, $[U^nv]\neq 0$, for any $n \geq 0$ and there exists $w \in \CFm(\H,\s)$ such that $\partial w= (1+\iota)v$. 
\smallskip
\item
For the upper involutive correction term, consider sets of homogeneous elements $x,y,z \in \CFm(\H,\s)$, with at least one of $x$ and $y$ nonzero, such that  $\partial y = (1+\iota)x$, $\partial z = U^m x$ for some $m \geq 0$, and $[U^n(U^my + (1+ \iota) z)] \neq 0$ for any $n \geq 0$. Then $\du(Y,\s)$ is either three more than the maximum possible grading of $x$ or, if $x$ is zero, two more than the maximum possible grading of $y$.
\end{enumerate}
\end{lemma}

\begin{proof} (a) We observe that the existence of the exact triangle 
\begin{equation}
\label{pic:exact3}
\begin{tikzpicture}[baseline=(current  bounding  box.center)]
\node(1)at(0,0){$\HFIm(Y,\s)$};
\node(2)at (-2,1){$\HFm(Y, \s)$};
\node(3)at (2,1){$Q \ccdot \HFm(Y,\s)[-1]$};
\path[->](2)edge node[above]{$Q(1+ \inv_*)$}(3);
\path[->](3)edge (1);
\path[->](1)edge(2);
\end{tikzpicture}
\end{equation}
implies that $\dl(Y,\s)$ is two more than the maximal degree of an element $\alpha$ in $ker(1+\iota_*)$ such that $U^n\alpha \neq 0 \in \HFm(Y,\s)$. (More concretely, we see that $[v + Qw]$ generates the first tail in $\HFIm(Y,\s)$.)

(b) Suppose that $Qa$ is a closed element of $\CFIm(Y,\s)$ with the property that $[U^nQa] \neq 0$ for all $n>0$, that is, such that $[Qa]$ lies in the second tail of $\HFIm(Y,\s)$. Notice that the element $a$ must also generate the $U$-tail in $\HFm(Y,\s)$ (for if $\partial b = U^{\ell}a$ for some $\ell$, then $\partial^{\iota} Qb = U^{\ell}Qa$ as well). Then by definition, $\du(Y,\s)$ is three more than the maximum grading of a closed, homogeneous element $x+Qy \in \CFIm(Y,\s)$ with the property that $[U^m(x+Qy)] = [Qa]$ for some $m\geq 0$. Since $0=\partial^{\iota}(x+Qy) = \partial x + Q((1+\iota)x + \partial y) $, we see that $\partial y = (1+\iota)x$. Furthermore, suppose we choose an element $z+Qc$ with the property that
\begin{align*}
\partial^{\iota}(z+Qc)&= U^m(x+Qy) + Qa \\
\partial z + (1+\iota)Qz + Q\partial c &=U^mx + Q(U^my+a)
\end{align*}
This implies that $U^mx = \partial z$ and $(1+\iota)z + \partial c = U^m y +a$. Rearranging, we see that $(1+\iota)z + U^my = a + \partial c$, or in particular $[(1+ \iota)z + U^m y] = [a]$ generates the $U$-tail in $\HFm(Y,\s)$. \end{proof}

\section{The graph TQFT for Heegaard Floer homology}
\label{sec:graphs}

Our proof of the connected sum formula for involutive Heegaard Floer homology will use the graph cobordism maps from \cite{Zemke2}. In this section, we review the maps involved, summarize important relations, and explain the interaction with the involution.

\subsection{Multiple basepoints} We will need to use a more general construction of Heegaard Floer complexes, which allows for multiple basepoints, as in \cite{Links, Zemke2}. Consider a closed, oriented, possibly disconnected three-manifold $Y$, together with a collection of basepoints
$$ \ws = \{w_1, \dots, w_k\} \subset Y,$$
with at least one basepoint on each component. This can be represented by a multi-pointed Heegaard diagram 
$$ \H = (\Sigma, \alphas, \betas, \ws),$$
with $\ws \subset \Sigma$ (and $\Sigma$ is possibly disconnected). One can then define Heegaard Floer complexes, generated by intersection points $\x \in \Ta \cap \Tb$, and keeping track of the basepoints in various ways. For our purposes, it suffices to consider the version where we have a single variable $U$, and the differential on the Floer complex is given by
$$ \del \mathbf{x} =\sum_{\mathbf{y}\in \mathbb{T}_\alpha\cap \mathbb{T}_\beta}\sum_{\substack{\phi\in \pi_2(\mathbf{x},\mathbf{y})\\\mu(\phi)=1}} \# \widehat{\mathcal{M}}(\phi) \cdot U^{n_{\mathbf{w}}(\phi)} \mathbf{y},$$
with
$$n_{\ws}(\phi) = n_{w_1}(\phi) + \dots + n_{w_k}(\phi).$$
For a $\spinc$ structure $\s$ on $Y$, we will denote the resulting complex by $\CFm(Y, \ws, \s)$.\footnote{This is not quite standard. In the literature, $\CFm$ is usually defined with one $U$ variable for each basepoint. More generally, e.g., as in \cite{Zemke2}, one can color the basepoints and associate a $U$ variable to each color. In this paper we use the trivial coloring.} There are related complexes $\CFinf = U^{-1} \CFm$, $\CFp = \CFinf / \CFm$, and $\CFhat =\CFm/ (U=0)$. We use $\CFo$ to denote any of these versions.

The conjugation map $\iota$ can be defined in the presence of multiple basepoints. Given a diagram $\H=(\Sigma, \alphas,\betas,\mathbf{w})$,  we form the conjugate $\overline{\H}=(-\Sigma, \betas,\alphas,\mathbf{w})$. The map $\iota$ is then defined as the canonical map $\eta$ composed with a change of diagrams map induced by Heegaard moves from $\overline{\H}$ to $\H$. The map $\iota$ is $\Z_2[U]$-equivariant and well defined up to chain homotopy.

Let us say a few words about disconnected manifolds. Suppose $Y_i$ ($i=1,2$) are three-manifolds equipped with $\spinc$ structures $\s_i$ and sets of basepoints $\ws_i$. Then, for the disjoint union $Y = Y_1 \sqcup Y_2$, we have a natural isomorphism
\begin{equation}
\label{eq:disjunion}
 \CFm(Y, \ws_1 \cup \ws_2, \s_1 \cup \s_2) \cong \CFm(Y_1, \ws_1, \s_1) \otimes_{\Z_2[U]} \CFm(Y_2, \ws_2, \s_2).
 \end{equation}
 Similar isomorphisms hold for the versions $\CFinf$ and $\CFhat$, with the tensoring done over $\Z_2[U, U^{-1}]$ and $\Z_2$, respectively. (No such formula is available for $\CFp$, however.) Furthermore, for $\circ \in \{-, \infty, \widehat{\phantom{a}}\}$, if $\iota_i$ denote the conjugation involutions on $\CFo(Y_i, \ws_i, \s_i)$, then, with respect to the isomorphisms of the type \eqref{eq:disjunion}, the conjugation involution $\iota$ for $Y$ is given by $\iota_1 \otimes \iota_2$.

\subsection{The maps \texorpdfstring{$\Phi$}{Phi}}
\label{sec:Phi}
We now describe the maps $\Phi$ that made an appearance in the statement of Theorem~\ref{thm:ConnSum}. Given a multi-pointed Heegaard diagram $\H$ as above, to every $w \in \ws$ we can associate a chain map
$$ \Phi_w \co \CFm(Y, \ws, \s) \to \CFm(Y, \ws, \s)$$
by the formula
\begin{equation}
\label{eq:Phiw} \Phi_w(\mathbf{x})=\sum_{\mathbf{y}\in \mathbb{T}_\alpha\cap \mathbb{T}_\beta}\sum_{\substack{\phi\in \pi_2(\mathbf{x},\mathbf{y})\\\mu(\phi)=1}} \# \widehat{\mathcal{M}}(\phi) \cdot n_w(\phi)U^{n_{\mathbf{w}}(\phi)-1}\mathbf{y}.\end{equation}

We can interpret the sum of $\Phi_w$ as ``the formal $U$-derivative'' of the differential $\del$. There is a caveat here, that this is the $U$-derivative when we apply $\del$ to intersection points, rather than to linear combinations of them that may contain nontrivial powers of $U$. More precisely, if one defines $\frac{d}{dU}$ on $\CFm$ by
$$ \frac{d}{dU} \co \sum_{\x \in \Ta \cap \Tb} U^{n_{\x}} \x \mapsto  \sum_{\x \in \Ta \cap \Tb} n_{\x} U^{n_{\x}-1} \x,$$
then
 \[\frac{d}{d U}\circ \partial+\partial\circ \frac{d}{d U}= \sum_{w \in \ws} \Phi_w.\]

The map $\frac{d}{dU}$ is not $U$-equivariant. Thus, we see that $\sum_{w \in \ws} \Phi_w$ is chain homotopic to zero, albeit not by an $U$-equivariant homotopy.

The maps $\Phi_w$ can also be defined on the complexes $\CFinf$, $\CFp$ and $\CFhat$, by similar  formulas. In the case of $\CFhat$, we set $U=0$ in the above formula, so we only pick up contributions from the differential that have $n_{\ws}(\phi)=1$.

\subsection{Overview of the graph cobordism maps}
In \cite{Zemke2}, the third author introduced a ``graph TQFT'' for Heegaard Floer homology. We say that a graph is cyclically ordered if, at every vertex, it comes equipped with a cyclic order of the edges  incident to that vertex. To a 4-dimensional cobordism $X$ with ends $(Y_1,\mathbf{w}_1)$ and $(Y_2,\mathbf{w}_2)$, an embedded, cyclically ordered graph $\Gamma\subset X$ interpolating the basepoints on the ends, and a $\spinc$ structure $\mathfrak{s}$ on $X$, there is a map
\[F_{W,\Gamma,\mathfrak{s}}^\circ \co CF^\circ(Y_1,\mathbf{w}_1,\mathfrak{s}|_{Y_1})\to CF^\circ(Y_2,\mathbf{w}_2,\mathfrak{s}|_{Y_2}).\]
These maps satisfy the standard $\spinc$ composition law and are defined up to $\Z_2[U]$-equivariant chain homotopy. In \cite{Zemke2}, the maps also required a choice of coloring of the basepoints and the graph $\Gamma$; in this paper we only consider the trivial coloring, which we suppress from the notation.

The maps $F_{W,\Gamma,\mathfrak{s}}^\circ$ are defined as a composition of handle attachment maps similar to the ones defined by Ozsv\'{a}th and Szab\'{o} in \cite{HolDiskFour}, as well as two new maps, introduced by the third author. The first new map is the free-stabilization map, which gives a  map between complexes when we add or remove a new basepoint.  The second map  is the relative homology map, corresponding to an action of certain relative homology groups when there are multiple basepoints. We will adopt slightly different notation than in \cite{Zemke2}. For the maps above, as constructed in \cite{Zemke2}, we will write $F_{W,\Gamma,\mathfrak{s}}^A$. We will have need for a variation, $F_{W,\Gamma,\mathfrak{s}}^B$, which we will define below.

\subsection{Free-stabilization maps}

The third author introduced free-stabilization maps in \cite[Section 6]{Zemke2}, which are maps corresponding to adding or removing basepoints. These are maps
\[S_w^{+}\co CF^\circ(Y,\mathbf{w},\mathfrak{s})\to CF^\circ(Y,\mathbf{w}\cup \{w\},\mathfrak{s})\] and
\[S_w^{-} \co CF^\circ(Y,\mathbf{w}\cup \{w\},\mathfrak{s})\to CF^\circ(Y,\mathbf{w},\mathfrak{s}).\] 
We will refer to them as the positive and negative free-stabilization maps. (One could also call $S_w^-$ a free-destabilization map.)

We have the following:

\begin{lemma}\label{lem:iotaandS_wcommute}The maps $\iota$ and $S_{w}^{\pm}$ satisfy \[\iota\circ S_{w}^{\pm}\simeq S_{w}^{\pm}\circ \iota.\]
\end{lemma}
\begin{proof}The map $\iota$ is a composition of two maps, the tautological map $\eta$, as well as a change of diagrams map. Since it is shown in \cite{Zemke2} that $S_{w}^{\pm}$ commutes with the change of diagrams map up to chain homotopy, it is sufficient to show that $S_{w}^{\pm}$ commutes with the canonical involution map, which is easily computed.
\end{proof}

We will use the following notation to denote the composition of free-stabilization maps (of the same type) at different basepoints:
$$ S^{\pm}_{w_1w_2 \dots w_k} = S^{\pm}_{w_1} S^{\pm}_{w_2} \dots S^{\pm}_{w_k}.$$

\subsection{Relative homology maps}\label{subsec:relhomologymaps}
In \cite[Section 5]{Zemke2}, the third author defines relative homology maps on the Heegaard Floer chain complexes. Given two basepoints $w,w'\in \mathbf{w}$ and a path $\lambda$ in $Y$ from $w$ to $w'$, there is a chain map\footnote{In the case of a general coloring considered in \cite{Zemke2}, the maps $A_\lambda$ are not chain maps, but instead chain homotopies between the $U$ actions corresponding to $w$ and $w'$. In our case, all $U$ variables are identified.}
\[A_\lambda \co CF^\circ(Y,\mathbf{w},\mathfrak{s})\to CF^\circ(Y,\mathbf{w},\mathfrak{s}).\] The map $A_\lambda$ is defined by first performing a homotopy of the path $\lambda$ so that it lies in the Heegaard surface $\Sigma$, and then defining $A_\lambda(\mathbf{x})$ by the formula
\[A_\lambda(\mathbf{x})=\sum_{\mathbf{y}\in \mathbb{T}_\alpha\cap \mathbb{T}_\beta}\sum_{\substack{\phi\in \pi_2(\mathbf{x},\mathbf{y})\\\mu(\phi)=1}} a(\lambda,\phi)\# \widehat{\mathcal{M}}(\phi)U^{n_{\mathbf{w}}(\phi)}\cdot \mathbf{y}\] where $a(\lambda,\phi)$ denotes the sum of changes in multiplicities of $\phi$ across just the $\mathbf{\alpha}$-curves (but not the $\beta$-curves), as one travels along the path $\lambda$. In \cite[Section~5]{Zemke2}, it is shown that the map $A_\lambda$ depends only on the relative homology class of $\lambda$ (up to chain homotopy) and commutes up to chain homotopy with the change of diagrams maps.

As a variation, one can also consider the maps $B_\lambda$, which count Maslov index one disks, but with a factor equal to the sum of the changes in multiplicities across any $\mathbf{\beta}$-curve as one traverses $\lambda$.  The maps $B_\lambda$ and the maps $A_\lambda$ are related, though are not equal or even chain homotopic in all cases. Indeed, to compute the endomorphism $A_{\lambda}+B_{\lambda}$, one counts the changes across all $\mathbf{\alpha}$- and $\mathbf{\beta}$-curves as one traverses $\lambda$. The sum telescopes and we are left with
\[A_\lambda(\mathbf{x})+B_{\lambda}(\mathbf{x})=\sum_{\mathbf{y}\in \mathbb{T}_\alpha\cap \mathbb{T}_\beta}\sum_{\substack{\phi\in \pi_2(\mathbf{x},\mathbf{y})\\\mu(\phi)=1}} (n_w(\phi)+n_{w'}(\phi))\# \widehat{\mathcal{M}}(\phi)U^{n_{\mathbf{w}}(\phi)}\cdot \mathbf{y},\] showing that
\[A_\lambda+B_{\lambda}=U ( \Phi_w +\Phi_{w'}),\] where $\Phi_w$  and $\Phi_{w'}$ are as defined in \eqref{eq:Phiw}.

The maps $F_{W,\Gamma,\mathfrak{s}}^\circ$ defined in \cite{Zemke2} were defined using the maps $A_\lambda$, but one could also define the maps by replacing each instance of $A_\lambda$ with $B_\lambda$, and the same proof of invariance holds. We define
\[F_{W,\Gamma,\mathfrak{s}}^A \qquad \text{ and }\qquad  F_{W,\Gamma,\mathfrak{s}}^B\] to be the graph cobordism maps defined with the maps $A_\lambda$ or $B_\lambda$ respectively. Hence the maps $F_{W,\Gamma,\mathfrak{s}}^A$ denote the maps defined in \cite{Zemke2}, which were denoted by $F_{W,\Gamma,\mathfrak{s}}^\circ,$ in that context.

The maps $A_\lambda$ and $B_\lambda$ do not commute with the involution. Instead we have the following:

\begin{lemma}\label{lem:relativehomologyandiota}The maps $A_\lambda,B_\lambda$ and $\iota$ satisfy
\[A_\lambda\circ \iota\simeq \iota \circ B_\lambda.\]
\end{lemma}

\begin{proof}The map $\iota$, as defined in \cite{HMinvolutive}, is defined as a composition of two maps. The first is a ``canonical'' involution which maps between two separate complexes, and the second is a change of diagrams map. If $\eta$ denotes the canonical involution, then we have the equality
\[A_\lambda\circ \eta=\eta \circ B_\lambda,\] because the involution switches the roles of the $\mathbf{\alpha}$- and $\mathbf{\beta}$-curves. In \cite{Zemke2}, it is shown that the maps $A_\lambda$ commute with the change of diagrams maps up to chain homotopy. The same argument works to show that the maps $B_\lambda$ commute up to chain homotopy with the change of diagrams maps.
\end{proof}

\subsection{Summary of the relations between maps appearing in the graph TQFT}

In this subsection, we collect some important relations that the free-stabilization maps, relative homology maps, and the maps $\Phi_w$ satisfy. Unless otherwise noted, the symbol $\simeq$ will  mean $\Z_2[U]$-equivariantly chain homotopic.

If $\lambda$ is a path from $w$ to $w'$, then we have 
(after identifying $U_w$ and $U_{w'}$ to a common variable $U$)
\begin{equation}\partial A_\lambda+ A_\lambda\partial=\partial B_\lambda+B_\lambda\partial=0\tag{R1} \label{eq:R1}, 
\end{equation}

\begin{equation} A_\lambda^2\simeq B_\lambda^2\simeq U \tag{R2} \label{eq:R2},
\end{equation}

\begin{equation}A_\lambda+B_\lambda=U\Phi_w+U\Phi_{w'} \tag{R3} \label{eq:R3},
\end{equation}
and

\begin{equation}A_\lambda\Phi_w+\Phi_w A_\lambda=B_\lambda\Phi_w+\Phi_w B_\lambda\simeq 1 \tag{R4} \label{eq:R4}.
\end{equation}
The proofs can be found in \cite[Lemma 5.3]{Zemke2}, \cite[Lemma 5.11]{Zemke2}, Subsection~\ref{subsec:relhomologymaps} of the present paper, and \cite[Lemma 14.10]{Zemke2}, respectively.

The map $\Phi_w$ satisfies
\begin{equation} \Phi_w^2\simeq 0 \tag{R5} \label{eq:R5},
\end{equation}
$\Z_2[U]$-equivariantly. The proof can be found in \cite[Proposition 14.12]{Zemke2}.

If $w$ is not the only basepoint on the component of $Y$ containing $w$, then
\begin{equation}\Phi_w\simeq S_w^+S_w^-\tag{R6} \label{eq:R6}.\end{equation}The proof can be found in \cite[Remark 14.9]{Zemke2}.

If $\lambda$ is a path from $w$ to $w'$, and $\phi$ denotes the diffeomorphism induced by a finger move along $\lambda$ from $w$ to $w'$, then
\begin{equation}
S_{w}^- A_\lambda S_{w}^+\simeq \id  \label{eq:R7} \tag{R7},
\end{equation}
and 
\begin{equation}S_{w}^- A_\lambda S_{w'}^+\simeq \phi_*\label{eq:R8} \tag{R8}.
\end{equation}
The proofs of these two relations can be found in \cite[Lemma 7.7]{Zemke2} and \cite[Lemma 14.7]{Zemke2}, respectively. Similar relations hold for $B_{\lambda}$ instead of $A_{\lambda}$.

In addition, if $w$ and $w'$ are two different basepoints, then
\begin{equation}\label{eq:R9}
S_{w}^{\circ} S_{w'}^{\circ'}\simeq S_{w'}^{\circ'}S_{w}^{\circ} \tag{R9},\end{equation}  for $\circ,\circ'\in \{+,-\},$ and

\begin{equation}\label{eq:R10} \Phi_{w} S_{w'}^{\pm}\simeq S_{w'}^{\pm}\Phi_w \tag{R10}.\end{equation} Equation \eqref{eq:R9} is proven in \cite{Zemke2}. Equation \eqref{eq:R10} can be proven by differentiating the expression $0=\partial S_{w'}^{\pm}+S_{w'}^{\pm} \partial$ with respect to $U$.

If $w$ is a basepoint which is not an endpoint of a path $\lambda$, then
\begin{equation}A_\lambda S_{w}^{\pm} \simeq S_{w}^{\pm} A_\lambda \tag{R11} \label{eq:R11}.
\end{equation}
The proof can be found in \cite[Lemma 6.6]{Zemke2}. A similar relation holds for $B_{\lambda}$.

If $w$ is a basepoint, then we have
\begin{equation}S_w^- S_w^+\simeq 0 \tag{R12} \label{eq:R12},
\end{equation}
which is immediate from the formulas defining $S_w^{+}$ and $S_w^-$.

\subsection{The graph action map}
\label{sec:GraphAction}
For future reference, we mention that one of the ingredients of the graph cobordism maps is a so called ``graph action map'' for graphs embedded in a fixed three-manifold. A flowgraph $\mathcal{G}=(\Gamma,V_{\text{in}},V_{\text{out}})$ is a cyclically ordered, colored graph $\Gamma$ embedded in $Y^3$ with nonempty, disjoint subsets $V_{\text{in}}$ and $V_{\text{out}}$ of vertices that have valence one in $\Gamma$. According to \cite[Theorem B]{Zemke2}, there is a well defined map
\[A_{\mathcal{G},\mathfrak{s}}\co CF^\circ(Y,V_{\text{in}},\mathfrak{s})\to CF^\circ(Y,V_{\text{out}},\mathfrak{s}),\] which we call the graph action map. To define the graph action map, one takes a decomposition of the flow graph into elementary flow graphs and then defines a map for each piece. The maps for each piece are certain compositions of free-stabilization and relative homology maps. In particular, the map associated to an elementary flow graph with a vertex which has valence at least three requires an ordering of the edges adjacent to the vertex, though the map is invariant under cyclic permutation of the edges. We refer to  \cite[Section 7]{Zemke2} for more details. (The graph action maps there, denoted $A^V_{\mathcal{G},\mathfrak{s}}$, are more general: They depend also on the choice of a subset $V$ of the vertices of $\mathcal{G}$, such that $V_{\text{in}} \subseteq V$. In this paper we only consider the case $V=V_{\text{in}}$, and drop $V$ from the notation, as in \cite[Theorem B]{Zemke2}.)

\subsection{Relation between the involution and the graph cobordism maps}

The involution interacts in a predictable way with the graph cobordism maps.  The maps $F_{W,\Gamma,\mathfrak{s}}^{A}$ and $F_{W,\Gamma, \mathfrak{s}}^B$ are defined as a composition of 1-,2- and 3-handle maps, as well free-stabilization maps and relative homology maps. As shown in \cite{HMinvolutive}, the 1-,2- and 3-handle maps commute with the involution up the chain homotopy. As the graph cobordism maps are defined as a composition of these maps as well as the free-stabilization and relative homology maps, immediately following from Lemmas~\ref{lem:iotaandS_wcommute} and \ref{lem:relativehomologyandiota}, we have the following:

\begin{proposition}\label{prop:iotacommuteswithgraphcobordismmaps}If $(W,\Gamma)\co(Y_1,\mathbf{w}_1)\to (Y_2,\mathbf{w}_2)$ is cyclically ordered, trivially colored graph cobordism, and $\mathfrak{s}$ is a 4-dimensional $\spinc$ structure on $W$, then the following diagram commutes up to equivariant chain homotopy:
\[\begin{tikzcd} CF^\circ(Y_1,\mathbf{w}_1,\mathfrak{s}_1)\arrow{r}{F_{W,\Gamma,\mathfrak{s}}^A}\arrow{d}{\iota}&CF^\circ(Y_2,\mathbf{w}_2,\mathfrak{s}_2)\arrow{d}{\iota}\\
CF^\circ(Y_1,\mathbf{w}_1,\overline{\mathfrak{s}}_1)\arrow{r}{F_{W,\Gamma,\overline{\mathfrak{s}}}^B}& CF^\circ(Y_2,\mathbf{w}_2,\overline{\mathfrak{s}}_2)
\end{tikzcd}.\] Here $\mathfrak{s}_1$ and $\mathfrak{s}_2$ denote the restriction of $\mathfrak{s}$ to $Y_1$ and $Y_2$ respectively.
\end{proposition}

\section{Absolute Gradings and graph cobordisms}
\label{sec:absolute}
 
 In \cite{HolDiskFour}, Ozsv\'{a}th and Szab\'{o} assign an absolute $\Q$-grading to the groups $HF^\circ(Y,w,\s)$, when $Y$ is connected and $c_1(\s)$ is torsion, such that the cobordism maps defined in \cite{HolDiskFour} satisfy a grading change formula. In this section, we extend their absolute $\Q$-grading to 3-manifolds which are disconnected, and may have multiple basepoints, and we compute the grading change induced by graph cobordism maps.
 
 Given a flow graph $\Gamma \co V_0\to V_1$, we define the reduced Euler characteristic
 $\wt{\chi}(\Gamma)$ by the formula
 \[\wt{\chi}(\Gamma)=\chi(\Gamma)-\tfrac{1}{2}(|V_0|+|V_1|).\]
 
 \begin{proposition}\label{prop:gradingchangeformula}There is an absolute $\Q$-grading $\wt{\gr}$ on the chain complexes $CF^\circ(Y,\mathbf{w},\t)$, when $c_1(\t)$ torsion, for arbitrary 3-manifolds $Y$, lifting the relative Maslov grading. The graph cobordism maps $F^\circ_{W,\Gamma,\s}$ are graded, and satisfy
 \begin{equation}\wt{\gr}(F^\circ_{W,\Gamma,\s}(\mathbf{x}))-\wt{\gr}(\mathbf{x})=\frac{ c_1(\s)^2-2\chi(W)-3\sigma(W)}{4}+\wt{\chi}(\Gamma),\label{eq:maingradingformula}\end{equation} for any homogeneous $\x$.  The variable $U$ acts with degree $-2$. The gradings are normalized so that $HF^-(S^3,w)$ ($S^3$ with a single basepoint) has top degree generator in degree $-2$.
 \end{proposition}
 
We will break up the proof into several steps, first establishing the absolute $\Q$-grading, then computing the grading change due to the various maps in the construction of the graph cobordism maps.

\begin{lemma}\label{lem:pathcobordismgradingchange}There is an absolute $\Q$-grading $\wt{\gr}$ on the chain complexes $CF^\circ(Y,\mathbf{w},\t)$, when $c_1(\t)$ torsion, for arbitrary 3-manifolds $Y$, lifting the relative Maslov grading. If $(W,\gammas)$ is a path cobordism (a possibly disconnected cobordism, with nonempty ends in each component, and paths from $Y_1$ to $Y_2$) then the maps $F^\circ_{W,\gammas,\s}$ are graded, and satisfy
 \begin{equation}\wt{\gr}(F^\circ_{W,\gammas,\s}(\mathbf{x}))-\wt{\gr}(\mathbf{x})=\frac{ c_1(\s)^2-2\chi(W)-3\sigma(W)}{4}+\wt{\chi}(\gammas)\end{equation} for any homogeneous $\x$.  The variable $U$ acts with degree $-2$. The gradings are normalized so that $HF^-(S^3,w)$ ($S^3$ with a single basepoint) has top degree generator in degree $-2$.
\end{lemma} 
 
 \begin{proof} First note that for paths $\gammas$, we have $\wt{\chi}(\gammas)=0$. 
 
 Our argument is a slight modification of the proof of \cite[Theorem 7.1]{HolDiskFour}. We first define the absolute grading for disjoint unions of $S^3$, with various collections of basepoints. Suppose that $Y=\sqcup_i S^3_i$, is a disjoint union of copies of $S^3$, and that $\mathbf{w}=\sqcup_i \mathbf{w}_i$ is a collection of basepoints with $\mathbf{w}_i\subset S^3_i$. The $\Z_2[U]$-module $HF^-(Y,\mathbf{w},\t_0)$ is isomorphic to 
 \[ H_*(T^n; \Z_2) \otimes \Z_2[U],\] where $n=\sum_{i}(|\mathbf{w}_i|-1 )$. As such, $HF^-(Y,\ws,\t_0)$ has a top degree generator with respect to the relative Maslov grading. We specify a grading $\wt{\gr}_0$ on disjoint unions of spheres by declaring the top degree generator $\xi_{\text{top}}$ to have grading\footnote{In some of the Heegaard Floer literature, for example in \cite{MOS}, adding an extra $w$ basepoint to a three-manifold is taken to result in tensoring with a two-dimensional vector space $V$ supported in homological gradings $0$ and $-1$. Our conventions here are different: the vector space is supported in gradings $\frac{1}{2}$ and $-\frac{1}{2}$. This is more natural from our viewpoint, since it ensures that the connected sum isomorphisms defined in Section \ref{sec:proof} are grading preserving.}
\begin{equation}\wt{\gr}_0(\xi_{\text{top}})=-2+\tfrac{1}{2}\sum_{i}(|\mathbf{w}_i|-1 ).\label{eq:declaregradinginspheres}
\end{equation}

Using $\wt{\gr}_0$, we can now specify the grading on an arbitrary 
 based 3-manifold $(Y,\mathbf{w})$. Write $(Y,\mathbf{w})=\sqcup_i (Y_i, \mathbf{w}_i)$, where each $Y_i$ is connected. We pick a framed link $\mathbb{L}\subset \sqcup_i S^3$, such that surgery on $\mathbb{L}$ produces $Y$. We then pick a triple $(\Sigma, \alphas,\betas,\gammas,\mathbf{w})$ which is subordinate to a bouquet for $\mathbb{L}$. By assumption, $(\Sigma, \alphas,\betas,\mathbf{w})$ represents $(\sqcup_i  S^3_i, \mathbf{w}_i)$, $(\Sigma, \betas,\gammas,\mathbf{w})$ represents $(\sqcup_i (S^1\times S^2)^{\# k_i}, \mathbf{w})$, for certain integers $k_i$, and $(\Sigma, \alphas,\gammas,\ws)$ represents $Y$. If $\psi\in \pi_2(\mathbf{x}_0,\mathbf{x}_1,\mathbf{y})$ is a homology triangle  with $\mathbf{x}_1$ in the same relative grading as the top degree generator $\widehat{HF}(\Sigma, \betas,\gammas,\mathbf{w})$ we define
 \begin{equation}\wt{\gr}_{Y,\mathbb{L},\s}(\mathbf{y})=\wt{\gr}_0(\x_0)-\mu(\psi)+2n_{\mathbf{w}}(\psi)+\frac{c_1(\s)^2-2\chi(W)-3\sigma(W)}{4},\label{eq:gradingformula}\end{equation}  where $W=W(\mathbb{L})$ and $\s=\s_{\mathbf{w}}(\psi)$ is a $\spinc$ structure on $W$ with $\t=\s|_{Y}$. We then define
 \[\wt{\gr}_{Y,\mathbb{L},\s}(\y\cdot U^i)=\wt{\gr}_{Y,\mathbb{L},\s}(\y)-2i,\] for arbitrary elements of $CF^-$.
 
 The proof in \cite{HolDiskFour} goes through essentially without change to show that this formula is independent of the choice of triangle $\psi$ with $\s_{\mathbf{w}}(\psi)=\s$, as well as the bouquet of $\mathbb{L}$ or the diagram $(\Sigma, \alphas,\betas,\gammas,\mathbf{w})$ subordinate to it. Hence we get well defined absolute gradings $\wt{\gr}_{Y,\mathbb{L},\s}$ on $CF^-(Y,\mathbf{w},\t)$.
 
To show that $\wt{\gr}_{Y,\mathbb{L},\s}$ is independent of the choice of framed link $\mathbb{L}$ and $\spinc$ structure $\s$, we use the excision property of the gradings $\wt{\gr}_{Y,\mathbb{L},\s}$ (\cite[Lemma 7.6]{HolDiskFour}) to reduce the question to the case when $Y=\sqcup_i S^3_i$. If surgery on two framed links in $\sqcup_i S^3_i$ produces $\sqcup_i S^3_i$, then the links must be related by a sequence of blow-ups, blow-downs, or handleslides. Similarly the $\spinc$ structures on the 4-dimensional cobordisms must also be related by blow-ups or blow-downs. The proof from \cite{HolDiskFour} that the absolute gradings are invariant under these moves goes through without change. Hence the gradings $\wt{\gr}_{Y,\mathbb{L},\s}$ are independent of the choice of $\mathbb{L}$ and $\s$.

We now check that the absolute grading $\wt{\gr}_{\sqcup_i S_i^3}$ defined using a surgery diagram agrees with the grading $\wt{\gr}_0$ we defined in Equation~\eqref{eq:declaregradinginspheres}. We note that we can take $\mathbb{L}$ to be the empty link. Picking $(\Sigma, \alphas,\betas,\gammas,\mathbf{w})$ to be such that $\gammas$ is a small Hamiltonian isotopy of the $\betas$ curves, and  picking $\psi$ to be a small triangle homology class, we see that the gradings $\wt{\gr}_{\sqcup_i S_i^3}$ agree with the ones we declared earlier.

We finally must check the grading formula. We first observe that for a path cobordism $(W,\gammas)$ which is composed entirely of 2-handles, the formula follows identically to how it did in \cite{HolDiskFour}. Similarly 1-handles with both ends in the same component of $Y_1$, or 3-handles which are non-separating, both induce grading change $+\tfrac{1}{2}$. Noting that if $\gammas$ is a collection of paths in $W$ from $Y_1$ to $Y_2$, then $\wt{\chi}(\gammas)=0$, the grading formula is satisfied for any path cobordism $(W,\gammas)$ which can be obtained by 1-handles which don't connect distinct components, 2-handles, and non-separating 3-handles. 

We now consider 1-handles which connect distinct components of $Y_1$.  Let $\mathbb{S}_0\subset Y_1$ be a 0-sphere corresponding to a 1-handle attachment of this form. Write $Y_1$ as surgery on a framed link $\mathbb{L}\subset \sqcup_i S_i^3$, and let $\mathbb{S}_0'\subset \sqcup_i S^3_i$ be the  0-sphere in $\sqcup_i S^3_i$ that turns into $\mathbb{S}_0$ after surgery. Suppose that $(\Sigma, \alphas,\betas,\gammas,\mathbf{w})$ is a triple which is subordinate to $\mathbb{L}$,  and suppose that $\mathbb{S}_0'\subset \Sigma\setminus (\alphas\cup \betas\cup \gammas\cup \mathbf{w})$. Remove a disk from $\Sigma$ at each point in $\mathbb{S}_0'$, and connect the resulting boundary components by an annulus in the surgered manifold $(\sqcup_i S_i^3)(\mathbb{S}_0')$ to get a surface $\Sigma'$ in $(\sqcup_i S_i^3)(\mathbb{S}_0')$. Pick homologically nontrivial simple closed curves $\alpha_0,\beta_0,$ and $\gamma_0$ on the annulus on $\Sigma'$ such that all are Hamiltonian isotopies of each other, and each pair of the $\alpha_0,\beta_0$ and $\gamma_0$ curves intersect each other twice. Note that
\[(\Sigma',\alphas\cup \{\alpha_0\},\betas\cup \{\beta_0\},\gammas\cup \{\gamma_0\},\mathbf{w})\] is subordinate to a bouquet for the link $\mathbb{L}\subset (\sqcup_i S_i^3)(\mathbb{S}'_0)$, and surgery on $\mathbb{L}$ yields $Y_1(\mathbb{S}_0)=Y_2$. Let 
\[x_0^+\in \alpha_0\cap \beta_0,\qquad  x_1^+\in \beta_0\cap \gamma_0, \text{ and} \qquad y^+\in \alpha_0\cap \gamma_0\] all denote the higher graded intersection point. If $\psi\in \pi_2(\x_0,\x_1,\y)$, is a triangle with $\x_1$ in the top relative grading of $\widehat{HF}(\sqcup_i(S^1\times S^2)^{\# k_i},\mathbf{w}_i,\s_0)$, then we can form a stabilized triangle $\psi'\in\pi_2(\x_0\times x^+_0,\x_1\times x_1^+ ,\y\times y^+)$, formed by stabilizing $\psi$ with a small triangle between the points $x_0^+,x_1^+$ and $y^+$. Now using Equation \eqref{eq:gradingformula}, we have simply that
\[\wt{\gr}_{Y(\mathbb{S}_0)}(\y\times y^+)-\wt{\gr}_{Y}(\y)=\wt{\gr}_0(\x_0\times x_0^+)-\wt{\gr}_0(\x_0).\] By the definition of the absolute grading $\wt{\gr}_0$ in Equation~\eqref{eq:declaregradinginspheres}, we see that the above equation is exactly $+\frac{1}{2}$, agreeing with the grading change formula.

The maps for separating 3-handles induce grading change $+\tfrac{1}{2}$ by the exact same argument as in the previous paragraph.

Since every path cobordism can be decomposed into 1-handle attachments, followed by 2-handle attachments, followed by 3-handle attachments (this follows from the proof of \cite[Theorem 8.9]{JuhaszSutured}, that every special cobordism has a nice Morse function) the maps $F_{W,\gammas,\s}^\circ$ have the stated grading change formula.
\end{proof}

\begin{lemma}\label{lem:gradingchangeflowgraphs}If $\Gamma \co V_0\to V_1$ is a flow graph in a three-manifold $Y$, as in Section~\ref{sec:GraphAction}, then the graph action map
\[A_\Gamma \co CF^\circ(Y,V_0,\s)\to CF^\circ(Y,V_1,\s),\] has induced grading change $\wt{\chi}(\Gamma)$.
\end{lemma}
\begin{proof}We first show that $S_{w}^{+}$ and $S_{w}^-$ both induce grading changes of $+\tfrac{1}{2}$ and $A_\lambda$ induces a grading change $-1$.

To see that $S_w^+$ and $S_w^-$ both induce grading change $+\tfrac{1}{2}$, one adapts the argument in Lemma~\ref{lem:pathcobordismgradingchange} that 1-handles and 3-handles which join or separate components of $Y$ induce grading change $+\tfrac{1}{2}$. If $(Y,\ws)$ is a 3-manifold with $w\not\in \ws$, one first picks a link $\mathbb{L}$ and a diagram $(\Sigma, \alphas,\betas,\gammas,\ws)$ which is subordinate to a bouquet for $\mathbb{L}$ such that surgery on $\mathbb{L}$ produces $Y$, which also has $w\in \Sigma\setminus (\alphas\cup\betas\cup\gammas\cup\ws)$. One then performs free-stabilization on the triple $(\Sigma,\alphas,\betas,\gammas,\ws)$ to get a triple which can be used to compute the absolute grading for $(Y,\ws\cup \{w\})$. Adapting the argument for 1-handles and 3-handles which connect or separate (respectively) components of $Y$, our declaration of the grading on disjoint unions of spheres in Equation~\eqref{eq:declaregradinginspheres} immediately implies that $S_{w}^+$ and $S_{w}^-$ both have grading change $+\tfrac{1}{2}$.

The relative homology maps change degree by $-1$ since they count Maslov index 1 disks.

We now check the grading formula for graph action maps. Suppose that $\Gamma \co V_0\to V_1$ is a flow graph. Note that $\wt{\chi}(\Gamma)$ and $A_\Gamma$ are both invariant under subdivision of $\Gamma$. Given a Cerf decomposition of $\Gamma$, the graph action map $A_\Gamma$ is a composition of free-stabilization maps, as well as relative homology maps. In the composition, there is one relative homology map for each edge of $\Gamma$, hence the relative homology maps contribute total grading change of $-|E(\Gamma)|$. For each vertex $w\in V(\Gamma)\setminus (V_0\cup V_1)$, there is an $S_w^+$ and an $S_w^-$, each of which has grading change $+\tfrac{1}{2}$. Each vertex in $w\in V_0$ contributes one $S_w^-$ term, and each vertex $w\in V_1$ contributes one $S_w^+$ term. Hence the total grading change from all vertices and edges is
\[|V(\Gamma)|-\tfrac{1}{2}(|V_0|+|V_1|)-|E(\Gamma)|,\] which is exactly $\wt{\chi}(\Gamma)$.\end{proof}
\begin{lemma}\label{lem:0-,4-handlegradingchange}The 0- and 4-handle maps defined in \cite[Section 11]{Zemke2} are both zero graded.
\end{lemma}
\begin{proof} This follows from an adaptation of our argument from Lemma~\ref{lem:pathcobordismgradingchange} that 1-handles and 3-handles which connect or separate components induce maps which are $+\tfrac{1}{2}$ graded, and our declaration of the absolute grading for disjoint unions of spheres in Equation~\eqref{eq:declaregradinginspheres}.
\end{proof}

Having computed the grading change due to all maps involved, we can now verify the full grading change formula for the graph cobordism maps:

\begin{proof}[Proof of Proposition~\ref{prop:gradingchangeformula}] The absolute gradings were defined in Lemma~\ref{lem:pathcobordismgradingchange}, where we proved the grading change formula for path cobordisms. We now aim to extend it to all graph cobordisms.

Applying \cite[Lemma 12.6]{Zemke2}, and \cite[Lemma 13.1]{Zemke2}, we can decompose an arbitrary graph cobordism as 
\[(W,\Gamma)=(W_4,\Gamma_4)\circ (W_{\gammas},\gammas)\circ (W_\Gamma,\Gamma_N)\circ (W_1,{\gammas}_1)\circ (W_0,\Gamma_0),\] such that the following hold:

\begin{itemize}
\item $(W_4,\Gamma_4)$ and $(W_0,\Gamma_0)$ are 4-handle and 0-handle cobordisms respectively, and $\Gamma_i$ consists of one arc in $B^4$, connecting a vertex on the interior of $B^4$ to a vertex on the boundary;
\item $(W_{\gammas}, \gammas)$ is a path cobordism;
\item $(W_\Gamma,\Gamma_N)$ is a graph cobordism with $W_\Gamma\cong N\times I$, for a 3-manifold $N$;
\item $(W_1,{\gammas}_1)$ is a path cobordism obtained by adding only 1-handles.
\end{itemize}

Applying the composition law for graph cobordisms  \cite[Theorem E]{Zemke2}, the map $F_{W,\Gamma,\s}$ is the composition of the graph cobordism maps for each of the cobordisms listed above.

 It is thus sufficient to verify that both sides of Equation \eqref{eq:maingradingformula} are additive under composition, and that the formula is true for the cobordisms listed in the previous paragraph. Additivity of the left side of Equation \eqref{eq:maingradingformula} is trivial. Additivity of 
\[\frac{c_1(\s)^2-2\chi(W)-3\sigma(W)}{4}\] is straightforward. Additivity of of $\wt{\chi}(\Gamma)$ under concatenation of flow graphs can be seen as follows: If $\Gamma \co V_0\to V_1$ and $\Gamma' \co V_1\to V_2$ are two flowgraphs, then letting $\Gamma'\circ \Gamma$ denote their concatenation, we have
\[\wt{\chi}(\Gamma)+\wt{\chi}(\Gamma')=\chi(\Gamma)+\chi(\Gamma')-\tfrac{1}{2}(|V_0|+2|V_1|+|V_2|)=\chi(\Gamma)+\chi(\Gamma')-|V_1|-\tfrac{1}{2}(|V_0|+|V_2|)\]
\[=\chi(\Gamma\circ \Gamma')-\tfrac{1}{2}(|V_0|+|V_2|)=\wt{\chi}(\Gamma'\circ \Gamma).\]

Note that by \cite[Lemma 12.3]{Zemke2}, the map associated to $(W_\Gamma,\Gamma_N)$ is just the graph action map for the graph $\Gamma_N$, projected into the ends of $W_\Gamma\cong N\times I$. Hence by Lemma \ref{lem:gradingchangeflowgraphs}, the grading formula holds for $(W_\Gamma, \Gamma_N)$.

Hence it remains to verify Equation \eqref{eq:maingradingformula} for the following cobordisms: 0-handle cobordisms, 4-handle cobordisms, path cobordisms formed with 1-handles, flow graphs, and path cobordism maps. We have already verified the formula for path cobordism maps in Lemma~\ref{lem:pathcobordismgradingchange} (and hence 1-handle maps, as a special case). Similarly we checked the grading change formula for flow graphs in Lemma \ref{lem:gradingchangeflowgraphs}.

The final step is to check that Equation \eqref{eq:maingradingformula} is true for 0- and 4-handle cobordisms. According to Lemma~\ref{lem:0-,4-handlegradingchange}, the maps associated to these are zero graded. So we need to show that the right hand side of Equation \eqref{eq:maingradingformula} is zero for these cobordisms. Topologically, these both increase $\chi(W)$ by $1$, but don't affect $c_1(\s)$ or $\sigma(W)$. Hence we need to show that for the graph $\Gamma$ associated to a 0-handle or a 4-handle cobordism, we have $\wt{\chi}(\Gamma)=\tfrac{1}{2}$. The 0-handle and 4-handle cobordisms are 4-balls with a graph which has one vertex on the interior, and a single edge connecting it to the boundary 3-sphere, which has a single vertex on it. In particular, we have for both a 0-handle and a 4-handle that \[|V_0|+|V_1|=V(\Gamma)=1\] and hence
$\wt{\chi}(\Gamma)=\tfrac{1}{2},$ completing the proof that the grading change due to a 0-handle or 4-handle map agrees with Equation \eqref{eq:maingradingformula}.

Since the graph cobordisms are a composition of the above maps, and each of the constituent maps satisfies the grading formula in \eqref{eq:maingradingformula}, and \eqref{eq:maingradingformula} is additive under composition, the grading change formula now follows for arbitrary graph cobordisms $(W,\Gamma)$.
 \end{proof}
 
\begin{remark} \label{rem:absolute6} We have adopted the convention of \cite{HolDiskFour} that $HF^-(S^3,w,\s_0)$ has maximally graded element in grading $-2$. As such, the canonical isomorphism \eqref{eq:disjunion} between $CF^-(Y_1\sqcup Y_2,w_1,w_2,\s_1\sqcup \s_2)$ and $CF^-(Y_1,w_1,\s_1)\otimes_{\Z_2[U]} CF^-(Y_2,w_2,\s_2)$ is not grading preserving, and instead we have
\[CF^-(Y_1\sqcup Y_2,w_1,w_2,\s_1\sqcup \s_2)\cong CF^-(Y_1,w_1,\s_1)\otimes_{\Z_2[U]} CF^-(Y_2,w_2,\s_2)[-2].\]
\end{remark}

\begin{remark}
Proposition~\ref{prop:gradingchangeformula} is phrased for the setting considered in this paper, with a single $U$ variable for all basepoints. However, the proposition also holds for arbitrary colorings of the basepoints; the proof goes through without change.
\end{remark}

\section{A first version of the connected sum formula}
\label{sec:proof}

Building up to Theorem~\ref{thm:ConnSum}, in this section we give a preliminary identification of the conjugation involution on the connected sum:

\begin{proposition} 
\label{prop:ConnSum}
Suppose $Y_1$ and $Y_2$ are three-manifolds equipped with spin structures $\s_1$ and $\s_2$. Let $\iota_1$ and $\iota_2$ denote the conjugation involutions on the Floer complexes $\CFm(Y_1, \s_1)$ and $\CFm(Y_2,\s_2)$. For $i=1,2$, let also $\Phi_i$ denote the formal derivative of the differential on $\CFm(Y_i, \s_i)$ with respect to the $U$ variable, as in Section~\ref{sec:Phi}. Then, there is an equivalence
$$\CFm(Y_1 \# Y_2, \s_1 \# \s_2) \simeq \CFm(Y_1,\s_1) \otimes_{\Z_2[U]} \CFm(Y_2, \s_2)[-2],$$
with respect to which the conjugation involution $\iota$ on $\CFm(Y_1 \# Y_2,$ $ \s_1 \# \s_2)$ corresponds to  
$$\iota_1 \otimes \iota_2 + U(\Phi_1 \iota_1 \otimes \Phi_2 \iota_2).$$ 
\end{proposition}

There are two steps in the proof of Proposition~\ref{prop:ConnSum}. The first is to give an alternate proof of the K\"{u}nneth theorem for ordinary Heegaard Floer homology. We describe chain homotopy equivalences between $CF^\circ(Y_1\sqcup Y_2, w_1,w_2,\mathfrak{s}_1\sqcup \mathfrak{s}_2)$ and $CF^\circ(Y_1\# Y_2, w, \mathfrak{s}_1\# \mathfrak{s}_2)$, which are defined as cobordism maps for 1-handle or 3-handle cobordisms with a trivalent graph connecting the basepoints in the three ends, using the graph TQFT constructed by the third author in \cite{Zemke2}. The second step is to consider how the involution interacts with the chain homotopy equivalences between $CF^\circ(Y_1\sqcup Y_2, w_1,w_2,\mathfrak{s}_1\sqcup \mathfrak{s}_2)$ and $CF^\circ(Y_1\# Y_2, w, \mathfrak{s}_1\# \mathfrak{s}_2)$.

We note that the strategy of using maps induced by 1-handle cobordisms to prove connected sum formulas has been used in other contexts. An early example can be found in \cite{FukayaConnSums}, in the setting of instanton homology. Such cobordisms have also been considered in the context of monopole Floer homology in \cite{KLT5} and \cite{LinConnSums}.

It is also worthwhile to compare our graph cobordism maps to the maps  considered  in \cite{HolDiskTwo} for the original proof of the connected sum formula for Heegaard Floer homology. Ozsv\'{a}th and Szab\'{o} construct a map from $CF^-(Y_1,\s_1)\otimes_{\Z_2[U]} CF^-(Y_2,\s_2)$ to $CF^-(Y_1\# Y_2,\s_1\# \s_2)$ as a composition of many 1--handle maps and a single triangle map. Firstly, it is not immediately obvious how to interpret their map in terms of 4--dimensional cobordisms. Also, there is an asymmetry between the roles of $Y_1$ and $Y_2$ in the construction of their map, in the ordering that they appear in the triangle map. Indeed the map $\iota$ can be seen to switch the roles of $Y_1$ and $Y_2$ in their connected sum map. In parallel, there are two choices of cyclic ordering at the vertex in the natural graph cobordism for the connected sum cobordism. Paralleling the interaction of $\iota$ with the Ozsv\'{a}th-Szab\'{o} connected sum map, we will see that $\iota$ interacts with the graph cobordism maps by switching the cyclic ordering at the trivalent vertex.

\subsection{A new proof of the connected sum formula for ordinary Heegaard Floer homology}
In \cite{HolDiskTwo}, Ozsv\'{a}th and Szab\'{o} prove that 
\begin{equation}
\label{eq:OSconnsum}
HF^-(Y_1\# Y_2,\mathfrak{s}_1\# \mathfrak{s}_2)\cong H_*(CF^-(Y_1,\mathfrak{s}_1)\otimes_{\Z_2[U]} CF^-(Y_2,\mathfrak{s}_2)).\end{equation} They describe a map between the two complexes, constructed as a composition of 1-handle maps, and a triangle map, which they show to be a quasi-isomorphism. In this section, we provide an alternate proof of this fact, using graph cobordism maps. Our proof will be a slight improvement, since we will also exhibit a $\Z_2[U]$-equivariant homotopy inverse instead of only proving that our maps are quasi-isomorphisms. We consider graph cobordism maps $F^A$ and $G^A$ associated to the graph cobordisms shown in Figure~\ref{fig::3}. The underlying 4-manifold in both cases is simply a 1-handle joining the two components, or a 3-handle splitting them. These are shown in Figure~\ref{fig::3}.

\begin{figure}[ht!]
\centering
\includegraphics[scale=1.2]{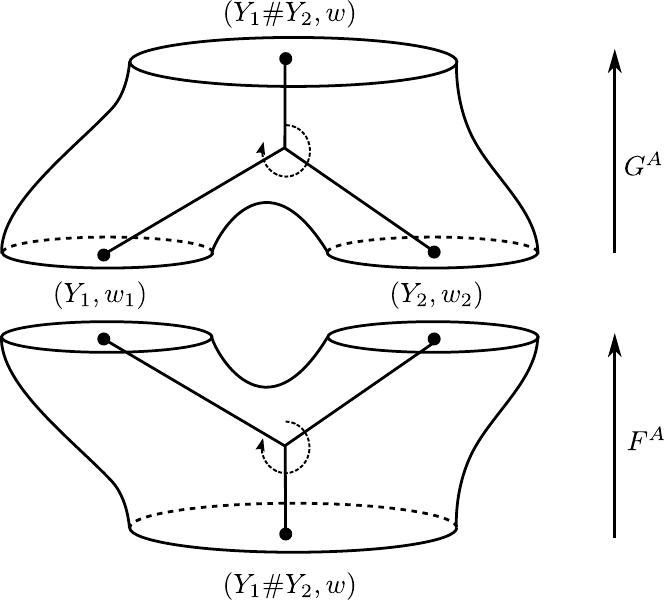}
\caption{The maps $F^A$ and $G^A$ are defined by the 1-handle or 3-handle cobordisms shown, with embedded trivalent graph with given cyclic ordering.\label{fig::3}}
\end{figure}

\begin{proposition}\label{prop:maincomputation}The graph cobordism maps 
\[F^A \co CF^\circ(Y_1\# Y_2, w,\mathfrak{s}_1\# \mathfrak{s}_2)\to CF^\circ(Y_1\sqcup Y_2, w_1,w_2,\mathfrak{s}_1\sqcup \mathfrak{s}_2) \] and 
\[G^A \co CF^\circ(Y_1\sqcup Y_2, w_1,w_2,\mathfrak{s}_1\sqcup \mathfrak{s}_2)\to CF^\circ(Y_1\# Y_2, w,\mathfrak{s}_1\# \mathfrak{s}_2)\] for the graph cobordisms shown in Figure~\ref{fig::3} are $\Z_2[U]$-equivariant homotopy inverses of each other. \end{proposition}

Proposition~\ref{prop:maincomputation} readily implies \eqref{eq:OSconnsum}, in view of the formula \eqref{eq:disjunion} for disjoint unions.

\begin{remark}Proposition~\ref{prop:maincomputation} holds for all flavors of Heegaard Floer homology, whereas the K\"{u}nneth theorem proven by Ozsv\'{a}th and Szab\'{o} is stated only for the $-,\infty$ and hat flavors. One should note that, of course, $CF^+(Y_1\sqcup Y_2,w_1,w_2,\mathfrak{s}_1\sqcup \mathfrak{s}_2)$ is not the tensor product of $CF^+(Y_1,w_1,\mathfrak{s}_1)$ and $CF^+(Y_2,w_2,\mathfrak{s}_2)$.
\end{remark}

The proof of Proposition~\ref{prop:maincomputation} will be to explicitly compute $F^A\circ G^A$ and $G^A\circ F^A$ by computing the cobordism map associated to the composition of the two cobordisms.

A trick that we will repeatedly use is the effect of attempting to slide an edge of $\Gamma$ across another edge of $\Gamma$. To this effect, we have the following.

\begin{lemma}\label{lem:replacetwoarcswithconcatenation}Suppose that $w,w'$ and $w''$ are distinct basepoints in $Y^3$, and that $e_1$ is an edge from $w$ to $w'$ and $e_2$ is an edge from $w'$ to $w''$. Let $e_1* e_2$ denote the concatenation of $e_1$ and $e_2$. Then, we have that
\[A_{e_1}A_{e_2}=A_{e_1*e_2}A_{e_2}+U=A_{e_1}A_{e_1*e_2}+U,\] where $*$ denotes concatenation of arcs.
\end{lemma}
\begin{proof} By construction we have $A_{e_1*e_2}=A_{e_1}+A_{e_2}$. By \eqref{eq:R2} we have 
\[A_{e_1}A_{e_1}\simeq A_{e_2}A_{e_2}\simeq U.\] Hence we have 
\[A_{e_1*e_2}A_{e_2}=(A_{e_1}+A_{e_2})A_{e_2}=A_{e_1}A_{e_2}+A_{e_2}^2=A_{e_1}A_{e_2}+U.\] The first equality of the lemma now follows. The second follows in an identical manner.
\end{proof}

We need to perform two computations. Specifically, we need to come up with a convenient expression for the graph action maps (as in Section~\ref{sec:GraphAction}) induced by the two flowgraphs $\mathcal{G}_F$ and $\mathcal{G}_G$ shown in Figure~\ref{fig::8}.

\begin{figure}[ht!]
\centering
\includegraphics[scale=1.2]{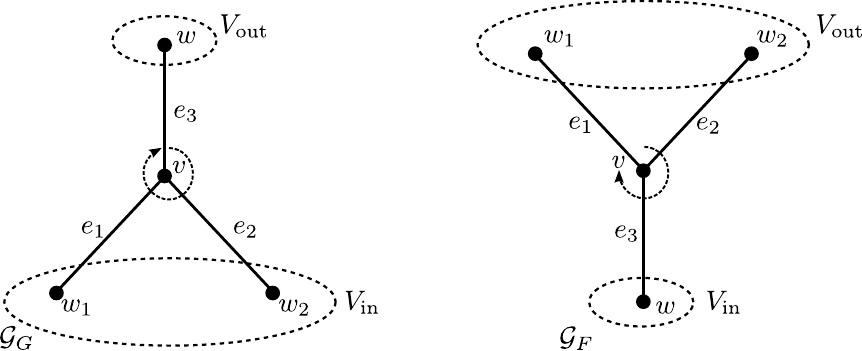}
\caption{Two flowgraphs $\mathcal{G}_G$ and $\mathcal{G}_F$, with vertices and edges labelled, which we consider in Lemmas~\ref{lem:graphactioncomp1} and \ref{lem:graphactioncomp2}.\label{fig::8}}
\end{figure}

\begin{lemma}\label{lem:graphactioncomp1}Suppose $\mathcal{G}_G=(\Gamma, V_{\text{in}}=\{w_1,w_2\}, V_{\text{out}}=\{w\})$ is the cyclically ordered flow graph shown in Figure~\ref{fig::8}, embedded in a three-manifold $Y$. Let $\phi$ be the diffeomorphism of $Y$ obtained by moving the basepoint $w_1$ to $w$ along the path $e_3*e_1$. Then
\[A_{\mathcal{G}_G}\simeq \phi_* S^-_{w_2}A_{e_2*e_1}+U\phi_*\Phi_{w_1} S^-_{w_2}.\]
\end{lemma}

\begin{proof}In the terminology of \cite[Section 7]{Zemke2}, the graph $\mathcal{G}_C$ is an elementary strong ribbon flow graph of type 2. As such, the corresponding graph action map is obtained as the composition of positive free-stabilizations at the outgoing vertices and interior vertices, followed by the relative homology maps for the edges (in the cyclic order specified at the interior vertex), followed by negative free-stabilizations at the incoming vertices and interior vertices:
\[A_{\mathcal{G}_G}\simeq S^-_{w_1w_2v} A_{e_3}A_{e_1}A_{e_2} S^+_{vw}.\]
 We perform the following manipulation:
 \begin{align*}A_{\mathcal{G}_G}&\simeq S^-_{w_1w_2v} A_{e_3}A_{e_1}A_{e_2} S^+_{vw}&&\\
 &\simeq S^-_{w_1w_2v} A_{e_3}A_{e_1}A_{e_2*e_1} S^+_{vw}+U S_{w_1w_2v}^- A_{e_3} S^+_{vw}&& (\text{Lem.  \ref{lem:replacetwoarcswithconcatenation}})\\
 &\simeq (S^-_{w_1v} A_{e_3}A_{e_1} S^+_{vw}) (S^-_{w_2}A_{e_2*e_1})+U S_{w_1w_2}^- S^+_{w}&&\eqref{eq:R11}, \eqref{eq:R7}\\
 &\simeq (S^-_{v} A_{e_3}S^+_w S^-_{w_1}A_{e_1} S^+_{v}) (S^-_{w_2}A_{e_2*e_1})+U S_{w_1w_2}^- S^+_{w}&&\eqref{eq:R9}, \eqref{eq:R11}\\
 &\simeq \phi_* S^-_{w_2}A_{e_2*e_1}+U S^-_{w_1}S^+_{w}S_{w_2}^-&&\eqref{eq:R8}, \eqref{eq:R9}\\
 &\simeq \phi_* S^-_{w_2}A_{e_2*e_1}+U(S^-_{w_1} A_{e_3* e_1} S^+_{w_1}) S^-_{w_1}S^+_{w}S_{w_2}^-&&\eqref{eq:R7}\\
 &\simeq \phi_* S^-_{w_2}A_{e_2*e_1}+US^-_{w_1} A_{e_3* e_1} \Phi_{w_1} S^+_w S_{w_2}^-&&\eqref{eq:R6}\\
 &\simeq \phi_* S^-_{w_2}A_{e_2*e_1}+U(S^-_{w_1} A_{e_3* e_1} S^+_w)\Phi_{w_1}  S_{w_2}^-&&\eqref{eq:R10}\\
 &\simeq \phi_* S^-_{w_2}A_{e_2*e_1}+U\phi_*\Phi_{w_1}  S_{w_2}^-,&&\eqref{eq:R8}
 \end{align*}
as we wanted.

\end{proof}

\begin{lemma}\label{lem:graphactioncomp2} Suppose $\mathcal{G}_F=(\Gamma, V_{\text{in}}=\{w\},V_{\text{out}}=\{w_1,w_2\})$ is the flowgraph shown in Figure~\ref{fig::8}, embedded in a three-manifold $Y$. Let $\phi$ be the diffeomorphism of $Y$ obtained by moving $w$ to $w_1$ along the path $e_1* e_3$. Then
\[A_{\mathcal{G}_F}\simeq A_{e_2* e_1} S_{w_2}^+ \phi_*+U S_{w_2}^+\Phi_{w_1}\phi_*.\]
\end{lemma}

\begin{proof}The proof is analogous to the proof of Lemma~\ref{lem:graphactioncomp1}. By definition, we have
\[A_{\mathcal{G}_F}\simeq S^-_{wv} A_{e_2} A_{e_1} A_{e_3} S^+_{w_1w_2v}.\] Performing the analogous manipulations yields
\begin{align*}
A_{\mathcal{G}_F}&\simeq S^-_{wv} A_{e_2} A_{e_1} A_{e_3} S^+_{w_1w_2v}&&\\
&\simeq S_{wv}^- A_{e_2* e_1} A_{e_1} A_{e_3} S_{w_1w_2v}^++U S_{wv}^- A_{e_3} S_{w_1w_2v}^+&&(\text{Lem. } \ref{lem:replacetwoarcswithconcatenation})\\
&\simeq A_{e_2* e_1}(S_{wv}^- A_{e_1} A_{e_3} S_{w_1v}^+)S_{w_2}^++U S_{w}^-S_{w_1w_2}^+&&\eqref{eq:R11},\eqref{eq:R7}\\
&\simeq A_{e_2* e_1}(S_{v}^- A_{e_1} S^+_{w_1}S^-_wA_{e_3} S_{v}^+)S_{w_2}^++U S_{w}^-S_{w_1w_2}^+&&\eqref{eq:R9},\eqref{eq:R11}\\
&\simeq A_{e_2* e_1} \phi_* S_{w_2}^++U S_{w}^-S_{w_1}^+S^+_{w_2}&&\eqref{eq:R8}, \eqref{eq:R9}\\
&\simeq A_{e_2* e_1}  S_{w_2}^+\phi_*+U S_{w}^-S_{w_1}^+S_{w_2}^+&&\\
&\simeq A_{e_2* e_1}  S_{w_2}^+\phi_*+U S_{w}^-S_{w_1}^+(S_{w_1}^- A_{e_3* e_1} S_{w_1}^+)S_{w_2}^+&& \eqref{eq:R7}\\
&\simeq A_{e_2* e_1}  S_{w_2}^+\phi_*+U S_{w}^-\Phi_{w_1} A_{e_3* e_1} S_{w_1}^+S_{w_2}^+&&\eqref{eq:R6}\\
&\simeq A_{e_2* e_1}  S_{w_2}^+\phi_*+U \Phi_{w_1} (S_{w}^-A_{e_3* e_1} S_{w_1}^+)S_{w_2}^+&&\eqref{eq:R10}\\
&\simeq A_{e_2* e_1}  S_{w_2}^+\phi_*+U \Phi_{w_1} \phi_* S_{w_2}^+&&\eqref{eq:R8}\\
&\simeq  A_{e_2* e_1}  S_{w_2}^+\phi_*+U \Phi_{w_1}  S_{w_2}^+\phi_*,&&
\end{align*}
as we wanted.
\end{proof}

As a final step towards proving Proposition~\ref{prop:maincomputation}, we perform a computation of the disks counted by the differential on a diagram for $(Y_1\# Y_2, w_1,w_2)$. Suppose that $(\Sigma_1,\alphas_1,\betas_1,w_1)$ and $(\Sigma_2,\alphas_2,\betas_2,w_2)$ are diagrams for $(Y_1,w_1)$ and $(Y_2,w_2)$ respectively. We form the connected sum $\Sigma_1\# S^2\# \Sigma_2$ with curves $\alphas_1\cup \{\alpha_0\}\cup \alphas_2$ and $\betas_1\cup \{\beta_0\}\cup \betas_2$, where $(S^2,\alpha_0,\beta_0,p_1,p_2)$ is the diagram for $(S^3,p_1,p_2)$ shown in Figure~\ref{fig::13}. 

The following lemma is phrased in terms of the ``cylindrical reformulation'' of Heegaard Floer homology developed by Lipshitz in \cite{LipshitzCyl}. In \cite[Section~13]{LipshitzCyl}, Lipshitz shows that if $(\Sigma,\mathbf{\alpha},\mathbf{\beta},w)$ is a singly pointed Heegaard diagram, then for appropriate choices of almost complex structures, there are isomorphisms between the moduli spaces of holomorphic disks mapping into $\Sym^g(\Sigma)$ with boundary values on the tori $\mathbb{T}_\alpha$ and $\mathbb{T}_\beta$, originally considered by Ozsv\'{a}th and Szab\'{o}, and certain moduli spaces of holomorphic curves mapping into $\Sigma\times [0,1]\times \R$, with boundary values on $\mathbf{\alphas}\times \{1\}\times \R$ and $\mathbf{\betas}\times \{0\}\times \R$. A similar correspondence holds for multi-pointed Heegaard diagrams.

\begin{figure}[ht!]
\centering
\includegraphics[scale=1.3]{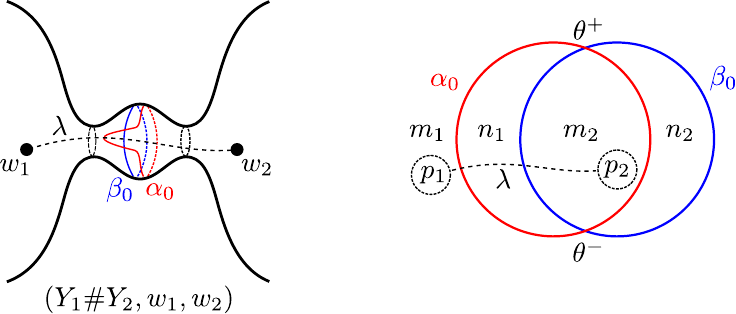}
\caption{On the left is a diagram $\widehat{\mathcal{H}}$ for $(Y_1\# Y_2, w_1,w_2)$ resulting from adding a one handle to $Y_1\sqcup Y_2$ near $w_1$ and $w_2$. On the right is a diagram for $(S^3, p_1,p_2)$, used for the 1-handle. The points $p_1$ and $p_2$ correspond to where the connected sum is taken.\label{fig::13}}
\end{figure}

\begin{lemma}\label{lem:diffmodelcomp}Suppose that $\mathcal{H}$ is a diagram for $(Y_1\sqcup Y_2,w_1,w_2)$ and $\widehat{\mathcal{H}}$ is the diagram formed by adding a 1-handle as in Figure~\ref{fig::13}. As Abelian groups, there is a decomposition 
\begin{equation}
\label{eq:thetas}
\CFm(\widehat{\mathcal{H}}) \cong \CFm(\mathcal{H})\langle \theta^+ \rangle \oplus  \CFm(\mathcal{H})\langle \theta^- \rangle.
\end{equation}
Then, for an almost complex structure on $(\Sigma_1 \# S^2\# \Sigma_2)\times [0,1]\times \R$ which is sufficiently stretched along the dashed circles in Figure~\ref{fig::13}, and with connected sum points $p_1,p_2\in S^2$ sufficiently close to $\alpha_0\setminus \beta_0$, there are identifications
\[\partial_{\widehat{\mathcal{H}}}=\begin{pmatrix} \partial_{\mathcal{H}}& U_{w_1}+U_{w_2}\\
0& \partial_{\mathcal{H}}
\end{pmatrix}\] and
\[A_\lambda=\begin{pmatrix}0 & U_{w_1}\\
1& 0
\end{pmatrix},\]
with respect to the decomposition \eqref{eq:thetas}. 
\end{lemma}

\begin{proof}Our proof involves explicitly counting holomorphic disks. Detailed proofs of similar results can be found in \cite{Links} and \cite{Zemke2}, so we only highlight the differences. If $\phi=\phi_1\# \phi_{S^2}\in \pi_2(\mathbf{x}\times x,\mathbf{y}\times y)$ is a relative homology class of disks on $\widehat{\mathcal{H}}$, with $\phi_1$ a class of disks on $\mathcal{H}$ and $\phi_{S^2}$ a class of disks on $(S^2,\alpha_0,\beta_0,p_1,p_2)$ then using \cite[Lemma 6.3]{Zemke2} as well as the combinatorial formula for the Maslov index, the Maslov index can be calculated as
\[\mu(\phi)=\mu(\phi_1)+\Delta(x,y),\] where $\Delta(x,y)$ denotes the drop in grading between $x$ and $y$ (here $x,y\in \alpha_0\cap \beta_0$). 

For disks with $\Delta(x,y)=1$, i.e. such that $x = \theta^+$ and $y=\theta^-$, by stretching the neck and extracting a weak limit to broken disks on $\mathcal{H}$ and broken disks on $(S^2,\alpha_0,\beta_0)$, we see that $\mu(\phi_1)=0$, and hence $\phi_0$ is a constant disk by transversality. There are two such disks which are counted, i.e. the ones with $n_1=1$ and $n_2=0$, or $n_2=1$ and $n_1=0$ (with $n_1,n_2,m_1$ and $m_2$ the labelling multiplicities in Figure~\ref{fig::13}). These contribute zero to the lower left entry of $\partial_{\widehat{\mathcal{H}}}$ (since they cancel, modulo 2), and 1 to the lower left entry of $A_{\lambda}$ (since only one is counted by $A_{\lambda}$).

We now consider the diagonal entries of $\partial_{\widehat{\mathcal{H}}}$, i.e. disks with $\Delta(x,y)=0$, so that $x = y=\theta^{\pm}$. For such disks, by stretching the neck and extracting a weak limit, we are forced to have $\mu(\phi_1)=1$, and hence $\phi_1$ must be represented by a genuine Maslov index one holomorphic strip. Further, since $\phi_1$ is a holomorphic strip on $(\Sigma_1\sqcup \Sigma_2)\times [0,1]\times \R$, we know that it can only be nonconstant in exactly one of the $\Sigma_i\times [0,1]\times \R$, because of the Maslov index. For the sake of demonstration, assume that $\phi_1$ is constant on $\Sigma_2\times [0,1]\times \R$ (so in particular $n_{w_2}(\phi_1)=0)$). The homology classes of disks which can have representatives along the diagonal components of $\partial_{\widehat{\mathcal{H}}}$ are thus of the form $\phi_1\#(n_1\cdot A+n_2 \cdot B)$, where $A$ is the Maslov index 2 $\alpha$-boundary degeneration with $n_{p_2}(A)=0$ on the diagram in Figure~\ref{fig::13}, and $B$ is the Maslov index 2 $\beta$-boundary degeneration with $n_{p_2}(B)=0$, and here $\phi_1$ is a homology class of disks on $(\Sigma_1\sqcup \Sigma_2)\times [0,1]\times \R$ which is supported on $\Sigma_1\times [0,1]\times \R$. Here also we have $n_1+n_2=m_1$. By a Gromov compactness argument (cf. \cite[Lemma 6.4]{Links}), the total number of holomorphic representatives over all disks of the form $\phi_1\#(n_1\cdot A+n_2 \cdot B)$ is equal (modulo 2) to the number of holomorphic representatives of $\phi_1$. Since the same holds in the case that $\phi_1$ is constant on $\Sigma_1$ instead of $\Sigma_2$, we thus know that the diagonal entries of $\partial_{\widehat{\mathcal{H}}}$ are just $\partial_{\mathcal{H}}$. By the argument from \cite[Lemma 14.3]{Zemke2}, by sending the connected sum points $p_1$ and $p_2$ towards points on $\alpha_0\setminus \beta_0$, we can ensure that the only holomorphic disks which have representatives are those of the form $\phi_1\# (m_1\cdot B+0\cdot A)$. Such disks contribute zero to the diagonal entries of $A_\lambda$ since they have zero change over the curve $\alpha_0$.

We now just need to compute the upper right entry of $\partial_{\widehat{\mathcal{H}}}$ and $A_{\lambda}$. Such disks have $\Delta(x,y)=-1$, so that $\x =\theta^-$ and $\y=\theta^+$. Thus, we are considering disks classes $\phi_1$ on $(\Sigma_1\sqcup \Sigma_2)\times [0,1]\times \R$ with $\mu(\phi_1)=2$. We first claim that as we degenerate the almost complex structure by sending $p_1$ and $p_2$ towards points on $\alpha_0\setminus\beta_0$, any limiting curves on $(\Sigma_1\sqcup \Sigma_2)\times[0,1]\times \R$ can only be supported on one of $\Sigma_1\times [0,1]\times \R$ or $\Sigma_2\times [0,1]\times \R$. If limiting curves existed which were supported on both, then by transversality, both must be Maslov index 1 disks. Since there are only finitely many disks on either $\Sigma_i\times [0,1]\times \R$ of Maslov index one, and the broken limits on $S^2\times [0,1]\times \R$ must satisfy a matching condition with respect to the disks on $(\Sigma_1\sqcup \Sigma_2)\times [0,1]\times \R$, by sending $p_1$ and $p_2$ to points on $\alpha_0\setminus \beta_0$, and applying an open mapping principle as in  \cite[Lemma 14.3]{Zemke2}, we can ensure that $n_1\ge m_2$ and $n_2\ge m_1$. On the other hand, one can compute that
\[\Delta(x,y)=n_1+n_2-m_1-m_2,\] which is $-1$ by assumption, thus resulting in a contradiction. Hence by degenerating the connected sum points, we can assume that $\phi_1$ is supported on exactly one of $\Sigma_1\times [0,1]\times \R$ or $\Sigma_2\times [0,1]\times \R$, and has Maslov index 2. Such disks are counted in \cite[Proposition 6.5]{Links}, in the context of invariance of Heegaard Floer homology under adding basepoints. For almost complex structures sufficiently stretched, with connected sum points sufficiently close to points on $\alpha_0\setminus \beta_0$, there are two possibilities for such classes which have holomorphic representatives. One possibility is for $m_1=1$ and $n_1=n_2=m_2=0$, and the other possibility is for $m_2=1$ and $n_1=n_2=m_1=0$. Each possibility admits one (modulo 2) representative, and the contribution to $\partial_{\widehat{\mathcal{H}}}$ is $U_{w_1}+U_{w_2}$ in the upper right corner, and the contribution to $A_{\lambda}$ is $U_{w_1}$ in the upper right corner.
\end{proof}

We can now give a proof of Proposition~\ref{prop:maincomputation}, which implies the K\"{u}nneth formula \eqref{eq:OSconnsum}.

\begin{proof}[Proof of Proposition~\ref{prop:maincomputation}] We will compute directly that $F^A\circ G^A\simeq \id$ and $G^A\circ F^A\simeq \id$.

We first compute $G^A\circ F^A$. Using the composition law for graph cobordisms, as well as the fact that the graph cobordism maps for 4-manifolds of the form $W=Y\times I$ with embedded graph are just the graph action maps (\cite[Theorem E]{Zemke2}, and \cite[Lemma 12.3]{Zemke2}, resp.), we observe that the composition $G^A\circ F^A$ can be written as a graph action map for a graph like $\mathcal{G}_F$, shown in Figure~\ref{fig::8}, followed by a 3-handle map, then a 1-handle map, and finally followed by a graph action map for a graph like $\mathcal{G}_G$, shown in Figure~\ref{fig::8}.  The graph action maps of the two relevant trivalent graphs were computed in Lemmas~\ref{lem:graphactioncomp1} and \ref{lem:graphactioncomp2}.

Notice that the formulas for $A_{\mathcal{G}_F}$ and $A_{\mathcal{G}_G}$ both involve a diffeomorphism $\phi$, which sends a basepoint $w$ to another basepoint $w_1$, or vice-versa. This diffeomorphism can be absorbed into the 1-handle or 3-handle cobordism map, and upon identifying $w$ with $w_1$, we thus have that
\[G^A\circ F^A=(S_{w_2}^- A_{\lambda}+U\Phi_{w_1}S_{w_2}^-)F_1 G_3(A_{\lambda} S_{w_2}^++U S_{w_2}^+\Phi_{w_1}),\] where $G_3$ denotes the 3-handle map, $F_1$ denotes the 1-handle map, and $\lambda$ is the path shown in Figure~\ref{fig::13}. Let $H$ denote the composition $H=F_1\circ G_3$. Note that $H$ is a cobordism map for a 3-handle followed by a 1-handle.

 Multiplying out the above equation, we get
\begin{align*}G^A\circ F^A& \simeq S_{w_2}^-A_\lambda H A_\lambda S_{w_2}^++US_{w_2}^- A_\lambda H S_{w_2}^+ \Phi_{w_1}\\
&\qquad +U \Phi_{w_1} S_{w_2}^- H A_\lambda S_{w_2}^++ U^2 \Phi_{w_1} S_{w_2}^- H S_{w_2}^+ \Phi_{w_1}\\
&\simeq S_{w_2}^- A_\lambda H A_{\lambda} S_{w_2}^++U^2 \Phi_{w_1}^2 S_{w_2}^- H S_{w_2}^+\\
&\simeq S_{w_2}^- A_\lambda H A_{\lambda} S_{w_2}^+\\
&\simeq S_{w_2}^- A_\lambda F_1G_3 A_\lambda S_{w_2}^+
\end{align*}

The second equality follows by noting that the middle two terms of the right side of the top equation are equal, since the $S_{w_2}^{\pm}$ terms commute with the handle attachment maps $F_1$ and $G_3$, and hence can be moved so that a $S_{w_2}^- A_\lambda S_{w_2}^+$ term appears. Upon applying relation \eqref{eq:R7} to replace these terms with $1$, the middle two terms on the right side of the first equality are equal and sum to zero. Similarly in the last expression of the first equation, the $\Phi_{w_1}$ term commutes with everything, so one can rearrange it to get the last term in the second equation. The third equality follows from the fact that $\Phi_{w_1}^2\simeq 0$. The final equality follows since $H=F_1G_3$ when there are enough basepoints to form the composition.

Hence to show that $G^A\circ F^A\simeq 1$, it is sufficient to show that
\[S_{w_2}^- A_\lambda F_1 G_3 A_\lambda S_{w_2}^+\simeq 1.\] To this end, we observe that in the matrix decomposition of higher and lower generators in the diagram of Figure~\ref{fig::13}, using Lemma~\ref{lem:diffmodelcomp} we have
\[A_\lambda F_1 G_3 A_\lambda= \begin{pmatrix}
0& U\\
1& 0\end{pmatrix} \begin{pmatrix} 1\\0\end{pmatrix} \begin{pmatrix}0& 1 \end{pmatrix}\begin{pmatrix}0& U\\
1& 0
\end{pmatrix}=\begin{pmatrix} 0&0\\
1& 0\end{pmatrix}=A_\lambda+U F_1G_3,\] and hence we know that $A_\lambda F_1G_3 A_\lambda\simeq A_\lambda+UF_1 G_3$.

We now compute that
\begin{align*}S_{w_2}^- A_\lambda F_1 G_3 A_\lambda S_{w_2}^+&\simeq S_{w_2}^- (A_\lambda+U F_1 G_3) S_{w_2}^+\\
&\simeq S_{w_2}^- A_\lambda S_{w_2}^+ +US_{w_2}^-H S_{w_2}^+\\
&\simeq 1+U S_{w_2}^- S_{w_2}^+ H\\
&\simeq 1, \end{align*} as we wanted. The first equality follows from the computation in the previous paragraph, the second equality follows from the fact that $H=F_1G_3$,  the third equality follows since the free-stabilization maps commute with cobordism map $H$, and the fourth equality follows since $S_{w_2}^- S_{w_2}^+\simeq 0$. This completes our proof that $G^A\circ F^A\simeq 1$.

We now proceed in a similar manner to compute $F^A\circ G^A$. Analogously to the previous computation, we have that
\[F^A\circ G^A\simeq G_{3}(A_{\lambda} S_{w_2}^++U S_{w_2}^+\Phi_{w_1})(S_{w_2}^- A_{\lambda}+U\Phi_{w_1}S_{w_2}^-) F_{1},\] where $\lambda$ is the path shown in Figure~\ref{fig::13}, and $F_1$ is the 1-handle map, and $G_3$ is the 3-handle map.

Multiplying out the above expression, we see that
\begin{align*}F^A\circ G^A&\simeq G_{3} (A_{\lambda} S_{w_2}^++U S_{w_2}^+\Phi_{w_1})(S_{w_2}^- A_{\lambda}+U\Phi_{w_1}S_{w_2}^-) F_{1}\\
&G_3( A_\lambda S_{w_2}^+S_{w_2}^- A_\lambda+UA_{\lambda} S_{w_2}^+\Phi_{w_1} S_{w_2}^-\\
&\qquad +U S_{w_2}^+\Phi_{w_1} S_{w_2}^- A_{\lambda}+ U^2S_{w_2}^+\Phi_{w_1}^2 S_{w_2}^-) F_1\\
&\simeq G_3( A_\lambda \Phi_{w_2} A_{\lambda}+U A_{\lambda} \Phi_{w_2} \Phi_{w_1}+U\Phi_{w_2}\Phi_{w_1} A_{\lambda}+0)F_1.
\end{align*}

We claim that the above composition simplifies to $G_3 A_\lambda F_1$. We compute each of the three terms in the above expression individually, and check that they add up to $G_3 A_\lambda F_1$. We note that using \eqref{eq:R4}, we have
\[G_3A_\lambda \Phi_{w_2} A_{\lambda} F_1\simeq G_3 A_\lambda F_1+ G_3 \Phi_{w_2} A_{\lambda}^2 F_1\simeq G_3 A_\lambda F_1+  U G_3 \Phi_{w_2} F_1\simeq G_3 A_\lambda F_1+ U \Phi_{w_2} G_3F_1\]
\[\simeq G_3A_\lambda F_1+0.\]

The other terms yield
\[UG_3( A_{\lambda} \Phi_{w_2} \Phi_{w_1}+\Phi_{w_2}\Phi_{w_1} A_{\lambda})F_1\simeq U G_3( \Phi_{w_1} +\Phi_{w_2} A_\lambda \Phi_{w_1} +\Phi_{w_2} A_{\lambda} \Phi_{w_1} +\Phi_{w_2})F_1\]
\[\simeq U G_3(\Phi_{w_1} +\Phi_{w_2}) F_1\simeq U \Phi_{w_1} G_3 F_1+ U\Phi_{w_2} G_3 F_1\simeq 0.\]

Adding these contributions together, we see that $F^A\circ G^A\simeq G_3 A_{\lambda} F_1$. Again using Lemma~\ref{lem:diffmodelcomp}, we see that with respect to a particular almost complex structure, we have
\[A_\lambda=\begin{pmatrix}0&U\\
1& 0\end{pmatrix},\] and hence we have that
\[G_3A_\lambda F_1=\begin{pmatrix}0&1
\end{pmatrix}\begin{pmatrix}0&U\\1&0\end{pmatrix} \begin{pmatrix}
1\\0\end{pmatrix}=1,\] completing the proof.
\end{proof}

\subsection{A formula for the involution on the connected sum}

In this section we prove Proposition~\ref{prop:involutiveconnectedsumform}, a formula for the connected sum on the involution in terms of the chain homotopy equivalences $F^A$ and $G^A$ described in the previous subsection. Proposition~\ref{prop:ConnSum} will follow readily from this.

As we discussed in Subsection~\ref{subsec:relhomologymaps}, the graph cobordism maps considered in \cite{Zemke2} use the relative homology maps $A_\lambda$ which count holomorphic disks with a multiplicity which depends on how they interact with the $\alphas$-curves. One could instead count disks based on how they interact with the $\betas$-curves to define maps $B_\lambda$, and use these to define similar graph cobordism maps. In this section, we need to consider both types of maps, and we will write
 \[F_{W,\Gamma,\mathfrak{s}}^B\] for the maps defined with the relative homology maps $B_\lambda$, and we will write $F_{W,\Gamma,\mathfrak{s}}^A$ for the maps defined with the maps $A_\lambda$, which were the ones considered in \cite{Zemke2}.

Let $F^B$ and $G^B$ be the graph cobordism maps for the graph cobordisms shown in Figure~\ref{fig::3}, using the maps $B_\lambda$.

Applying Proposition~\ref{prop:iotacommuteswithgraphcobordismmaps}, we immediately get the following:

\begin{proposition}The following diagrams commute up to $\Z_2[U]$-equivariant chain homotopy:

\[\begin{tikzcd} CF^\circ(Y_1\# Y_2,w,\mathfrak{s}_1\# \mathfrak{s}_2)\arrow{r}{F^A}\arrow{d}{\iota}&CF^\circ(Y_1\sqcup Y_2,\{w_1,w_2\},\mathfrak{s}_1\sqcup \mathfrak{s}_2)\arrow{d}{\iota}\\
CF^\circ(Y_1\# Y_2,w,\overline{\mathfrak{s}}_1\# \overline{\mathfrak{s}}_2)\arrow{r}{F^B}& CF^\circ(Y_1\sqcup Y_2,\{w_1,w_2\},\overline{\mathfrak{s}}_1\sqcup \overline{\mathfrak{s}}_2)
\end{tikzcd},\]

\[\begin{tikzcd}CF^\circ(Y_1\sqcup Y_2,\{w_1,w_2\},\mathfrak{s}_1\sqcup \mathfrak{s}_2)\arrow{d}{\iota}\arrow{r}{G^A}& CF^\circ(Y_1\# Y_2,w,\mathfrak{s}_1\# \mathfrak{s}_2)\arrow{d}{\iota}\\
 CF^\circ(Y_1\sqcup Y_2,\{w_1,w_2\},\overline{\mathfrak{s}}_1\sqcup \overline{\mathfrak{s}}_2)\arrow{r}{G^B}& CF^\circ(Y_1\# Y_2,w,\overline{\mathfrak{s}}_1\# \overline{\mathfrak{s}}_2).\end{tikzcd}\] All of the above complexes are colored so that all $U$ variables are identified. Furthermore, $F^A$ and $G^A$ are homotopy inverses, and $F^B$ and $G^B$ are also homotopy inverses of each other.
\end{proposition}
\begin{proof}The fact that the diagrams commute up to chain homotopy follows from Lemma~\ref{prop:iotacommuteswithgraphcobordismmaps}. The fact that $F^A$ and  $G^A$ are equivariant homotopy inverses of each other follows from Lemma~\ref{prop:maincomputation}. That $F^B$ and $G^B$ are homotopy inverses follows from the proof of Lemma~\ref{prop:maincomputation} by simply changing all of the $A$'s to $B$'s.
\end{proof}

We now need to consider the difference between the cobordism maps $F^A_{W,\Gamma,\mathfrak{s}}$ and $F^B_{W,\Gamma,\mathfrak{s}}$. We have the following:

\begin{lemma}\label{lem:relatetypeAtypeB} If $(W,\Gamma)$ is a graph cobordism such that the cyclically ordered graph $\Gamma$ has only vertices of valence at most three. Then
\[F^B_{W,\Gamma,\mathfrak{s}}=F^A_{W,\overline{\Gamma},\mathfrak{s}},\] where $\overline{\Gamma}$ is the cyclically ordered graph which has the same underlying graph as $\Gamma$, but with all cyclic orders reversed.
\end{lemma}

\begin{proof}By the composition law for graph cobordism maps, it is sufficient to show that the claim is true for the graph action map. Since the graph action map is functorial under composition, it is sufficient to show the claim for the elementary flow graphs described in \cite[Section 7]{Zemke2}; see Figure 8 in that paper. For elementary flowgraphs with no vertices of valence greater than two, using \eqref{eq:R3}, \eqref{eq:R6} and \eqref{eq:R12}, it is an easy computation that the type-$A$ graph action map and type-$B$ graph action map agree. Hence it is sufficient to show the claim for an elementary flow graph with a single trivalent vertex. There are four elementary flow graphs which have a trivalent vertex, as shown in Figure~\ref{fig::14}. Since the graph action map is invariant under the decomposition of the graph into elementary flow graphs, by changing the decomposition as for example in Figure~\ref{fig::15}, we can assume that the trivalent, elementary flow graph is the graph $\Gamma = \mathcal{G}_F$ from Figure~\ref{fig::8}.

\begin{figure}[ht!]
\centering
\includegraphics[scale=1.2]{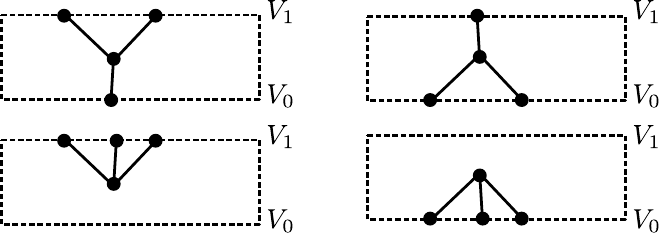}
\caption{The four types of elementary flow graphs with a trivalent vertex.\label{fig::14}}
\end{figure}
\begin{figure}[ht!]
\centering
\includegraphics[scale=1.2]{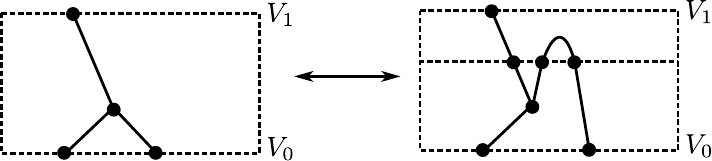}
\caption{A manipulation to change an elementary flow graph with a trivalent vertex, into a composition of different one of the trivalent elementary flowgraphs shown in Figure~\ref{fig::14}, composed with another elementary flowgraph which has no interior vertices.\label{fig::15}}
\end{figure}

 By definition, the graph action map for $\Gamma=\mathcal{G}_F$ is
\[A_\Gamma=S^-_{wv} A_{e_2}A_{e_1}A_{e_3} S_{w_1w_2v}^+.\] We compute using Relation \eqref{eq:R3}, and then repeatedly applying the relations $\Phi_u=S^+_u S_u^-$ and $S^-_u S_u^+=0$, to get
\begin{align*}A_\Gamma&\simeq S^-_{wv} A_{e_2}A_{e_1}A_{e_3} S_{w_1w_2v}^+\\
&\simeq S^-_{wv} (B_{e_2}+U\Phi_v+U\Phi_{w_2})(B_{e_1}+U\Phi_v+U\Phi_{w_1})(B_{e_3}+U\Phi_v+U \Phi_{w}) S_{w_1w_2v}^+\\
&\simeq S^-_{wv} (B_{e_2}+U\Phi_{w_2})(B_{e_1}+U\Phi_v)(B_{e_3}) S_{w_1w_2v}^+\\
&\simeq S_{wv}^- B_{e_2}B_{e_1}B_{e_3} S_{w_1w_2v}^+ +US_{wv}^-B_{e_2}\Phi_v B_{e_3}  S_{w_1w_2v}^+\\
&\qquad +US_{wv}^-  \Phi_{w_2}B_{e_1}B_{e_3}S_{w_1w_2v}^++U^2S_{wv}^- \Phi_{w_2} \Phi_v B_{e_3}S_{w_1w_2v}^+\\
&\simeq  B_\Gamma+U S_{wv}^-\Phi_v B_{e_2} B_{e_3}  S_{w_1w_2v}^++US_{wv}^- B_{e_3} S_{w_1w_2v}^+\\
&\qquad +US_{wv}^- B_{e_1}B_{e_3}\Phi_{w_2} S_{w_1w_2v}^++0\\
&\simeq B_\Gamma+US_{wv}^- B_{e_3} S_{w_1w_2v}^+\\
&\simeq B_\Gamma+U S_{w}^- S_{w_1w_2}^+.
\end{align*}
For example, to go from the second line above to the third, in the first parenthesis we eliminated the term $U\Phi_{v}$ using the relation $S^-_v \Phi_{v} = S^-_v S^+_v S^-_v = 0$, and we similarly eliminated terms from the other two parentheses. To eliminate the term $U^2S_{wv}^- \Phi_{w_2} \Phi_v B_{e_3}S_{w_1w_2v}^+$ from the fifth line, we used the relation \eqref{eq:R4} and the cancellations of $\Phi_{w_2}$ and $\Phi_v$ against the corresponding positive free-stabilizations. To obtain the last line from the previous one, we used the analogue of \eqref{eq:R7} for the $B$ maps.

Let us now consider the cyclically ordered graph $\overline{\Gamma}$ which is the graph $\Gamma$ with cyclic orders reversed. In \cite[Lemma 5.9]{Zemke2}, the commutator of two relative homology maps is computed. In the case that $\lambda_1$ and $\lambda_2$ share exactly one vertex, the formula reads $A_{\lambda_1}A_{\lambda_2}+A_{\lambda_2}A_{\lambda_1}\simeq U$. Using this, we see that
\begin{align*}A_{\overline{\Gamma}}&\simeq S^-_{wv} A_{e_1}A_{e_2}A_{e_3} S_{w_1w_2v}^+\\
&\simeq  S^-_{wv} A_{e_2}A_{e_1}A_{e_3} S_{w_1w_2v}^++U S^-_{wv} A_{e_3} S_{w_1w_2v}^+\\
&\simeq S^-_{wv} A_{e_2}A_{e_1}A_{e_3} S_{w_1w_2v}^++U S^-_{w} S_{w_1w_2}^+.
 \end{align*}
 
 Combining this with our computation of $B_\Gamma$, we thus conclude that $B_\Gamma=A_{\overline{\Gamma}}$, completing the proof.
\end{proof}

If $\Gamma$ is a graph, using \cite[Lemma 5.9]{Zemke2}, which computes the commutator of two relative homology maps, the difference between $A_\Gamma$ and $A_{\overline{\Gamma}}$ can be described by the relation shown in Figure~\ref{fig::2}.

\begin{figure}[ht!]
\centering
\includegraphics[scale=1.2]{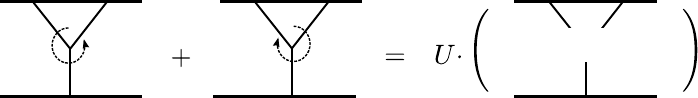}
\caption{The effect of changing the cyclic ordering of the edge adjacent to a trivalent vertex. This relation follows immediately from the commutator of two relative homology maps, computed in \cite[Lemma 5.9]{Zemke2}.\label{fig::2}}
\end{figure}

We are finally in position to compute the involution on connected sums.

\begin{proposition}\label{prop:involutiveconnectedsumform}If $G^A:CF^\circ(Y_1\sqcup Y_2,w_1,w_2,\mathfrak{s}_1\sqcup \mathfrak{s}_2)\to CF^\circ(Y_1\# Y_2, w, \mathfrak{s}_1\# \mathfrak{s}_2)$ denotes the homotopy equivalence defined earlier, then
\[\iota G^A+G^A (\iota_1\otimes \iota_2)\simeq U G^A (\Phi_1\otimes \Phi_2)(\iota_1\otimes \iota_2),\] where $\Phi_1=\Phi_{w_1}$ and $\Phi_2=\Phi_{w_2}$.
\end{proposition}

\begin{proof} By Proposition~\ref{prop:iotacommuteswithgraphcobordismmaps}, we know that
\[\iota G^A\simeq G^B(\iota_1\otimes \iota_2).\] Using Lemma~\ref{lem:relatetypeAtypeB}, we know that $G^B=G^A_{\text{opp}}$, where $G^A_{\text{opp}}$ is the graph cobordism map for the graph cobordism with opposite cyclic orders. Using the relation of the graph action map shown in Figure~\ref{fig::2}, we see that
\[G^B=G^A+U F_{W,\Gamma_0,\mathfrak{s}}^A,\] where $(W,\Gamma_0,\mathfrak{s})$ is the graph cobordism with the same underlying 4-manifold as the one defined $G^A$ and $G^B$ (a 3-handle cobordism), but where $\Gamma_0$ is the graph obtained by removing a neighborhood of the trivalent vertex from the graph used to define $G^A$. Using the manipulation of graphs shown in Figure~\ref{fig::11}, together with the formula \eqref{eq:R6}, we see that
\[F_{W,\Gamma_0,\mathfrak{s}}^A=G^A(\Phi_1\otimes \Phi_2).\]

\begin{figure}[ht!]
\centering
\includegraphics[scale=1.2]{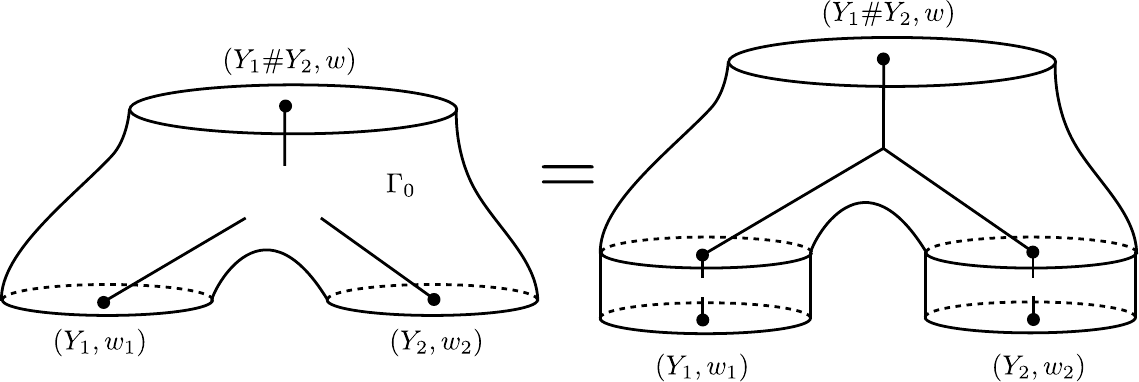}
\caption{A graph manipulation that shows that $F^A_{W,\Gamma_0,\mathfrak{s}}=G^A(\Phi_1\otimes \Phi_2)$. Note that adding trivial strands doesn't affect the graph cobordism map by \cite[Lemma 7.11]{Zemke2}.\label{fig::11}}
\end{figure}

Hence, putting these terms together, we see that
\[\iota G^A\simeq G^B(\iota_1\otimes \iota_2)\simeq G^A(\iota_1\otimes \iota_2)+UG^A(\Phi_1\otimes \Phi_2)(\iota_1\otimes \iota_2),\] completing the proof.
\end{proof}

\begin{proof}[Proof of Proposition~\ref{prop:ConnSum}]
Ignoring the absolute gradings, the theorem is a restatement of Proposition~\ref{prop:involutiveconnectedsumform}, combined with the isomorphism \eqref{eq:disjunion}. Using Proposition~\ref{prop:gradingchangeformula}, one calculates that the maps $F^A$ and $G^A$ have absolute grading change of zero. Furthermore, we have that \[CF^-(Y_1\sqcup Y_2,w_1,w_2,\s_1\sqcup \s_2)\cong CF^-(Y_1,w_1,\s_1)\otimes_{\Z_2[U]} CF^-(Y_2,w_2,\s_2)[-2],\] as absolutely graded $\Z_2[U]$-modules; cf. Remark~\ref{rem:absolute6}.
\end{proof}

\section{Proof of the main theorem}
\label{sec:nullchainhomotopy}
In this section we prove Theorem~\ref{thm:ConnSum}, the connected sum formula for involutive Heegaard Floer homology. We will deduce it from Proposition~\ref{prop:ConnSum}, by showing that the term $U(\Phi_1 \iota_1 \otimes \Phi_2 \iota_2)$ is $U$-equivariantly chain homotopic to zero. 

We start with a few algebraic results. The following lemma is well-known, but we include a proof for completeness.

\begin{lemma}[Classification of free, finitely generated chain complexes over a PID]\label{lem:classify-cc-overPID} If $(C,\partial)$ is a free, finitely generated chain complex over a PID, then $C$ is isomorphic to the direct sum of complexes of the form $R$ (with vanishing differential) or complexes of the form $(R\xrightarrow{\phantom{a}\times p\phantom{a}} R)$ for $p\in R$.
\end{lemma}
\begin{proof}This follows from the stacked basis theorem for finitely generated modules over a PID (cf. \cite[Section 12.1, Theorem 4]{DummitFoote}) which says that if $M\subset N$ are free, finitely generated $R$-modules with $R$ a PID, then there are integers $m,n$ with $m\le n$, as well as nonzero elements $r_1,\dots, r_m\in R$ and $y_1,\dots, y_n\in N$ such that $y_1,\dots, y_n$ is a basis of $N$ and $r_1y_1,\dots, r_my_m$ is a basis for $M$.

 If $(C,\partial)$ is a free, finitely generated chain complex over $R$, we let $B=\im (\partial)$ and $Z=\ker (\partial)$. Pick a stacked basis $z_1,\dots,z_n$ of $Z$ such that there are nonzero $a_i\in R$ such that $a_1z_1,\dots, a_m z_m$ is a basis for $B$. Now pick elements $c_1,\dots, c_m$ such that $\partial( c_i)=a_i z_i$. We now claim that $c_1,\dots, c_m, z_1,\dots, z_n$ is a basis of $C$. Note that given the explicit description of the differential of these generators, the main claim follows once we establish that this is a basis. Suppose $x\in C$ is an arbitrary element. Then $\partial x\in B=\Span(a_1z_1,\dots, a_mz_m)$ and hence there are elements $r_1,\dots, r_m\in R$ such that
\[\partial(x)=r_1a_1 c_1+\dots r_ma_m c_m\] and hence $x-r_1c_1-\dots-r_mc_m\in Z=\ker (\partial)$. Hence in particular there are $r_1',\dots, r'_n\in R$ such that
\[x=r_1c_1+\cdots r_m c_m+r_1'z_1+\cdots r_n' z_n.\] Hence the set $c_1,\dots, c_1,z_1,\dots, z_n$ spans $C$. We need to now show that it is linearly independent. Suppose that $r_1,\dots, r_m,r_1',\dots, r_n'$ are elements of $R$ such that
\[r_1c_1+\cdots r_m c_m+r_1'z_1+\cdots r_n' z_n=0.\] Applying $\partial$ yields $r_1a_1z_1+\cdots +r_ma_mz_m=0$, implying that $r_i=0$, since $a_1z_1,\dots, a_mz_m$ is a basis for $B$. Hence $r_1'z_1+\cdots r_n'z_n=0$, implying that $r_i'=0$, since $z_1,\dots, z_n$ is a basis of $Z$.
\end{proof}

\begin{lemma}\label{lem:simplefactors} Let $M$ be a finitely generated module over $\Z_2[U]$ which admits a relative $\Z$-grading. If $M$ has a submodule isomorphic to $\Z_2[U]/(p(U))$ for some $p(U)\in \Z_2[U]$, then $p(U)=U^i$ for some $i\ge 0$.
\end{lemma}
\begin{proof} Suppose to the contrary that $M$ has a submodule (not necessarily a homogeneous submodule) isomorphic to $\Z[U]/(p(U))$ where $p(U)=U^i(1+Uq(U))$ with $q(U)\neq 0$. We can apply the Chinese Remainder Theorem to see that $M$ has a submodule isomorphic to
\[(\Z_2[U]/U^i)\oplus (\Z_2[U]/(1+Uq(U))),\] and hence in particular a submodule isomorphic to $\Z_2[U]/(1+Uq(U))$. The submodule isomorphic to $\Z_2[U]/(1+Uq(U))$ may not be a homogeneous submodule of $M$, but nonetheless, $U$ is invertible on this submodule since it has inverse $q(U)$. Given $x\in \Z_2[U]/(1+Uq(U))$ we can consider $m(x)\in \Z\cup \{-\infty\}$ defined to be the maximum grading of any nonzero homogeneous summand of $x$ (and $-\infty$ if $x=0$). Noting that $m(U\cdot x)\le m(x)-2$ and $m(q(U)\cdot x)\le m(x)$, but that $(U q(U))\cdot x=x$ for $x$ in this summand, we arrive at a contradiction.
\end{proof}

\begin{proposition}\label{prop:UPhi=0} If $C$ is a free, finitely generated chain complex over $\Z_2[U]$ which is relatively $\Z$-graded (with $U$ having grading $-2$), then $U\Phi\simeq 0$ and the chain homotopy can be taken to be $U$-equivariant.
\end{proposition}
\begin{proof} The proof uses the classification of finitely generated free chain complexes over a PID (Lemma \ref{lem:classify-cc-overPID}). By that lemma, we can write the chain complex $C$ as a direct sum of complexes of the form $\Z_2[U]$ (with vanishing differential) and complexes of the form $\Z_2[U]\xrightarrow{\times p(U)} Z_2[U]$, where $p(U)\in \Z_2[U]$. Using Lemma \ref{lem:simplefactors}, the assumption that $C$ is relatively graded forces each $p(U)$ to be of the form $U^i$ for some $i\ge 0$. Since $\Phi$ decomposes over direct sums in the natural way, it is thus sufficient to show the claim in the case that $C$ is one of the elementary chain complexes described above. If $C=\Z_2[U]$, with vanishing differential, the map $\Phi$ vanishes, so the claim is trivially true.

It remains to consider the case when $C$ is a two-step complex. Suppose $C$ has two generators $a$ and $b$, and differential $\del$ given by $\del a = U^n b$, $\del b=0$. Schematically, we write 
\[\partial=\big(\begin{tikzcd}a\arrow{r}{U^{n}}& b\end{tikzcd}\big),\]

By definition we have $\Phi(a) = n U^{n-1} b$ and $\Phi(b)=0$ which we will write as
\[\Phi=\big(\begin{tikzcd}a\arrow{r}{nU^{n-1}}& b\end{tikzcd}\big)\] and hence
\[U\Phi=\big(\begin{tikzcd}a\arrow{r}{nU^{n}}& b\end{tikzcd}\big).\]Define a map $H:C\to C$ as
\[H=\big(\begin{tikzcd}a& b\arrow[loop right]{l}{n}\end{tikzcd}\big).\]
 An easy computation shows that
\[U\Phi=\partial H+H\partial.\] We note that the map $H$ is also $U$-equivariant.

We now wish to show that $H$ can be taken to be a 0-graded map, with respect to the relative $\Z$-grading. Note that $\partial$ and $U\Phi$ are both $-1$ graded maps. Write $H=\sum_{k\in \Z} H_k$, where $H_k$ lowers the grading of a homogeneous element by $k$. We note that each $H_k$ is $U$-equivariant since $U$ shifts grading by $-2$. Since $\partial$ changes the grading by $-1$, we note that
\[(\partial H+H\partial)_k=\partial H_{k-1} +H_{k-1} \partial.\] Hence
\[U\Phi=(U\Phi)_{1}=(\partial H+H\partial)_{1}=\partial H_0+H_0\partial,\]
so we can replace $H$ by $H_0$ to get the desired homotopy.
\end{proof}

Now let $Y$ be a three-manifold equipped with a $\spinc$ structure $\s$. If $c_1(\s)$ is torsion, then the complexes $CF^-(Y,\s)$ have a relative $\Z$-grading, as shown in \cite{HolDisk}. The following is 
an immediate corollary to Proposition~\ref{prop:UPhi=0}. 

\begin{corollary}
\label{cor:CFM}
If $c_1(\s)$ is torsion, e.g. if $Y$ is a rational homology sphere, then $U\Phi\simeq 0$ on $\CFm(Y, \s)$, $U$-equivariantly.
\end{corollary}

\begin{remark}Care should be taken when applying Proposition \ref{prop:UPhi=0}. We list the following caveats:
\begin{enumerate}
\item The maps $\Phi$ are chain homotopic to zero over $\Z_2$ (cf. Section~\ref{sec:Phi}), but usually not $U$-equivariantly chain homotopic to zero. 
\item The maps $U_w \Phi_w$ will not be equivariantly chain homotopic to zero in the presence of multiple basepoints on a given component. Indeed if one identifies $U_w$ with $U_{w'}$ for $w'$ another basepoint on the same component as $w$, then it is easily seen that $U_w \Phi_w$ does not even vanish on homology.
\item Finally, Corollary~\ref{cor:CFM} fails when $\s$ is not torsion, as we no longer have a relative $\Z$-grading. For example, one can consider $Y=S^1\times S^2$ and a $\spinc$ structure $\s$ with $c_1(\s)=2$. A model for $CF^-(S^1\times S^2,\s)$ is $(a\xrightarrow{1-U} b)$. In this case the map $U\Phi$ is not $U$-equivariantly chain homotopic to zero; indeed, if we tensor with $\Z_2[U]/(U-1)$, then the resulting map acts nontrivially on homology.
\end{enumerate}
\end{remark}

\begin{proof}[Proof of Theorem~\ref{thm:ConnSum}]
Note that $c_1(\s)$ is torsion for any spin structure $\s$. In view of Corollary~\ref{cor:CFM}, if $\s_1$ and $\s_2$ are two spin structures on $Y_1$ and $Y_2$, we have a $U$-equivariant chain homotopy $U(\Phi_1 \otimes 1) \simeq 0$, and therefore
$$ U(\Phi_1 \iota_1 \otimes \Phi_2 \iota_2) = U(\Phi_1 \otimes 1) \circ (\iota_1 \otimes \Phi_2 \iota_2) \simeq 0.$$
The statement about the involution on $\CFm(Y_1 \# Y_2, \s_1 \# \s_2)$ in Theorem~\ref{thm:ConnSum} now follows from Proposition~\ref{prop:ConnSum}. The results for $\CFhat$ and $\CFinf$ are a consequence of the one for $\CFm$.
\end{proof}

\section{Inequalities}
\label{sec:inequalities}

In this section we prove the inequalities relating the involutive correction terms of connected sums to the involutive correction terms of the summands.

\begin{proof}[Proof of Proposition \ref{prop:inequalities}]
By Proposition \ref{prop:orientationsigns}, it suffices to prove the first and third inequalities. We begin with the first. Suppose that $v_1$ is a maximally-graded homogeneous element of $(\CFm(Y_1,\s_1), \del_1)$ with the property that $[U^n v_1]\neq 0$ for any $n$ and $(1+\iota_1)v_1=\partial_1 w_1$.  Recall from Lemma \ref{lemma:reformulation} that $\gr(v_1) +2 = \dl(Y_1,\s_1)$. Similarly let $v_2 \in (\CFm(Y_2,\s_2), \del_2)$ be a a maximally-graded homogeneous element with the property that $[U^nv_2]\neq 0$ for any $n$ and $(1+\iota_2)v_2=\partial_2 w_2$, so that $\gr(v_2)+2 = \dl(Y_2, \s_2)$. Now consider the element $v= v_1\otimes v_2 \in \CFm(Y_1\s_1)\otimes_{\Z_2[U]}\CFm(Y_2,\s_2)[-2]$. Let the differential on this complex be denoted $\del$. We see that $\partial(v_1 \otimes v_2) = 0$, and furthermore we have $[U^n(v_1 \otimes v_2)]\neq 0$ for all $n$. After applying Theorem~\ref{thm:ConnSum}, we can take the involution on the tensor product to be $\iota = \iota_1\otimes \iota_2$.  Set $w = v_1 \otimes w_2 + w_1 \otimes \iota_2(v_2)$. Then we see
\begin{align*}
\del(w) &= v_1 \otimes (1+\iota_2)v_2 + (1+\iota_1)v_1 \otimes \iota_2(v_2) \\
&= v_1 \otimes v_2 + \iota_1(v_1)\otimes \iota_2(v_2) \\
&= (1+\iota)(v)
\end{align*}
We conclude that $\dl(Y_1\#Y_2, \s_1\#\s_2)$ is at least the grading of $\gr(v)+2 = (\gr(v_1) + \gr(v_2) +2) +2 = (\gr(v_1) + 2) + (\gr(v_2) +2)$, which is exactly $\dl(Y_2, \s_1) +\dl(Y_2, \s_2)$. This proves the first inequality.

Now let us consider the third inequality. Let $v_1$ and $w_1$ be as above, so that $\gr(v_1) +2 = \dl(Y_1, \s_1)$. Choose $x_2, y_2,z_2$ elements of $\CFm(Y_2, \s_2)$ maximally graded such that at least one of $x_2$ and $y_2$ is nonzero, $\partial y_2 = (1+ \iota_2)x_2$, there exists $m$ such that,  $\partial_2 z_2 = U^m x_2$, and $[U^n(U^my_2 + (1+\iota_2)z_2)] \neq 0$ for any $n \geq 0$. Then by Lemma \ref{lemma:reformulation} either $\du(Y_2,\s_2) = \gr(x_2)+3$ or $\du(Y_2,\s_2) = \gr(y_2)+2$, whichever is defined. Now in the complex $\CFm(Y_1, \s_1)\otimes_{\Z_2[U]} \CFm(Y_2,\s_2)[-2]$, set %
\begin{align*}
x &= v_1 \otimes x_2 \\
y &= w_1 \otimes x_2 + \iota_1(v_1) \otimes y_2 \\
z &= v_1 \otimes z_2
\end{align*}
Notice that if $x=0$, then in particular $x_2 = 0$, implying that $y_2 \neq 0$, so $y = \iota_1v_1 \otimes y_2$ is nonzero. Hence at least one of $x$ and $y$ is nonzero. Then we see that $\partial(z) = \partial(v_1 \otimes z_2) = U^m(v_1 \otimes x_2) = U^m x$. Furthermore, we observe that
\begin{align*}
\partial(y) &= (1+\iota_1)v_1 \otimes x_2 + \iota_1(v_1) \otimes (1+\iota_2)x_2 \\
&= v_1 \otimes x_2 + \iota_1(v_1) \otimes \iota_2(x_2)\\
&= (1+\iota)x.
\end{align*}
Here the first step uses the fact that $\partial_2 x_2 = 0$ (since $U^m \partial_2 x_2 = \partial (U^m x_2) = \partial_2 z_2 =0$, and $\CFm(Y,\s)$ is a free complex over $\F_2[U]$). Finally, we need to check that $[U^m y + (1 + \iota)z]$ is a generator of the $U$-tail in $\HFm(Y_1 \# Y_2, \s_1 \# \s_2)$. But observe that
\begin{align*}
U^m y + (1+\iota)z &= U^m(w_1 \otimes x_2) + U^m(\iota_1v_1 \otimes y_2) + v_1\otimes z_2 + \iota_1(v_1)\otimes \iota_2(z_2) \\
&= \iota_1 v_1 \otimes U^m y_2 + \iota_1v_1 \otimes (z_2 + \iota_2z_2) + (v_1+\iota_1v_1)\otimes z_2 + w_1\otimes U^m x_2 \\
&= \iota_1 v_1 \otimes (U^m y_2 + (z_2 + \iota_2 z_2)) + \del(w_1 \otimes z_2)
\end{align*}

Since $[v_1]$ (and therefore also $[\iota(v_1)]$) is a generator of the $U$-tail in $\HFm(Y_1, \s_1)$ and $[U^my_2 + (z_2 + \iota_2 z_2)]$ is a generator of the $U$-tail in $\HFm(Y_2, \s_2)$, we conclude that $[y + (1 + \iota)z]$ is a generator for the $U$-tail in $\HFm(Y_1 \# Y_2, \s_1 \# \s_2)$. Therefore, $\dl(Y_1 \# Y_2, \s_1 \# \s_2)$ is at least three plus the grading of $x$, or at least two plus the grading of $y$ if $x$ is zero. But if $x\neq 0$, we have $\gr(v_1 \otimes x_2)+3 = \gr(v_1) + \gr(x_2) + 3+2 = (\gr(v_1) +2) + (\gr(x_2)+3)$, which is exactly $\dl(Y_1, \s_1)+ \du(Y_2, \s_2)$. Similarly if $y\neq 0$, we obtain the same inequality.\end{proof}

\section{\texorpdfstring{$\inv$}{inv}-complexes}
\label{sec:I}
In this section we do some algebraic constructions, with the goal of defining the group $\Inv$ that appears in Theorem~\ref{thm:Inv}, and of proving that theorem. 

\subsection{Definitions}

\begin{definition}
\label{def:icx}
An {\em $\inv$-complex} is a pair $(C, \inv)$, consisting of
\begin{itemize}
\item a $\Z$-graded, finitely generated, free chain complex $C$ over the ring $\Z_2[U]$, where $\operatorname{deg}(U)=-2$; moreover, we ask that there is a graded isomorphism
\begin{equation}
\label{eq:Utail}
U^{-1}H_*(C) \cong \Z_2[U, U^{-1}];
\end{equation}
\item a grading-preserving chain homomorphism $\iota \co C \to C$, such that $\iota^2$ is chain homotopic to the identity.
\end{itemize}
\end{definition}

Observe that condition \eqref{eq:Utail} in Definition~\ref{def:icx} is equivalent to asking that 
\begin{equation}
\label{eq:dec}
H_*(C) = \Z_2[U]_{(d)} \oplus C_{\red},
\end{equation}
where $\Z_2[U]_{(d)}$ is the infinite $U$-tail supported in degrees $d$, $d-2$, $\dots$, and $C_{\red}$ is a torsion $\Z_2[U]$ module. The requirement that the isomorphism \eqref{eq:Utail} is graded translates into $d$ being an even integer.

\begin{example}
If $Y$ is a homology sphere, then $\CFm(Y)[-2]$, equipped with the conjugation involution, is an $\inv$-complex.
\end{example}

To every $\inv$-complex $(C, \inv)$, we can associate an {\em involutive complex}
\begin{equation}
\label{eq:invcx}
I^-(C, \iota) := \Cone\bigl(C_* \xrightarrow{\phantom{o} Q (1+\inv) \phantom{o}} Q \ccdot C_* [-1]\bigr),
 \end{equation}
which is a free complex over $\Ring = \Z_2[Q, U]/(Q^2)$, and is the analogue of $\CFIm(Y)$. 

Next, we introduce two equivalence relations on $\inv$-complexes.  Recall that the symbol $\simeq$ denotes a $\Z_2[U]$-equivariant chain homotopy.

\begin{definition}
\label{def:E}
Two pairs $(C, \inv)$ and $(C', \inv')$ are called {\em equivalent} if there exist chain homotopy equivalences
$$ F \co C \to C', \ \ G \co C' \to C$$
that are homotopy inverses to each other, and such that 
$$F \circ \inv \simeq \inv' \circ F,  \ \ \ G \circ \inv' \simeq \inv \circ G.$$
\end{definition}

Observe that an equivalence of $\inv$-complexes induces a quasi-isomorphism between the respective involutive complexes \eqref{eq:invcx}.

\begin{example}
If $Y$ is a homology sphere, then different choices of Heegaard diagram $\H$ and Heegaard moves from $\bH$ to $\H$ yield different $\inv$-complexes $(\CFm(\H)[-2], \inv)$. However, these $\inv$-complexes are equivalent, and hence the corresponding involutive complexes $\CFIm(\H)$ are quasi-isomorphic; cf. \cite[Proposition 2.7]{HMinvolutive}. 
\end{example}

We also need a weaker form of equivalence, modelled on the relation on Heegaard Floer complexes  induced by a homology cobordism. 

\begin{definition}
\label{def:LE}
Two pairs $(C, \inv)$ and $(C', \inv')$ are called {\em locally equivalent} if there exist (grading-preserving) homomorphisms
$$ F \co C \to C', \ \ G \co C' \to C$$
such that 
$$F \circ \inv \simeq \inv' \circ F,  \ \ \ G \circ \inv' \simeq \inv \circ G,$$
and $F$ and $G$ induce isomorphisms on homology after inverting the variable $U$.
\end{definition}

\subsection{The multiplication}
Let $(C, \inv)$ and $(C', \inv')$ be $\inv$-complexes. We define their {\em product} by the formula
\begin{equation}
\label{eq:multi}
 (C, \inv) * (C', \inv') := (C \otimes_{\Z_2[U]} C', \ \inv \otimes \inv' ).
 \end{equation}

\begin{lemma}
\label{lem:eltInv}
The result of the multiplication \eqref{eq:multi} is an $\inv$-complex.
\end{lemma}

\begin{proof}
All of the required properties are easy to verify; for example, the fact that $C \otimes_{\Z_2[U]} C'$ satisfies \eqref{eq:Utail} follows from the corresponding statements for $C$ and $C'$, by inverting $U$ and then using the K\"unneth formula. 
\end{proof}

\begin{lemma}
\label{lem:3parts}
Let $(C, \inv)$ and $(C', \inv')$ be $\inv$-complexes.

$(a)$ If we change $(C, \inv)$ or $(C', \inv')$ by an equivalence in the sense of Definition~\ref{def:E}, then their product also only changes by an equivalence.

$(b)$ If we change $(C, \inv)$ or $(C', \inv')$ by a local equivalence in the sense of Definition~\ref{def:LE}, then their product changes by a local equivalence.
\end{lemma}

\begin{proof}
(a) Without loss of generality, we focus on changing the second factor. Suppose we have an equivalence between $(C', \inv')$ and $(C'', \inv'')$, given by homotopy inverses $F \co C' \to C''$ and $G \co C'' \to C'$. By tensoring them with the identity on $C$, we get homotopy inverses
$$ 1 \otimes F \co  C \otimes_{\Z_2[U]} C' \to C \otimes_{\Z_2[U]} C'', \ \ \ 1 \otimes G \co  C \otimes_{\Z_2[U]} C'' \to C \otimes_{\Z_2[U]} C'.$$
We need to check that these interchange the involutions, up to chain homotopy. Indeed, we have
\[ (1 \otimes F)\circ ( \inv \otimes \inv' ) - (\inv \otimes \inv'' ) \circ (1 \otimes F) = \inv \circ ( F \circ \inv' - \inv'' \circ F) \simeq 0.\]
The same argument applies to $1 \otimes G$.

 (b) Similar to part (a). 
\end{proof}

\subsection{The group of \texorpdfstring{$\inv$}{inv}-complexes}
We now construct the group $\Inv$  that appears in Theorem~\ref{thm:Inv}. We let $\Inv$ be the set of involutive classes modulo the relation of local equivalence from Definition~\ref{def:LE}. We put a group structure on $\Inv$ in the following way. We let $$e:= (\Z_2[U]_{(0)}, \id)$$ be the unit, and use \eqref{eq:multi} to define the multiplication.

Furthermore, we construct the inverse of a pair $(C, \inv)$ by considering the dual complex: 
$$ C^{\vee} := \operatorname{Hom}_{\Z_2[U]}(C, \Z_2[U]),$$
equipped with the map $\inv^{\vee}$ induced by $\inv$. Concretely,  if we have a set of free generators $\S$ for $C$ (over $\Z_2[U]$), for every $\x \in \S$ we introduce a free generator $\x^{\vee}$ for $C^{\vee}$ (the map that takes $\x$ to $1$ and all other elements of $\S$ to zero). Let $\S^{\vee} = \{\x^{\vee} \mid \x \in \S\}.$ The grading of $\x^{\vee}$ is the negative of the grading of $\x$. Moreover, for $\x, \y \in \S$, the term $U^n \y$ appears in $\del \x$ for $C$ if and only if $U^n \x^{\vee}$ appears in $\del^{\vee} \y^{\vee}$ for $C^{\vee}$. Also,  $U^n \y$ appears in $\inv(\x)$ for $C$ if and only if $U^n \x^{\vee}$ appears in $\inv^{\vee}(\y^{\vee})$ for $C^{\vee}$.

 \begin{proposition}
 \label{prop:icx}
With these definitions, $\Inv$ is a well-defined Abelian group.
 \end{proposition}

\begin{proof}
In Lemma~\ref{lem:eltInv} we have proved that the result of multiplication is an actual element of $\Inv$. Moreover, part (b) of Lemma~\ref{lem:3parts} shows the product is well-defined, up to local equivalence.

Next, let us verify that the multiplication is associative. If we have three $\inv$-complexes $(C, \iota), (C', \iota')$ and $(C'', \iota'')$, there is a natural isomorphism
$$ (C \otimes_{\Z_2[U]} C') \otimes_{\Z_2[U]} C'' \simeq C \otimes_{\Z_2[U]} (C' \otimes_{\Z_2[U]} C'').$$ 
The involutions $(\inv \otimes \inv') \otimes \inv''$ corresponds to $\inv \otimes (\inv' \otimes \inv'')$  under this isomorphism.

To see that the multiplication is commutative, we simply consider the isomorphisms between $C \otimes_{\Z_2[U]} C'$ and $C' \otimes_{\Z_2[U]} C$ given by reversing the factors. These produce equivalences of $\inv$-complexes.

Furthermore, it is clear that multiplication by the unit $e$ acts as the identity. 

Finally, we need to check that $(C, \inv) * (C^{\vee}, \inv^{\vee})$ is locally equivalent to the unit in $\Inv$. Pick a set $\S$ of free generators for $C$, and use the dual set $\S^{\vee}$ for $C^{\vee}$. We can define module maps
$$ \gamma \co  \Z_2[U]_{(0)} \to C \otimes_{\Z_2[U]} C^{\vee}; \ \ \  \ \ \ G(1) = \sum_{\x \in \S} \x \otimes \x^{\vee}$$
and
$$ \zeta \co C \otimes_{\Z_2[U]} C^{\vee} \to \Z_2[U]_{(0)}; \ \text{for } \x, \y \in \S, \  \x \otimes \y^{\vee} \mapsto \begin{cases} 1 & \text{if } \x = \y, \\
0 & \text{otherwise.} \end{cases}$$
It is straightforward to check that these are chain maps, and are isomorphisms on homology after inverting $U$. In fact, we can identify $ C \otimes_{\Z_2[U]} C^{\vee}$ with the complex
\begin{equation}
\label{eq:Hom}
 \Hom_{\Z_2[U]}(C, C).
 \end{equation}
Under this identification, $\gamma$ corresponds to the map $1 \mapsto \id$, and $\zeta$ is the trace. 
Observe that, because $H_*(C)$ has rank one over $\Z_2[U]$, the set $\S$ must have an odd number of elements, for Euler characteristic reasons. Thus, the composition $\zeta \circ \gamma$ gives the identity (coming from the cardinality of $\S$ mod $2$).

Also under the identification above, for homomorphisms $A \co C \to C$, $B \co C^\vee \to C^\vee$ (which can be represented by square matrices using the bases $\S$ resp. $\S^\vee$), the sum
$$ \sum_{\x \in \S} A(\x) \otimes B(\x^{\vee}),$$
as an element of \eqref{eq:Hom}, corresponds to the matrix $AB^T$ (where the $T$ superscript denotes the transpose).

To see that $\gamma$ commutes with the involutions up to chain homotopy, we need to show that
\begin{equation}
\label{eq:DS}
  \sum_{\x \in \S} \Bigl( \inv(\x) \otimes \inv^{\vee}(\x^{\vee}) \Bigr) -  \sum_{\x \in \S} \x \otimes \x^{\vee}
\end{equation}
is of the form $\del h$ for some $h \in C \otimes_{\Z_2[U]} C^{\vee}$, i.e., that it is zero in homology. Note that the matrix for $\inv^{\vee}$ (in terms of the set $\S^{\vee}$) is the transpose of the matrix for $\inv$ (in terms of $\S$). Therefore, in terms of the identification with \eqref{eq:Hom}, the expression~\eqref{eq:DS} is $ \inv^2 - 1.$ Since $\inv^2 \simeq 1$, the difference is chain homotopic to zero, and hence corresponds to a boundary element in the tensor product $C \otimes_{\Z_2[U]} C^{\vee}$.

With regard to the trace map $\zeta$, observe that this is the dual to $\gamma$. Since $\gamma$ commutes with the involutions up to chain homotopy, so does $\zeta$. This completes the proof that $(C, \inv) * (C^{\vee}, \inv^{\vee}) = e \in \Inv$.
\end{proof}

\subsection{Relation to homology cobordism}
With Proposition~\ref{prop:icx}  in place, we can now define the map
$$ h \co \Theta^3_{\Z} \to \Inv, \ \ Y \mapsto \bigl( \CFm(Y)[-2], \iota \bigr ),$$
as advertised in Theorem~\ref{thm:Inv}.

\begin{proof}[Proof of Theorem~\ref{thm:Inv}.]
To see that $h$ is well-defined as a map of sets, let $Y$ and $Y'$ be integral homology spheres that are homology cobordant. We need to check that they give rise to locally equivalent $\inv$-complexes. Indeed, if $W$ is the homology cobordism from $Y$ to $Y'$, then $W$ induces chain map 
$$F \co \CFm(Y) \to \CFm(Y'),$$
which is an isomorphism on homology after inverting $U$; cf. \cite[Section 9]{AbsGraded}. Furthermore, it follows from the proof of Proposition 4.9 in \cite{HMinvolutive} that $F$ commutes with the conjugation involutions up to chain homotopy. We can also construct a map $G$ in the other direction, by reversing the cobordism. Thus, we obtain the desired local equivalence of $\inv$-complexes.

The fact that $h$ is a homomorphism is immediate from Theorem~\ref{thm:ConnSum} and the definitions.  
\end{proof}

One can mimic many of the constructions in Heegaard Floer homology and apply them to elements of $\Inv$. For example, by taking the values of $d$ in \eqref{eq:dec} we obtain a homomorphism
$$ d_{\Inv} \co \Inv \to 2\Z.$$
Then, the homomorphism $d \co \Theta^3_{\Z} \to 2\Z$ (given by the usual correction term in Heegaard Floer homology) is simply the composition $d_{\Inv} \circ h$.

Recall that, starting from an element $[(C, \inv)] \in \Inv$, we can construct the involutive complex
$I^-(C, \iota)$ from \eqref{eq:invcx}, which is the analogue of $\CFIm(Y)$. By adapting the characterizations of $\dl$ and $\du$ from Lemma~\ref{lem:newinv}, we define maps (which are not homomorphisms)
$$\dl_{\Inv}, \du_{\Inv} \co \Inv \to 2\Z.$$
Then, the involutive correction terms $\dl, \du \co  \Theta^3_{\Z} \to 2\Z$ are the compositions $\dl_{\Inv} \circ h$ and $\du_{\Inv} \circ h$, respectively.

\section{Computations}
\label{sec:computations}

We now use Theorem~\ref{thm:ConnSum} to do some concrete computations. Note that Theorem~\ref{thm:ConnSum} is phrased in terms of the complexes $\CFm(Y_i, \s_i)$. However, in view of Lemma~\ref{lem:3parts} (a), we can replace $(\CFm(Y_i, \s_i), \inv)$ with any $\inv$-complex equivalent to it. This will yield the same involutive homologies and correction terms.


\subsection{Connected sums of the Brieskorn sphere \texorpdfstring{$\Sigma(2,3,7)$}{Sigma(2,3,7)}  with itself}

\begin{figure} 
\includegraphics{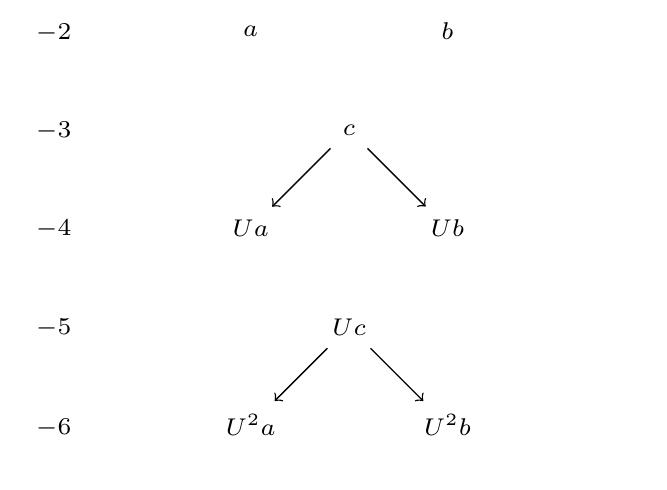}
\caption{The chain complex $\CFm(\Sigma(2,3,7))$. The homological gradings are shown on the left. The complex continues with further $U$-powers of the generators $a,b,c$.}
\label{fig:237}
\end{figure}

In this section, we compute the involutive correction terms of the connected sums $\#n(\Sigma(2,3,7))$ and prove Proposition \ref{prop:237sums}. As a warm-up, we compute the full involutive Heegaard Floer homology for $\Sigma(2,3,7)\#\Sigma(2,3,7)$.

\begin{figure}
\includegraphics[scale=.8]{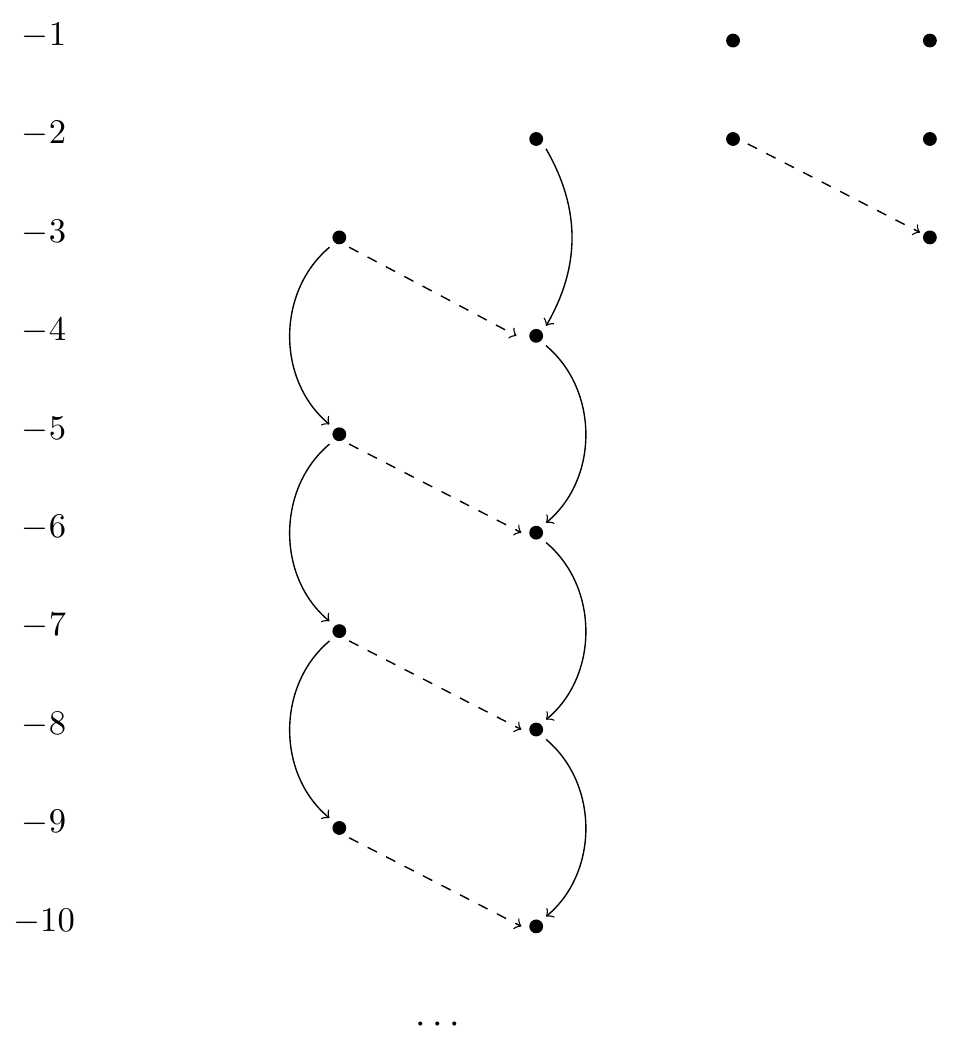}
\caption{A picture of $\HFIm(\Sigma(2,3,7)\# \Sigma(2,3,7))$. Curved solid arrows denote the action of the variable $U$, and dashed straight arrows denote the action of the variable $Q$.}
\label{fig:hfipic}
\end{figure}

\begin{lemma}\label{lem:Sigma(2,3,7)^2}
The involutive Heegaard Floer homology $\HFIm(\Sigma(2,3,7)\#\Sigma(2,3,7))$ is the module $\F_2[U]_{(-3)}\oplus \F_2[U]_{(-2)} \oplus (\F_2)^{\oplus 2}_{(-1)} \oplus (\F_2)_{(-2)}^{\oplus 2} \oplus (\F_2)_{(-3)}$ with $Q$ action as shown in Figure~\ref{fig:hfipic}. Therefore, the involutive correction terms of $\Sigma(2,3,7)\#\Sigma(2,3,7)$ are
\[ \dl(\Sigma(2,3,7)\#\Sigma(2,3,7))=-2, \qquad \du(\Sigma(2,3,7)\#\Sigma(2,3,7))= d(\Sigma(2,3,7)\#\Sigma(2,3,7)) = 0.\] 
\end{lemma}

\begin{proof} The involutive Heegaard Floer homology of $\Sigma(2,3,7)$ is computed in  \cite[Section 6.8]{HMinvolutive}, using the large surgery formula. In fact, by applying the slightly stronger version of that formula, Proposition~\ref{prop:surgeries2}, we can find a model for the chain complex $\CFp(\Sigma(2,3,7))$ together with the involution $\iota$. (By ``model'', we mean a chain complex homotopy equivalent to it, by equivalences that interchange the involutions, up to homotopy.) In light of Proposition \ref{prop:orientation}, we may dualize this to obtain a model for $\CFm(\Sigma(2,3,7))$ with its involution (also called $\iota$). This gives the chain complex in Figure \ref{fig:237}. Following the computation in \cite[Section 6.8]{HMinvolutive}, we see that
$$\HFIm(\Sigma(2,3,7)) \simeq \F_2[U]_{(-3)} \oplus \F_2[U]_{(-2)}\oplus (\F_2)_{(-1)}$$
where the $\F_2[U]$-summands are generated by $[Ua + Qc]$ and $[Qa]=[Qb]$, and the remaining copy of $\F_2$ is generated by $[a+b]$.

\begin{figure}
\includegraphics{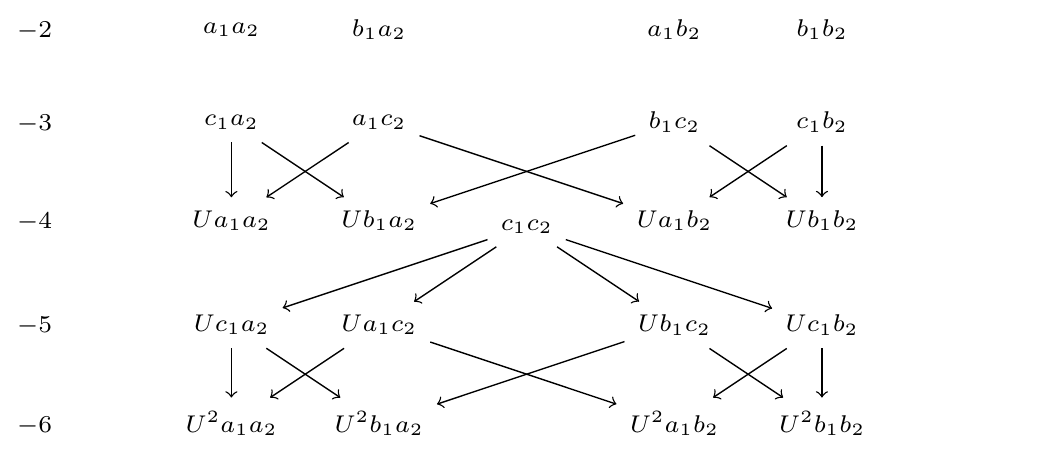}
\caption{The chain complex $\CFm(\Sigma(2,3,7)\#\Sigma(2,3,7))$. The complex continues with further $U$-powers of the nine generators.}
\label{fig:237twice}
\end{figure}

Now consider the connected sum $\Sigma(2,3,7)\#\Sigma(2,3,7)$. The chain complex $\CFm(\Sigma(2,3,7))$ $\otimes\CFm(\Sigma(2,3,7))[-2]$ appears in Figure \ref{fig:237twice}. By  Theorem~\ref{thm:ConnSum} we may take the involution to be $\iota\otimes\iota$, or reflection across the center line of the page in Figure \ref{fig:237twice}. Therefore the involution on $\CFm(\Sigma(2,3,7))\otimes\CFm(\Sigma(2,3,7))[-2]$ is given by
\begin{align*}
&\iota(a_1a_2) = b_1b_2 
\qquad \qquad\iota(a_1b_2) = b_1a_2 \\
&\iota(b_1a_2) = a_1b_2 
\qquad \qquad\iota(b_1b_2) = b_1b_2 \\
&\iota(a_1c_2) = b_1c_2 
\qquad \qquad\iota(b_1c_2) = a_1c_2 \\
&\iota(c_1a_2) = c_1b_2 
\qquad \qquad \iota(c_1b_2) = c_1a_2 \\
&\iota(c_1c_2) = c_1c_2
\end{align*}
We observe that the ordinary Heegaard Floer homology is $$\HFm(\Sigma(2,3,7)\#\Sigma(2,3,7)) = \F_2[U]_{(-2)}\oplus (\F_2)_{(-2)}^{\oplus 3} \oplus (\F_2)_{(-3)},$$ where $\F_2[U]_{(-2)}$ is generated by (say) $[a_1a_2]$, $(\F_2)_{(-2)}^{\oplus 3}$ is generated by $[a_1b_2+a_1a_2], [b_1a_2+a_1a_2]$, and $[b_1b_2+a_1a_2]$, and $(\F_2)_{(-3)}$ is generated by $[c_1a_2+c_1b_2+a_1c_2+b_1c_2]$. Notice that we have made a slightly arbitrary choice of generator for the $U$-tail; we could have chosen any of the elements $[a_1a_2]$, $[a_1b_2]$, $[b_1a_2]$, or $[b_1b_2]$, all of which are identified after multiplication with $U$, to wit: $U[a_1a_2]=U[a_1b_2]=U[b_1a_2]=U[b_1b_2]$. The involutive Heegaard Floer homology of this complex has differentials
\begin{align*}
&\delinv(a_1a_2)= \delinv(b_1b_2) = Q(a_1a_2+b_1b_2) \\
&\delinv(a_1b_2) = \delinv(b_1a_2) = Q(a_1b_2 + b_1a_2) \\
&\delinv(c_1a_2) = U(a_1a_2 + b_1a_2) + Q(c_1a_2+c_1b_2) \\
&\delinv(c_1b_2) = U(a_1b_2 + b_1b_2) + Q(c_1a_2+c_1b_2) \\
&\delinv(a_1c_2) =U(a_1a_2+a_1b_2) + Q(a_1c_2 + b_1c_2) \\
&\delinv(b_1c_2) = U(b_1a_2+b_1b_2)+ Q(a_1c_2+ b_1c_2) \\
&\delinv(c_1c_2) = U(c_1a_2 + c_1b_2 + a_1c_2 + b_1c_2) \\
&\delinv(Qa_1a_2) = \delinv(Qb_1b_2) = \delinv(Qa_1b_2) = \delinv(Qb_1a_2)=0\\
&\delinv(Qc_1a_2) = UQ(a_1a_2 + b_1a_2)\\
&\delinv(Qc_1b_2) = UQ(a_1b_2 + b_1b_2) \\
&\delinv(Qa_1c_2) =UQ(a_1a_2+a_1b_2)  \\
&\delinv(Qb_1c_2) = UQ(b_1a_2+b_1b_2) \\
&\delinv(Qc_1c_2) = UQ(c_1a_2 + c_1b_2 + a_1c_2 + b_1c_2) 
\end{align*}%
A straightforward computation now shows that the homology of this complex is $\F_2[U]_{(-3)}\oplus \F_2[U]_{(-2)} \oplus (\F_2)^{\oplus 2}_{(-1)} \oplus (\F_2)_{(-2)}^{\oplus 2} \oplus (\F_2)_{(-3)}$, where the tail $\F_2[U]_{(-3)}$ is generated by $[Ua_1a_2 +Qa_1c_2+Qc_1b_2]$, the tail $\F_2[U]_{(-2)}$ is generated by $[Qa_1a_2]$, the summand $(\F_2)_{(-1)}^{\oplus 2}$ is generated by $[a_1a_2+b_1b_2]$ and $[a_1b_2+b_1a_2]$, the summand $(\F_2)_{(-2)}^{\oplus 2}$ is generated by $[(c_1a_2+c_1b_2+a_1c_2+a_1b_2)]$ and $[Q(a_1b_2+a_1a_2)]$, and the summand $(\F_2)_{-3}$ is generated by $[Q(c_1a_2+c_1b_2+a_1c_2+a_1b_2)]$. This gives the $\F_2[U,Q]/(Q^2)$-module structure shown in Figure \ref{fig:hfipic}. We observe that $\dl(\Sigma(2,3,7)\#\Sigma(2,3,7))= -3+1=-2$ and $\du(\Sigma(2,3,7)\#\Sigma(2,3,7))=-2+2=0$. \end{proof}

Notice that if we were only interested in the involutive correction terms, instead of computing the entire involutive Heegaard Floer homology $\HFm(\Sigma(2,3,7)\#\Sigma(2,3,7))$, we could instead have given the following argument using Lemma \ref{lemma:reformulation}. The homogeneous elements $v$ of $\CFm(\Sigma(2,3,7))\otimes\CFm(\Sigma(2,3,7))[-2]$ with the property that $\del v = 0$ and $[U^nv]\neq 0$ for any $n$ are exactly the $U$-powers of the four elements $a_1a_2$, $a_1b_2$, $b_1a_2$, and $b_1b_2$, and sums of any three of these terms. Therefore, in order to compute $\dl(\Sigma(2,3,7)\#\Sigma(2,3,7))$, we need only find the smallest $U$-power of one of these eight generators for which $(1+\iota)(U^kv)$ is in the image of the ordinary differential $\del$. Now, the involution $\iota$ exchanges $a_1a_2$ with $b_1b_2$ and $a_1b_2$ with $b_1a_2$.  Therefore, $(1+\iota)$ applied to any of the eight possible generators of the tower is either $a_1a_2+b_1b_2$ or $a_1b_2 + b_1a_2$, neither of which is in the image of $\del$. However, $(1+\iota)(Ua_1a_2) = U(a_1a_2+b_1b_2) = \del(c_1a_2 + b_1c_2)$. We conclude that $\dl(\Sigma(2,3,7)\#\Sigma(2,3,7))= \gr(Ua_1a_2)+2 = -4+2=-2$. Now, in order to find $\du(\Sigma(2,3,7)\#\Sigma(2,3,7))$, observe that if $x=z=0$ and $y=a_1a_2$, we have $\del y = (1+\iota) x$, $\del z = x$, and $[U^n(y+(1+\iota)z)]\neq 0$ for any $n\geq 0$. Since the grading of $y$ is the highest grading in the entire complex, we conclude that $\du(\Sigma(2,3,7)\#\Sigma(2,3,7))=\gr(y)+2=-2+2=0$. 

The argument in the preceding paragraph generalizes to the case of $\#n\Sigma(2,3,7)$, allowing us to prove Proposition \ref{prop:237sums}.

\begin{proof}[Proof of Proposition \ref{prop:237sums}] Given $n$ copies of the complex $\CFm(\Sigma(2,3,7))$, index by subscripts as in the preceding example, so that the $i$th copy of the complex is generated over $\F_2[U]$ by $a_i, b_i, c_i$, and the copy of the chain complex $\CFm(\#n\Sigma(2,3,7))$ given by tensoring all $n$ copies of the original complex together is generated over $\F_2[U]$ by ordered $n$-tuples consisting of one of these three elements for each $i$. In this complex, the homogeneous elements $v$ of $\CFm(\#n\Sigma(2,3,7)))$ with the property that $\del v = 0$ and $[U^nv]\neq 0$ for any $n$ are exactly the $U$-powers of elements $v = v_1\cdots v_n$ where each $v_i$ is either $a_i$ or $b_i$, and sums of an odd number of such terms. These elements are in grading $-2$. The involution $\iota$ acts on such elements $v$ by replacing $a_i$ with $b_i$ and vice versa, and therefore exchanges these elements in pairs. Therefore, the map $(1+\iota)$ applied to a sum of an odd number of such terms is nonzero, and in particular not in the image of $\del$ (since no nonzero element in grading $-2$ in this complex is in the image of $\del$). However, we see that $(1+\iota)(Ua_1\cdots a_n) = U(a_1\cdots a_n+b_1\cdots b_n) = \del(\sum_{i=1}^n b_1\cdots b_{i-1} c_i a_{i+1}\cdots a_n)$. We conclude that $\dl(\#n\Sigma(2,3,7))= \gr(Ua_1 \cdots a_n)+2 = -4+2=-2$.

Similarly, to compute $\du(\#n\Sigma(2,3,7))$, observe that if $x=z=0$ and $y=a_1\cdots a_n$, then we have $\del y = (1+\iota) x$, $\del z = x$, and $[U^n(y+(1+\iota)z)]\neq 0$ for any $n\geq 0$. Since the grading of $y$ is the highest grading in the entire complex, we conclude that $\du(\#n\Sigma(2,3,7))=\gr(y)+2=-2+2=0$. \end{proof}

\subsection{A rational homology sphere with \texorpdfstring{$\dl \neq d \neq \du$}{dl not d not du}}

\begin{figure}
\includegraphics{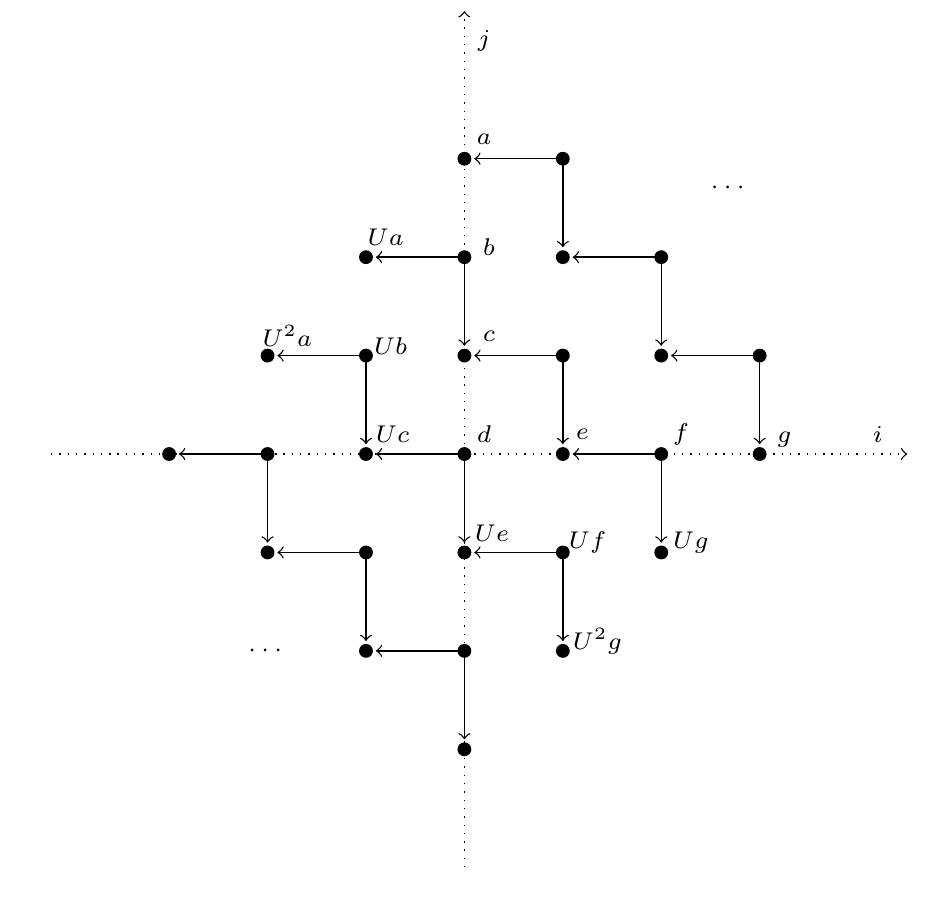}
\caption{A model for $\CFKinfty(T_{2,7})$.}
\label{fig:T27}
\end{figure}

In this section we use the connected sum formula to compute the involutive correction terms of the unique spin structure on $Y = S^3_{-3}(T_{2,7}) \# S^3_{-3}(T_{2,7}) \# S^3_{5}(-T_{2,11})$, and thus prove Proposition \ref{prop:alldifferent}. We also give the proof of Proposition~\ref{prop:undetermined}. Throughout this section, all three-manifolds that appear are $\Z_2$ homology spheres, and hence they admit with unique spin structures; therefore, as per Remark \ref{remark:suppress} we omit all references to these spin structures in the notation.

We begin by looking at $S^3_{-3}(T_{2,7})$. By Ni and Wu's computation of the  ordinary correction terms \cite[Proposition 1.6]{NiWu}, we see that
\begin{align*}
d(S^3_{-3}(T_{2,7})) &= -d(S^3_3(-T_{2,7})) \\
&= -d(L(3,1) +2V_0(-T_{2,7}))\\ 
&= -\left(\frac{3-1}{4} - 0\right) \\ 
&= -\frac{1}{2}.
\end{align*}

The chain complex $\CFKinfty(T_{2,7})$ appears in Figure \ref{fig:T27}.  Recall from \cite[Section 7]{HMinvolutive} that the map $\iota_K$ on $\CFKinfty(T_{2,7})$ is reflection across the line $y=x$. Next, recall that the complex $\CFKinfty(-T_{2,7})$ is dual to $\CFKinfty(T_{2,7})$. Therefore we can obtain $\CFKinfty(T_{2,7})$ by reversing all arrows and gradings in Figure \ref{fig:T27}, and the involution $\iota_K$ will continue to be a reflection across $y=x$ (by the same computation as in \cite{HMinvolutive}).

Now, since $3=g(T_{2,7})=g(-T_{2,7})$, we see that by Proposition~\ref{prop:surgeries2}, the Heegaard Floer complex $\CFp(S^3_3(-T_{2,7}))$ is isomorphic as a chain complex to $A_0^+(-T_{2,7})$, and the involution $\iota$ can be taken to be $\iota_K$. Furthermore, $A_0^+(-T_{2,7})$ is dual to the subcomplex $C\{i\leq 0 \text{ or } j\leq 0\}$ of $\CFKinfty(T_{2,7})$. But recall that in fact $\CFm(S^3_{-3}(T_{2,7}))$ is dual to $\CFp(S^3_3(-T_{2,7}))$, and by Proposition \ref{prop:orientation}, if we know the conjugation map $\iota$ on a representative of $\CFp(S^3_3(-T_{2,7}))$, then the conjugation map on $\CFm(S^3_{-3}(T_{2,7}))$ may be taken to be the dual map of $\iota$ (also called $\iota$). Therefore the chain homotopy class of $\CFm(S^3_{-3}(T_{2,7}))$ may be represented by the subcomplex of $\CFKinfty(T_{2,7})$ consisting of the second, third, and fourth quadrants, with involution consisting of reflection across the line $y=x$.

\begin{figure}
\includegraphics{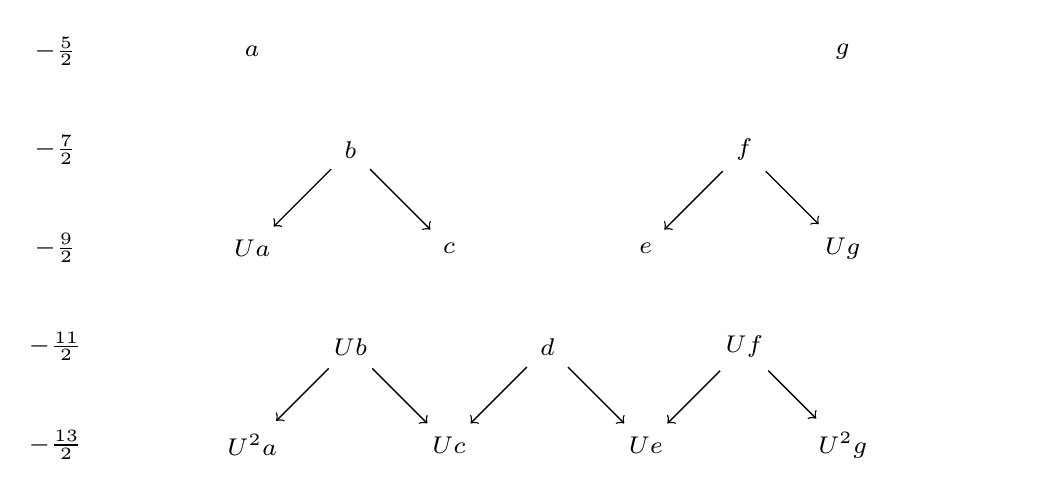}
\caption{The chain complex $\CFm(S^3_{-3}(T_{2,7}))$ arising from the surgery formula. The complex continues with further $U$-powers of the seven generators.}
\label{fig:2715}
\end{figure}

The preceding paragraph implies that $\CFm(S^3_{-3}(T_{2,7}))$ is the complex shown in Figure \ref{fig:2715}. The involution on the complex is horizontal reflection across the centerline of the page. We may substantially reduce the size of this complex as follows. First, we do a $U$-equivariant change of basis in which $c$ is replaced with $c+Ua$, $e$ is replaced with $e+Ug$, and $d$ is replaced with $d'=d+Ub+Uf$. Notice that this change of basis is $\iota$-equivariant. This leaves us with the complex in Figure \ref{fig:2715complexbasischange}. Deleting the acyclic summands, we obtain the complex in Figure \ref{fig:2715complexreduced}. Note that these changes are equivalences of $\inv$-complexes, so do not affect the final result.

\begin{figure}
\includegraphics{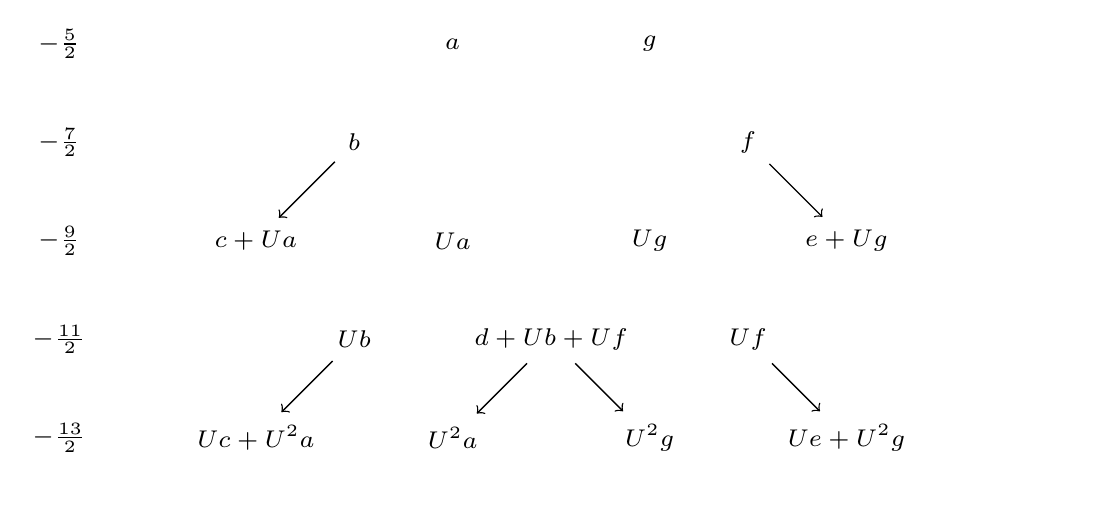}
\caption{A change of basis for $\CFm(S^3_{-3}(T_{2,7}))$.}
\label{fig:2715complexbasischange}
\end{figure}

\begin{figure}
\includegraphics{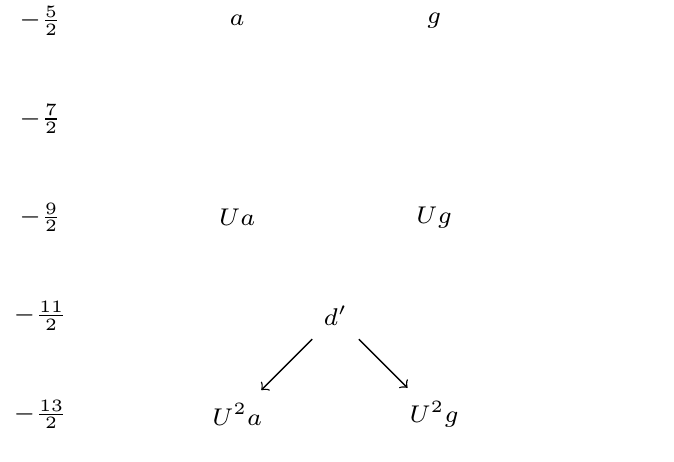}
\caption{The chain complex $\CFm(S^3_{-3}(T_{2,7}))$ with acyclic summands removed, and with $d' = d+Ub+Uf$.}
\label{fig:2715complexreduced}
\end{figure}
 
We now compute the correction terms of $S^3_{-3}(T_{2,7})$. This computation will not be used in the proof of Proposition~\ref{prop:alldifferent}, but it will be needed later in the proof of Proposition \ref{prop:undetermined}. We include it at this point in the narrative to build the reader's familiarity with arguments using Lemma \ref{lemma:reformulation}.

\begin{lemma} The involutive correction terms of $S^3_{-3}(T_{2,7})$ are 
\[ \dl(S^3_{-3}(T_{2,7}))=-\frac{9}{2}, \qquad \du(S^3_{-3}(T_{2,7}))=d(S^3_{-3}(T_{2,7}))=-\frac{1}{2}.\]
\end{lemma}

\begin{proof} Consider the copy of $\CFm(S^3_3(T_{2,7}))$ in Figure \ref{fig:2715complexreduced}, with involution given as usual by reflection across the page. We see that the $\F_2[U]$ summand of $\HFm(S^3_{-3}(T_{2,7}))$ is generated by either $[a]$ or $[g]$, which become equal after multiplication by $U$. Therefore $\gr(a)=\gr(g)=-\frac{5}{2}$ and $\gr(d')=-\frac{11}{2}$. Now, the homogeneous elements $v$ for which $\del v=0$ and $[U^nv]\neq 0$ for any $n$ are exactly the $U$-powers of $a$ and $g$.  We observe that $(1+\iota)a=(1+\iota)g=a+g$ and $(1+\iota)(Ua)= (1+\iota)Ug = U(a+g)$ are not in the image of $\del$, but $(1+\iota)U^2a=U^2(a+g) = \del(d')$. Therefore we conclude that $\dl(S^3_{-3}(T_{2,7}))=\gr(U^2a)+2=-\frac{13}{2}+2 = -\frac{9}{2}$. As to the upper involutive correction term, notice that if $x=z=0$ and $y=a$, the conditions of Lemma \ref{lemma:reformulation} are satisfied. So $\du(S^3_{-3}(T_{2,7}))=\gr(a)+2 = -\frac{1}{2}$. \end{proof}

Now we turn our attention to $S^3_{-3}(T_{2,7})\#S^3_{-3}(T_{2,7})$. We let the two copies of the chain complex $\CFm(S^3_{-3}(T_{2,7}))$ be labelled as in Figures \ref{fig:2715complexreduced} and \ref{fig:2715complexreducedsecondcopy}, for convenience. Tensoring these complexes together, we obtain the chain complex shown in Figure \ref{fig:2715twocopies}. We can now compute the involutive correction terms of this connected sum.

\begin{figure}
\includegraphics{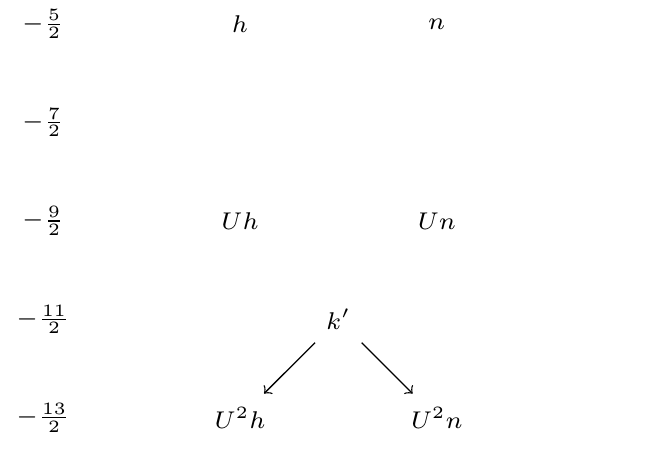}
\caption{A second copy of the chain complex $\CFm(S^3_{-3}(T_{2,7}))$, with different notation for convenience.}
\label{fig:2715complexreducedsecondcopy}
\end{figure}

\begin{figure}
\includegraphics{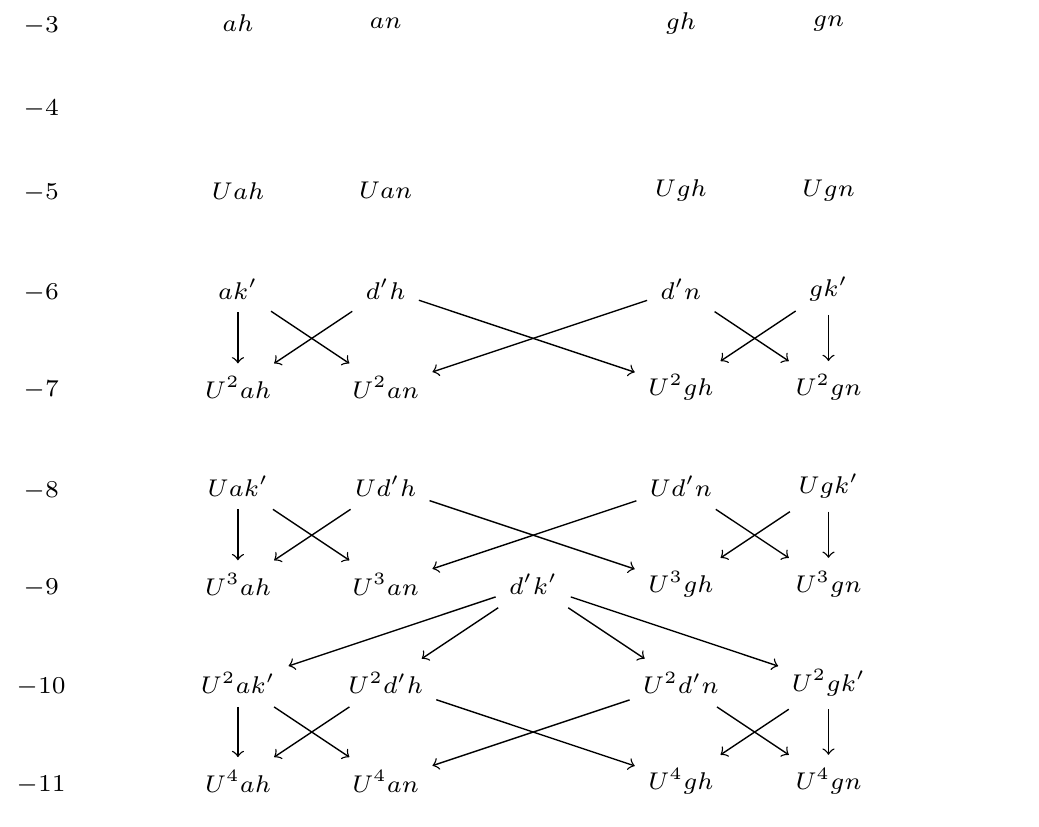}
\caption{The chain complex $\CFKm(S^3_{-3}(T_{2,7})\#S^3_{-3}(T_{2,7}))$. The complex continues with additional $U$-powers of the nine generators.}
\label{fig:2715twocopies}
\end{figure}

\begin{lemma} \label{lem:twocopies} The involutive correction terms of $Y_1 = S^3_{-3}(T_{2,7})\#S^3_{-3}(T_{2,7})$ are 
\begin{align*}&\dl(Y_1) = -5 \\ &\du(Y_1) = d(Y_1) = -1. \end{align*}
\end{lemma}

\begin{proof} Since the ordinary correction term is additive, $d(S^3_{-3}(T_{2,7})\#S^3_{-3}(T_{2,7}))=-\frac{1}{2}+-\frac{1}{2}=-1$. Consider the copy of $\CFm(S^3_{-3}(T_{2,7})\#S^3_{-3}(T_{2,7}))$ in Figure \ref{fig:2715twocopies}. The involution on the tensor product may be taken to be the tensor product of the involutions, or reflection across the centerline of the page. Notice that the $\F_2[U]$ summand in the homology of this complex is generated by any of $[ah]$, $[an]$, $[gh]$, and $[gn]$, so $\gr(ah)=\gr(an)=\gr(gh)=\gr(gn) = -3$. Therefore, the homogeneous elements $v$ with the property that $\del(v)=0$ and $[U^nv]\neq 0$ for any $n$ are exactly the $U$-powers of $ah$, $an$, $gh$, and $gn$, and sums of three of these terms. We observe that $1+\iota$ applied to any of these eight generators is either $ah+gn$ or $an+gh$, neither of which is in the image of $\del$. Similarly, $(1+\iota)$ applied to $U$ times any of these generators is either $U(ah+gn)$ or $U(an+gh)$, neither of which is in the image of $\del$. However, we have $(1+\iota)(U^2ah) = U^2ah+U^2gn = \del(ak'+d'n)$. Therefore \[\dl(S^3_{-3}(T_{2,7})\#S^3_{-3}(T_{2,7})) = \gr(U^2ah)+2 = -7+2=-5.\] 

Now we turn our attention to the upper involutive correction term. Notice that if $x=z=0$ and $y=ah$, the conditions of Lemma \ref{lemma:reformulation} are satisfied. Therefore \[\du(S^3_{-3}(T_{2,7})\#S^3_{-3}(T_{2,7}))=\gr(ah)+2=-1.\qedhere \] \end{proof}

\begin{figure}
\includegraphics{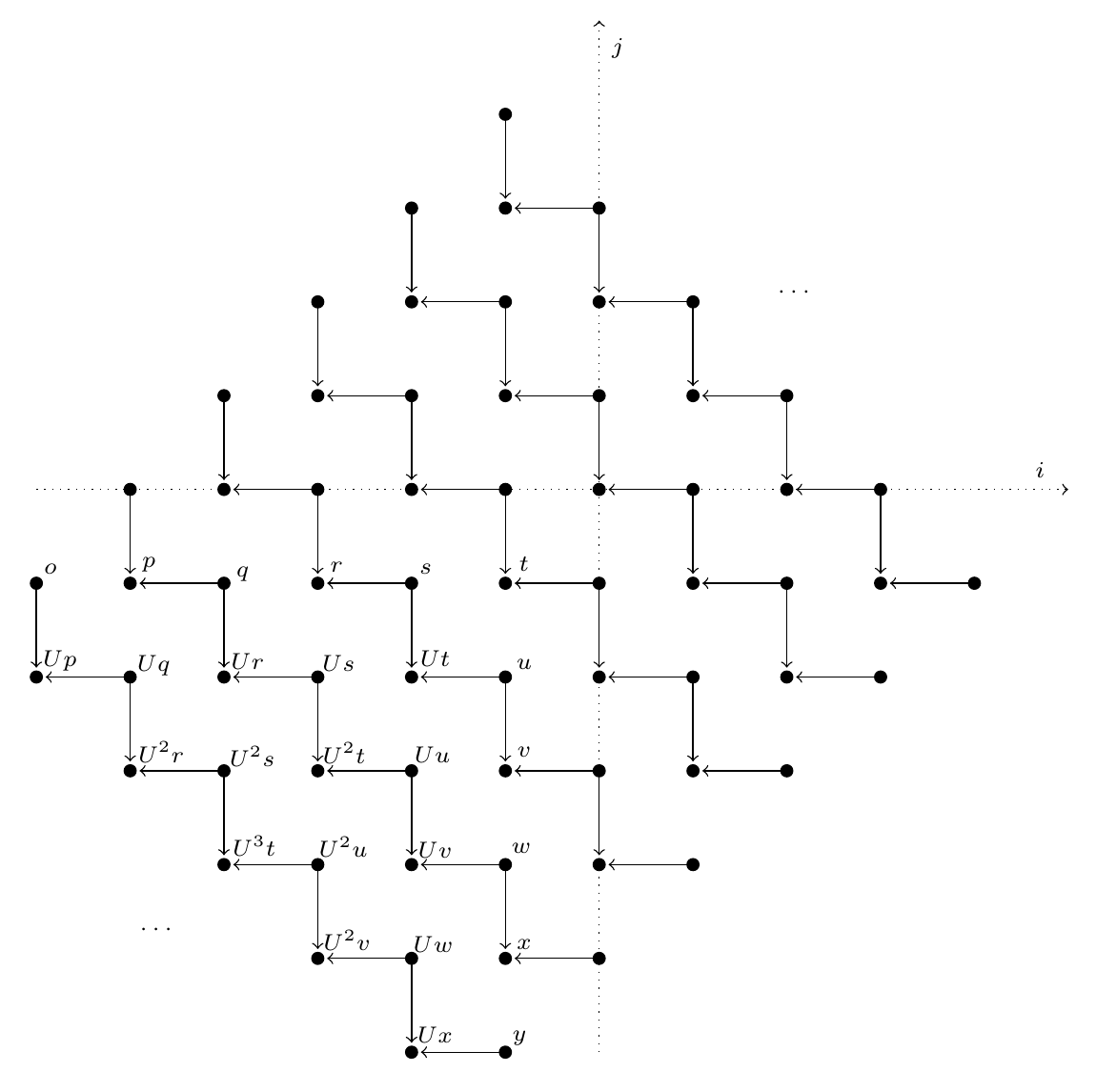}
\caption{A model for $\CFKinfty(-T_{2,11})$.}
\label{fig:-T211}
\end{figure}

We now consider the three-manifold $Z=S^3_5(-T_{2,11})$. By Ni and Wu's computation of the ordinary correction terms of surgeries \cite[Proposition 1.6]{NiWu} , we see that
\[
d(S^3_{5}(-T_{2,11})) = d(L(5,1) - 2V_0(-T_{2,11})) = \left(\frac{5-1}{4} - 0\right) = 1.
\]

The chain complex $\CFKinfty(-T_{2,11})$ is shown in Figure \ref{fig:-T211}. Recall from \cite[Section 7]{HMinvolutive} that the involution is a reflection across the line $y=x$. Now, since $5=g(-T_{2,11})$, there is an isomorphism of chain complexes between the subcomplex $A_0^-(-T_{2,11}) \subset \CFKinfty(-T_{2,11})$ and $\CFm(S^3_5(-T_{2,11}))$, and the involution $\iota$ may be taken to be the map $\iota_K$, which is a reflection across the line $y=x$. (Note that the large surgery formula, Proposition 6.8 in \cite{HMinvolutive}, and its strengthening, Proposition~\ref{prop:surgeries2}, were phrased for $\CFp$. However, the same arguments carry over for $\CFm$ essentially unchanged.) It follows that $\CFm(S^3_{5}(-T_{2,11}))$ may be represented by the complex in Figure \ref{fig:21123complex}, with involution given by horizontal reflection. As previously, we may reduce the size of this complex via a change of basis which is equivariant with respect to both $U$ and $\iota$. We  replace $r$ with $r+Ut$, $v$ with $v+Ut$, $q$ with $q+Us$, $y$ with $y+Uu$, $p$ with $p+Ur+U^2t$, $x$ with $x+Uv+U^2t$, $o$ with $o'=o+Uq+U^2s$, and $y$ with $y'=y+Uw+U^2r$. This gives the complex in Figure \ref{fig:21123complexbasischange}. Cancelling acyclic summands, we obtain the chain complex in Figure \ref{fig:21123complexreduced}. Let us now compute the involutive correction terms of $S^3_{5}(-T_{2,11})$.

\begin{figure}
\includegraphics{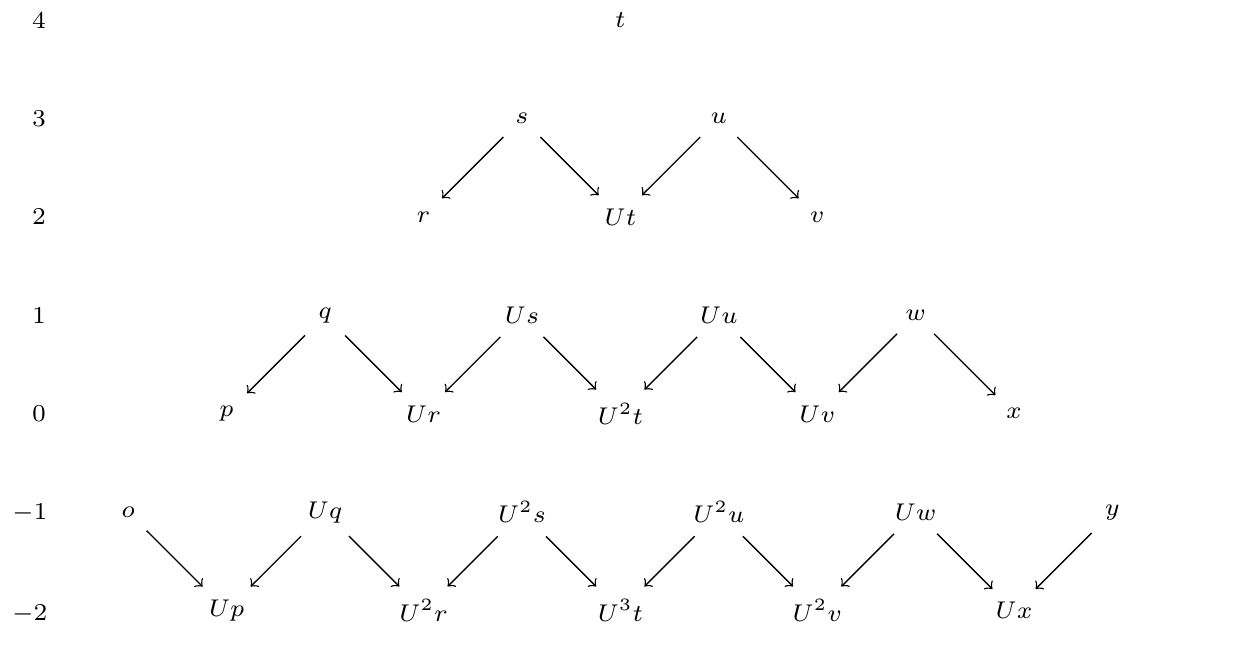}
\caption{The chain complex $\CFm(S^3_{5}(-T_{2,11}))$ arising from the surgery formula. The complex continues with additional $U$-powers of the eleven generators.}
\label{fig:21123complex}
\end{figure}

\begin{figure}
\includegraphics{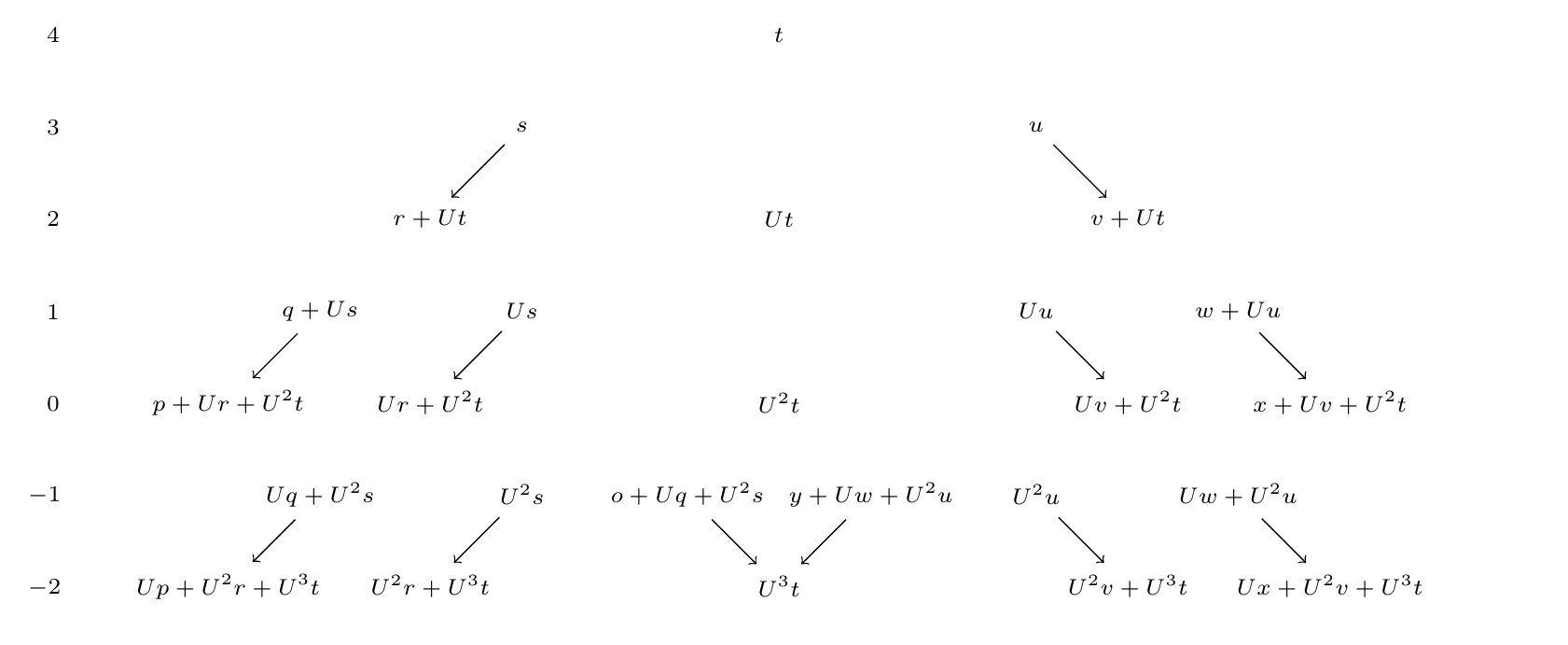}
\caption{A change of basis for $\CFm(S^3_{5}(-T_{2,11}))$.}
\label{fig:21123complexbasischange}
\end{figure}

\begin{lemma} The involutive correction terms of $Z=S^3_5(-T_{2,11})$ are
\[ \dl(Z) = d(Z) = 1 \qquad \du(Z) = 7.\]
\end{lemma}

\begin{proof} Consider the copy of $\CFm(S^3_5(-T_{2,11}))$ appearing in Figure \ref{fig:21123complexreduced}, with involution given by reflection across the centerline of the page as usual. We observe that the $\F_2[U]$ summand in $\HFm(S^3_{5}(-T_{2,11}))$ is generated by $[o'+y']$. Since $d(S^3_5(-T_{2,11}))=1$, we must have $\gr(o')=\gr(y')=-1$. Indeed, the homogenous elements $v$ of $\CFm(S^3_5(-T_{2,11}))$ with the property that $\del v =0$ and $[U^nv]\neq 0$ for all $n$ are exactly the $U$-powers of $o'+y'$. Furthermore, since $(1+\iota)(o'+y')=0$ is in the image of $\del$, by Lemma \ref{lemma:reformulation}, we see that $\dl(S^3_5(-T_{2,11})) =\gr(o'+y')+2 = 1$. As to the upper involutive correction term, observe that if $x=t$, $y=0$, and $z=o'$, we have $\del y=0=(1+\iota)x$, $\del z = U^3x$, and $[U^n(U^3y + (1+\iota)z)] = [U^n(o'+y')]\neq 0$ for any $n$. Therefore, since $t$ is the maximally graded element in the complex, we see that $\du(S^3_5(-T_{2,11})) = \gr(t)+3 = 4+3=7$. \end{proof}

\begin{figure}
\includegraphics{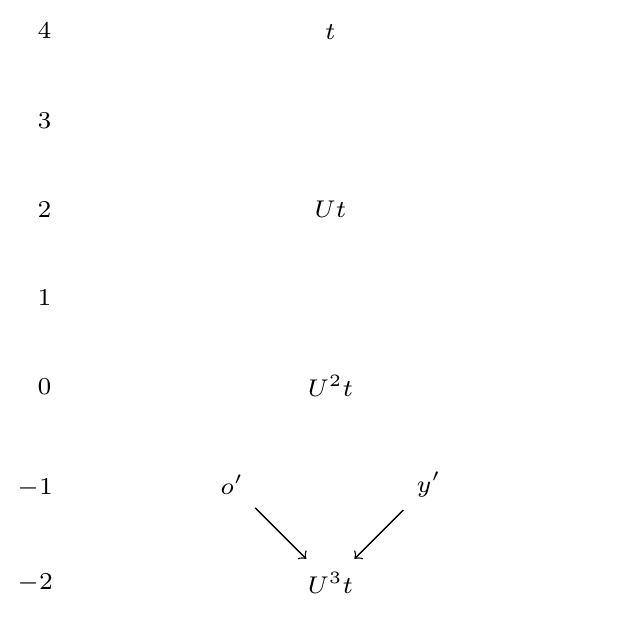}
\caption{The complex $\CFm(S^3_{5}(-T_{2,11}))$ with acyclic summands removed, $o' = o+Uq+U^2s$, and $y'=y+Uw+U^2r$.}
\label{fig:21123complexreduced}
\end{figure}

\begin{figure}
\includegraphics[scale=.85]{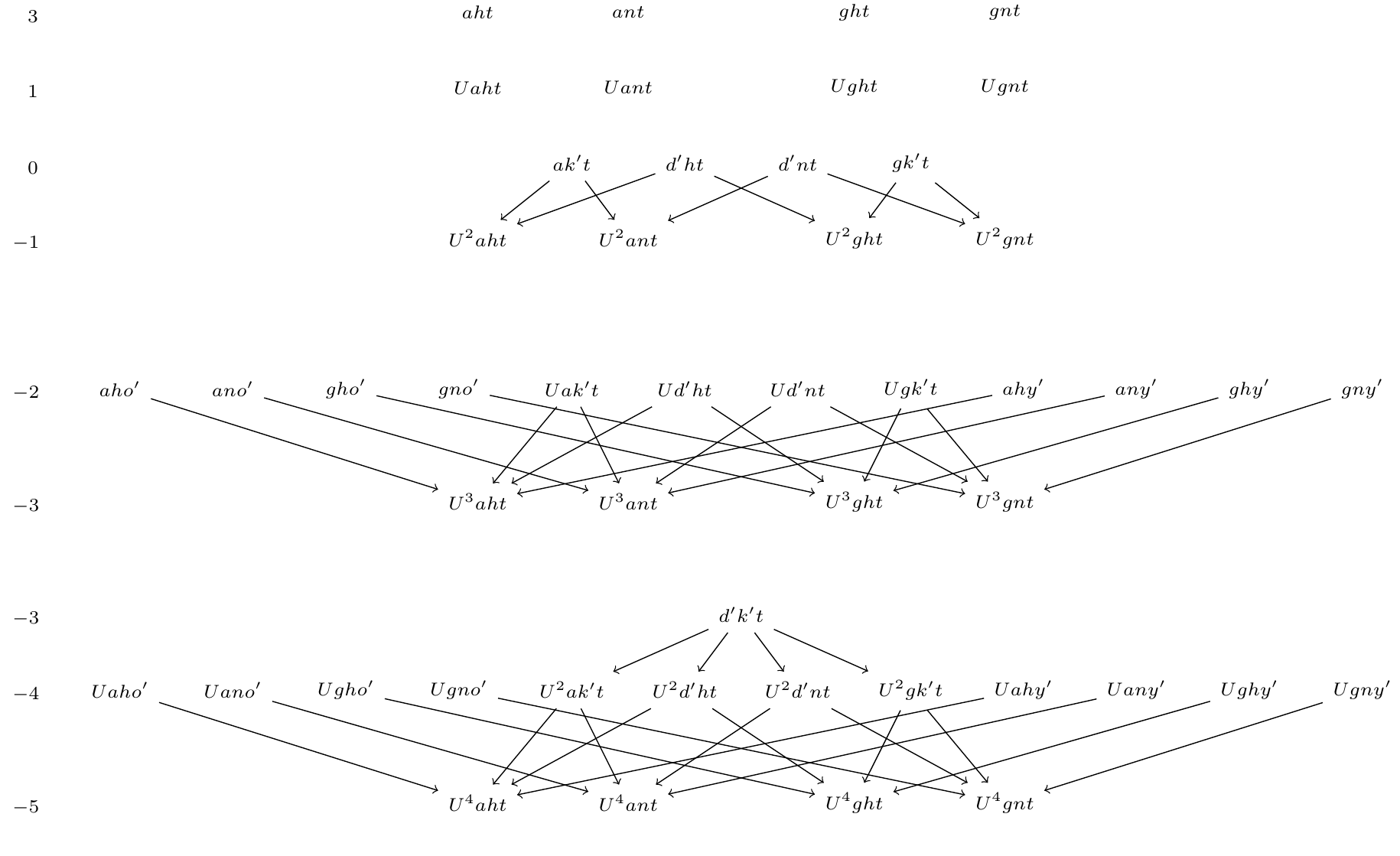}
\vskip 2mm
\includegraphics[scale=.85]{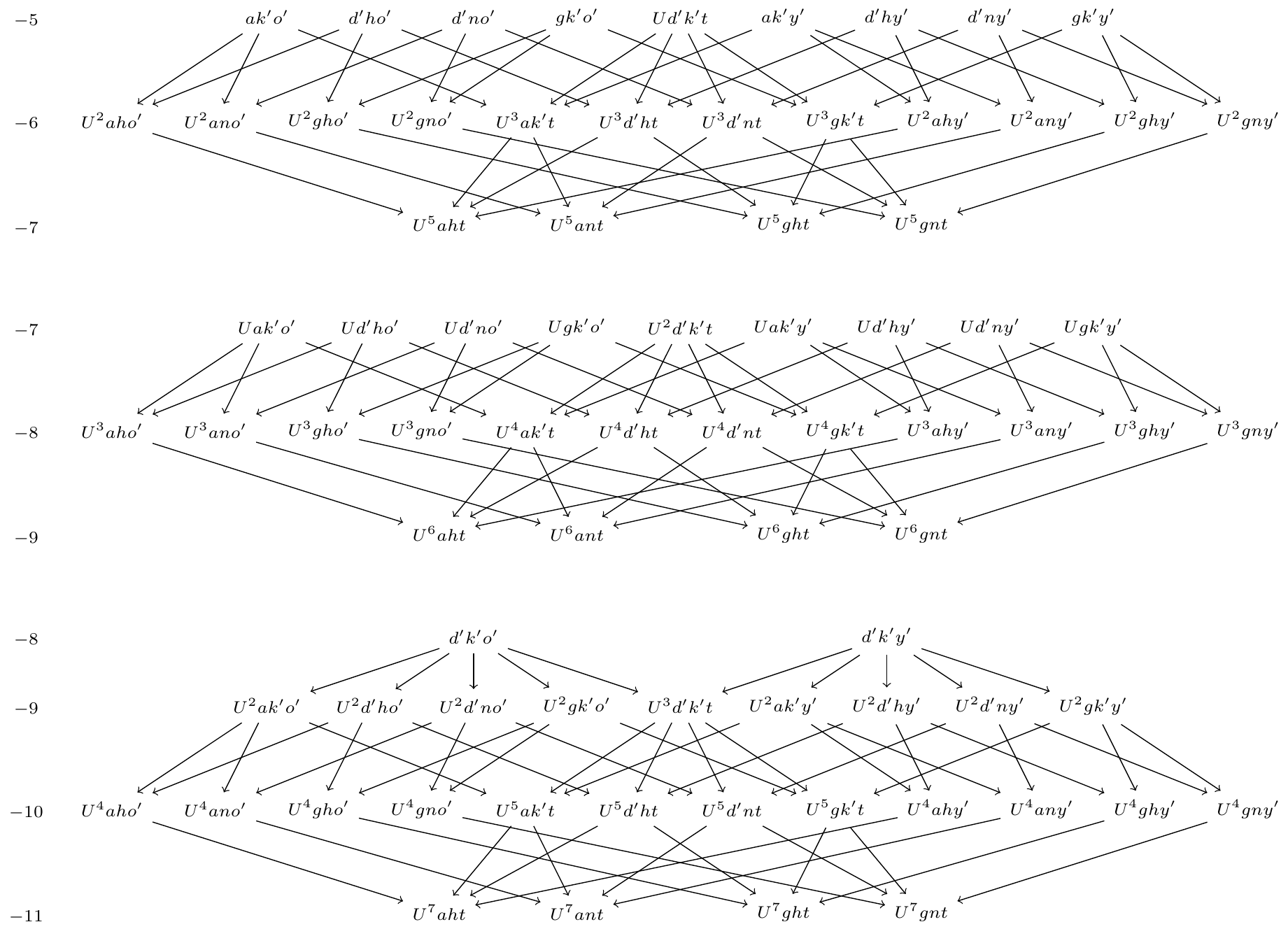}
\caption{The complex $\CFm(S^3_{-3}(T_{2,7})\#S^3_{-3}(T_{2,7})\#S^3_{5}(-T_{2,11}))$, decomposed into pieces that are not connected by differentials.}
\label{fig:totalcomplex}
\end{figure}

Now we are ready to prove Proposition \ref{prop:alldifferent}, verifying that the three-manifold $Y = S^3_{-3}(T_{2,7})\#S^3_{-3}(T_{2,7})\#S^3_5(-T_{2,11})$ has $\dl \neq d\neq \du$ in its unique spin structure.

\begin{proof}[Proof of Proposition \ref{prop:alldifferent}] Consider the three-manifold $Y = S^3_{-3}(T_{2,7})\#S^3_{-3}(T_{2,7})\#S^3_5(-T_{2,11})$. By additivity of the ordinary correction term, we see that $d(Y) = 0$. By tensoring the complexes in Figures \ref{fig:2715twocopies} and \ref{fig:21123complexreduced} over $\F_2[U]$, we obtain the model for $\CFm(Y)$ in Figure \ref{fig:totalcomplex}. The involution on the tensor product in Figure \ref{fig:totalcomplex} may be taken to be the tensor product of the involutions on the complexes in Figures~\ref{fig:2715twocopies} and \ref{fig:21123complexreduced}, which in particular is reflection across the centerline of the page. We now use Lemma \ref{lemma:reformulation} to  give a computation of the involutive correction terms. We begin by computing the lower involutive correction term. The $\F_2[U]$ summand in $\HFm(Y)$ is generated by the homology classes of any of the homogeneous elements $aho'+ahy'$, $ano'+any'$, $gho'+ghy'$, $gno'+gny'$, and furthermore by any sum of three of these elements. It is also generated by any sum of one of these eight generators with a subset of $aho'+ano'+Uak't$, $gho'+gno'+Ugk't$, $aho'+gho'+Ud'ht$, and $ano'+gno'+Ud'nt$, all of which become trivial in homology after multiplication by $U^2$. Suppose that $v$ is any of these possible generators. Then $(1+\iota) v$ is equal to the sum of one of $aho'+ahy'+gny'+gny'$ or $ano'+any'+gho'+ghy'$ with some subset of $aho'+ano'+Uak't + gho'+gno'+Ugk't$ and $aho'+gho'+Ud'ht +ano'+gno'+Ud'nt$. None of these possible elements $(1+\iota)v$ is in the image of $\del$.

However, let 
\begin{align*}
v  &= U\bigl( (gho'+ghy') + (aho'+ano'+Uak't) + (aho'+gho'+Ud'ht) \bigr) \\
&= Uano' +  Ughy' + U^2 ak't + U^2 d'ht
\end{align*}
 and $w = d'k't$. Then $\del v = 0$, $[U^nv]\neq 0$ for all $n$, and furthermore,
\begin{align*}
\del(w)&= U^2 ak't + U^2 d'ht + U^2 d'nt + U^2 gk't\\
       &=(1+\iota)v.
\end{align*}

\noindent We conclude that $\dl(Y)= \gr(v)+2 = -4+2=-2$.

Now we turn our attention to the upper involutive correction term. Let $x=U^2aht$, $z=aho'$, and $y=d'ht+gk't$. Notice that $\gr(x)=-1$ and $\gr(y)=0$. Then we have $\del y = U^2(aht + gnt)= (1+\iota)x$ and $\del z = U^3(aht)=Ux$. Furthermore, we claim \[[U^n(Uy+(1+\iota)z)]= [U^n(U(d'ht+gk't) + aho'+gny')]\neq 0\] for all $n$. We will show this by proving that for $n\geq 2$, this element is equivalent to $[U^n(aho'+ahy')]$, all $U$-powers of which are nonzero (for example, because $a$, $h$, and $o'+y'$ are generators of the $\F_2[U]$-tail in their original chain complexes). Indeed, we compute:
\begin{align*}
[U^3(d'ht +gk't) + U^2(aho'+gny')] &= [U^3(d'ht +gk't) + U^2(aho'+gny')]\\  & \qquad +[\del(d'hy')] + [\del(gk'y')]\\
&=[U^3(d'ht +gk't) + U^2(aho'+gny')]\\ &  \qquad +[U^3d'ht+U^2ahy'+U^2ghy'] \\ &  \qquad + [U^3gk't+U^2ghy'+U^2gny']\\
&=[U^2(aho'+ahy')]
\end{align*}

This shows that $\du(Y)$ is at least the grading of $U^2aht$ plus $3$. It remains to be shown that $x,y,z$ are the maximally graded set of elements in $\CFm(Y)$ satisfying the requirements of Lemma \ref{lemma:reformulation}. Suppose there is some other $\wt{x}, \wt{y},\wt{z}$ with at least one of $\wt{x}$ and $\wt{y}$ nonzero such that $\del(\wt{x})=(1+\iota)\wt{y}$, $\del \wt{z} = U^m\wt{x}$ for some $m\geq 0$, and $[U^n(U^m\wt{y}+ (1+\iota)\wt{z})]\neq 0$ for all $n\geq 0$, with the property that either $\gr(\wt{x}) > -1$ or $\gr(\wt{y}) > 0$. The condition $[U^n(U^m\wt{y}+ (1+\iota)\wt{z})]\neq 0$ for all $n$ implies that the gradings of $\wt{y}$ and $\wt{z}$ are even. But there are no nonempty even gradings greater than zero in this chain complex, so we conclude that $\wt{y}=0$. Then there must exist $\wt{z}$ such that $[U^n(1+\iota){\wt{z}}]\neq 0$ for any $n\geq 0$. But we have already determined which homogeneous elements of the chain complex generate the $\F_2[U]$-tail in $\HFm(Y)$, and none of them is in the image of $(1+\iota)$. Therefore this is impossible. We conclude that our previous $x,y,z$ are maximally graded homogeneous elements satisfying the requirements of Lemma \ref{lemma:reformulation}, and we see that $\dl(Y) = \gr(U^2aht)+3=-1+3=2$. \end{proof}

Finally, we turn our attention to Proposition~\ref{prop:undetermined}.

\begin{figure}
\includegraphics{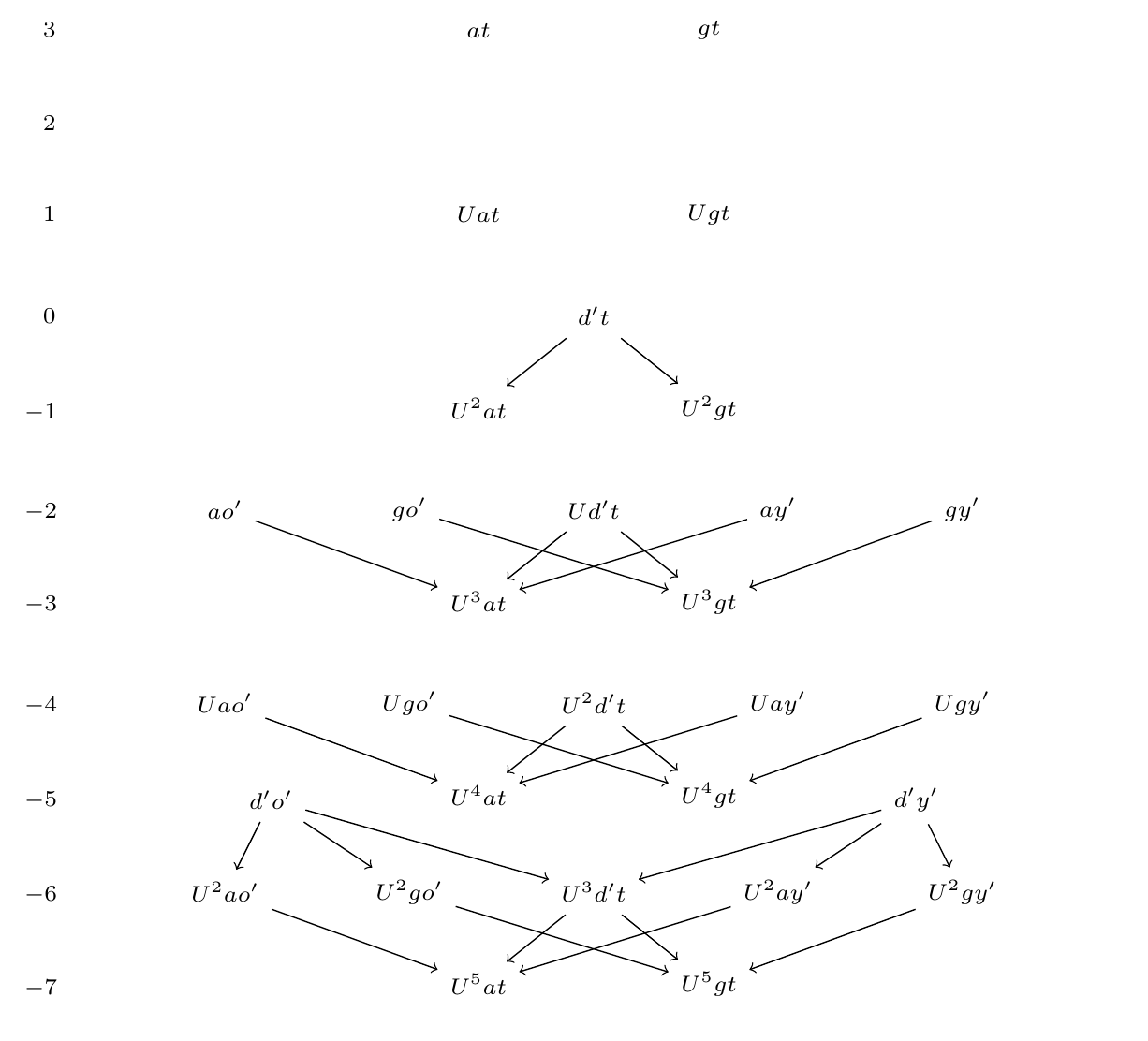}
\caption{A copy of $\CFm(S^3_{-3}(T_{2,7})\# (-L(3,1)) \# S^3_{5}(-T_{2,11}))$. The involution is reflection across the centerline of the page.}
\label{fig:271521123lens}
\end{figure}

\begin{proof}[Proof of Proposition~\ref{prop:undetermined}] Recall that 
\begin{align*}
&Y_1 = S^3_{-3}(T_{2,7})\#S^3_{-3}(T_{2,7}),\\
&Y_2 = S^3_{-3}(T_{2,7})\#(-L(3,1)),\\
&Z=S^3_{5}(-T_{2,11}).
\end{align*}
We begin by computing the correction terms of $Y_2$ and showing they are the same as the correction terms of $Y_1$, which were computed in Lemma~\ref{lem:twocopies}. Recall that $d(-L(3,1))= -d(L(3,1))=-\frac{3-1}{4}=-\frac{1}{2}$. The chain homotopy equivalence class of $\CFm(-L(3,1))$ may be represented by $\Z_2[U][\frac{1}{2}]$ with $\del = 0$; the only possible involution on this complex is $\iota=\Id$. This implies that $\CFm(Y_2) = \CFm(S^3_{-3}(T_{2,7}))[\frac{1}{2}]$ and that this identification respects the involutions on the complexes. In particular,
\begin{align*}
\dl(Y_2) &= \dl(S^3_{-3}(T_{2,7}))-\frac{1}{2} = -\frac{9}{2}-\frac{1}{2} = -5 \\
\du(Y_2) &= \du(S^3_{-3}(T_{2,7})) - \frac{1}{2} = -\frac{1}{2} - \frac{1}{2} = -1. 
\end{align*}
These are exactly the correction terms of $Y_1$ computed in Lemma~\ref{lem:twocopies}. Now we turn our attention to $Y_2 \# Z$. A copy of the chain complex $\CFm(Y_2\#Z) = \CFm(Y_2)\otimes\CFm(Z)[-2] = \CFm(S^3_{-3}(T_{2,7}))\otimes \CFm(Z)[-\frac{3}{2}]$ appears in Figure \ref{fig:271521123lens}; it is the tensor product of the complexes in Figures \ref{fig:2715complexreduced} and \ref{fig:21123complexreduced} with a grading shift for the lens space factor. The involution on $\CFm(S^3_{-3}(T_{2,7}))\otimes \CFm(Z)[-\frac{3}{2}]$ may be taken to be the tensor product of the involutions on the individual complexes, or reflection across the centerline of the page. Notice that $d(Y_2\#Z)= d(Y_2)+d(Z) = -1+1 =0$. We now use Lemma~\ref{lemma:reformulation} to compute the correction terms of this complex. First, note that $d(Y_2\#Z)= d(Y_2)+d(Z) = -1+1 =0$.

Let us begin with $\dl(Y_2\#Z)$. Observe that the homogeneous elements $v$ in $\CFm(Y_2\# Z)$ with the property that $[U^nv]\neq 0$ are exactly the four terms $ao'+ay'$, $go'+gy'$, $ao'+d't+gy'$, $go'+d't+ay'$ and their $U$-powers. In particular, the element $v = ao' + Ud't + gy'$ has the property that $[U^nv]\neq 0$ for all $n\geq 0$ and $(1+\iota)v=0$. Furthermore, $\gr(v)+2 = -2+2 = 0 = d(Y_2\#Z) = 0$. Since this is the largest possible value for $\dl(Y_2\#Z)$, we conclude that $\dl(Y_2\#Z) = 0 = d(Y_2\#Z)$.

Now let us compute $\du(Y_2\#Z)$. We see that if $y=d't$, $x=U^2at$, and $z= ao'$, then $\del y = U^2at + U^2gt$, $\del z = U^3at = Ux$, and furthermore $[U^n(Uy + (1+\iota)z)] = [U^n(Ud't + ao' + gy')]\neq 0$ for all $n\geq 0$. Therefore $\du(Y_2\#Z)$ is at least $\gr(y)+2 = 0 +2 = 2$. Suppose there is some other set of elements $\wt{x},\wt{y},\wt{z}$ satisfying the same relations such that either $\gr(\wt{y})>\gr(y)=0$ or $\gr(\wt{x})>\gr(x)=-1$. Then since $\wt{y}$ must lie in an even grading, and there are no nontrivial even-graded elements in degrees higher than $0$ in the complex, $\wt{y}=0$. This implies that there is some $\wt{z}$ such that $[U^n(1+\iota)\wt{z}]\neq 0$ for all $n\geq 0$. However, in the preceding paragraph we gave a complete list of elements $v$ in the chain complex such that $[U^n v] \neq 0$ for all $n\geq 0$, and no such element is in the image of $(1+\iota)$. Therefore our original $x,y,z$ have the maximal possible gradings for a triple satisfying the conditions of Lemma~\ref{lemma:reformulation}, and $\du(Y_2\#Z)=2$.

Now recall that the correction terms of $Y_1\#Z$ were computed in Proposition~\ref{prop:alldifferent}. We see that while $Y_1$ and $Y_2$ agree in all their correction terms, we have
\[
\dl(Y_1\#Z) = -2 \neq 0= \dl(Y_2 \# Z).
\]
We conclude that the involutive correction terms of a connected sum are not determined by those of the summands.
\end{proof}

\section{Comparison to Seiberg-Witten theory}
\label{sec:SW}
As explained in \cite[Section 3]{HMinvolutive}, involutive Heegaard Floer homology corresponds to $\Z_4$-equivariant Seiberg-Witten Floer homology. Specifically, for a rational homology sphere $Y$ equipped with a spin structure $\s$, the variant $\HFIp(Y, \s)$ is conjectured to be isomorphic to the $\Z_4$-equivariant Borel homology of the Seiberg-Witten Floer spectrum $\swf(Y, \s)$ constructed by the second author in \cite{Spectrum}. With regard to the minus version $\HFIm(Y, \s)$, up to a grading shift, this should correspond to the $\Z_4$-equivariant coBorel homology of $\swf(Y, \s)$. In other words, if we look at the cochain complex dual to $\CFIm(Y, \s)$, its cohomology should be the Borel cohomology $\tH^*_{\Z_4}(\swf(Y); \Z_2)$.

For connected sums, we expect that
$$\swf(Y_1 \# Y_2, \s_1 \# \s_2) \cong \swf(Y_1, \s_1) \wedge \swf(Y_2, \s_2).$$
Let $X_1 = \swf(Y_1, \s_1)$ and $X_2 = \swf(Y_2, \s_2)$. Although $X_1$ and $X_2$ are suspension spectra, by taking suitable suspensions we can assume without loss of generality that they are pointed spaces (with a $\Z_4$-action). Given a pointed $\Z_4$-space $X$, we denote by $\tC^*_{\Z_4}(X)$ the reduced singular cochain complex of $X \wedge_{\Z_4} (E\Z_4)_+$, with $\Z_2$ coefficients. We have
\begin{equation}
\label{eq:dgm}
 \tC^*_{\Z_4}(X_1 \wedge X_2) \cong \tC^*_{\Z_4}(X_1) \otimes_{\tC^*_{\Z_4}(S^0)}  \tC^*_{\Z_4}(X_2).
 \end{equation}
The differential graded algebra $\tC^*_{\Z_4}(S^0) \cong C^*(B\Z_4)$ is quasi-isomorphic to 
$$A:= \Z_2[Q, W]/(dW=Q^2),$$ with $W$ and $Q$ of degree $1$, and $W^2$ corresponding to the usual variable $U$. Note that the cohomology of $A$ is our ring $\Ring = \Z_2[Q, U]/(Q^2)$.

In view of \eqref{eq:dgm}, we may expect a rephrasing of our connected sum formula (Theorem~\ref{thm:ConnSum}) in terms of taking the tensor product of differential graded modules over $A$. However, in the context of involutive Heegaard Floer homology, it is not clear to us how to construct such modules starting from the data $(\CFm(Y, \s),\inv)$.

One could also look for a reformulation of the connected sum formula in terms of $\Ainf$-structures, along the lines of \cite{LinConnSums}. Indeed, every dga (such as $A$) can be viewed as an $\Ainf$-algebra, with trivial higher operations $\mu_n$ for $n\geq 3$. By Kadeishvili's theorem \cite{Kadeishvili}, every $\Ainf$-algebra is $\Ainf$-equivalent to its homology, equipped with suitable higher operations. In our case, one can check that $A$ is $\Ainf$-equivalent to the $\Ainf$-algebra $\RRing$, which is the ring $\Ring = H^*(A)$ with the only nontrivial higher operations being
$$ \mu_4(QU^a, QU^b, QU^c, QU^d) = U^{a+b+c+d+1}.$$

Kadeishvili's theorem also applies to dg modules, saying that they are $\Ainf$-equivalent to their homology. Therefore, we can rephrase \eqref{eq:dgm} as
\begin{equation}
\label{eq:ai}
  \tH^*_{\Z_4}(X_1 \wedge X_2) \cong \tH^*_{\Z_4}(X_1)\ \tilde\otimes_{\RRing} \ \tH^*_{\Z_4}(X_2),
  \end{equation}
where the Borel cohomology groups are equipped with certain higher operations, and thus viewed as $\Ainf$-modules over $\RRing$. The symbol $\tilde\otimes$ denotes the $\Ainf$ tensor product of two $\Ainf$-modules. 

There is also a version of \eqref{eq:ai} for Borel homology instead of cohomology:
\begin{equation}
\label{eq:ai2}
  \tH_*^{\Z_4}(X_1 \wedge X_2) \cong \tH_*^{\Z_4}(X_1)\ \tilde\otimes_{\RRing} \ \tH_*^{\Z_4}(X_2),
  \end{equation}

One consequence of \eqref{eq:ai2} is the existence of an Eilenberg-Moore spectral sequence
$$ E^2_{*,*} \cong \Tor^{\Ring}_{*,*}\Bigl(  \tH_*^{\Z_4}(X_1),  \tH_*^{\Z_4}(X_2) \Bigr) \ \Rightarrow \  \tH_*^{\Z_4}(X_1 \wedge X_2).$$
We refer to \cite[Section 8.2]{McCleary} for an explanation of how the Massey products determine the higher differentials of the Eilenberg-Moore spectral sequence, and to \cite[Lemma 1]{LinConnSums} for the relevant result in terms of $\Ainf$ structures. In the context of $\pin$ monopole Floer homology, an Eilenberg-Moore spectral sequence was established by Lin in  \cite[Corollary 1]{LinConnSums}.

Formula \eqref{eq:ai2} suggests that a similar isomorphism should hold for involutive Heegaard Floer homology:
\begin{equation}
\label{eq:conj1}
\HFIm(Y_1 \# Y_2, \s_1 \# \s_2) \cong \HFIm(Y_1, \s_1)\ \tilde\otimes_{\RRing} \ \HFIm(Y_2, \s_2).
\end{equation}
Of course, for this one would need to endow the involutive Heegaard Floer groups with higher operations, turning them into modules over the $\Ainf$-algebra $A$. We leave this as an open problem. It is possible that constructing such operations involves higher chain homotopies, and thus runs into higher naturality issues that are not accessible with current Heegaard Floer technology.

If \eqref{eq:conj1} were true, one would also obtain a spectral sequence of Eilenberg-Moore type:
\begin{equation}
\label{eq:EM}
 E^2_{*,*} \cong \Tor^{\Ring}_{*,*}\Bigl( \HFIm(Y_1, \s_1),  \HFIm(Y_2, \s_2) \Bigr) \ \Rightarrow \  \HFIm(Y_1 \# Y_2, \s_1 \# \s_2).
 \end{equation}

Note that, by purely algebraic considerations, there is a K\"{u}nneth spectral sequence (cf. \cite[Theorem 2.20]{McCleary}):
$$ E^2_{*,*} \cong \Tor^{\Ring}_{*,*}\Bigl( \HFIm(Y_1, \s_1),  \HFIm(Y_2, \s_2) \Bigr) \ \Rightarrow \  
H_*(CFI^-(Y_1,\s_1)\otimes_{\mathcal{R}} CFI^-(Y_2,\s_2)).$$

This has the same $E^2$ page as the conjectured spectral sequence \eqref{eq:EM}, but the modules they converge to are potentially different. According to Theorem~\ref{thm:ConnSum}, $\HFIm(Y_1 \# Y_2, \s_1 \# \s_2)$ is the homology of the mapping cone
\[\Cone \bigl (CF^-(Y_1)\otimes CF^-(Y_2)\xrightarrow{Q(1+\iota_1\otimes \iota_2)} Q\cdot (CF^-(Y_1)\otimes CF^-(Y_2)) \bigr).\]

On the other hand, $CFI^-(Y_1,\s_1)\otimes_{\mathcal{R}} CFI^-(Y_2,\s_2)$ is easily seen to be isomorphic to the mapping cone
\[\Cone\bigl (CF^-(Y_1)\otimes CF^-(Y_2)\xrightarrow{Q(\iota_1\otimes 1+1\otimes \iota_2)} Q\cdot (CF^-(Y_1)\otimes CF^-(Y_2)) \bigr).\] 

One can define a chain map between these two cones, as the identity on the first copy of $CF^-(Y_1)\otimes CF^-(Y_2)$ and $1\otimes \iota_2$ on the second copy, together with a chain homotopy between $\inv_1 \otimes 1$ and $\inv_1 \otimes \inv_2^2$ along the diagonal. The five lemma shows that the resulting map is a quasi-isomorphism over $\Z_2[U]$. However, the map is not $Q$-equivariant, so it does not give an $\mathcal{R}$-module quasi-isomorphism. Still, in many simple examples, the mapping cones turn out to have the same homology, as $\Ring$-modules. We leave it as another open problem to determine whether this is true in general, or to find a counterexample.

\bibliographystyle{custom} 
\bibliography{biblio} 

\end{document}